\documentclass{tac}

\usepackage{enumerate,xspace}
\usepackage{amsmath,amssymb,wasysym}
\usepackage[all]{xy}
\usepackage{proof}
\usepackage[svgnames]{xcolor}
\usepackage{tikz}
\usepackage[numbers,sort&compress]{natbib} %sorting

\usepackage{stmaryrd} %need for special interpretation bracket
\usepackage{mathtools}
\usepackage{latexsym}
\usepackage{dsfont}
\usepackage{multicol}
\usepackage{cmll}
\usepackage{multirow}
\usepackage{longtable}
\usepackage{graphicx} % provides "\rotatebox" and "\reflectbox" macros

\usepackage{lscape}
\usepackage{array}

\newcommand{\X}{\mathbb{X}}
\newcommand{\T}{\mathbb{T}}

\usepackage{hyperref} %hyperlink package
\hypersetup{
    colorlinks,
    citecolor=red,
    filecolor=red,
    linkcolor=blue,
    urlcolor=red
}

\newtheorem{observation}{Remark}[section]
\newtheorem{lemma}[observation]{Lemma}  %%share counter with remark
\newtheorem{theorem}[observation]{Theorem}
\newtheorem{definition}[observation]{Definition}
\newtheorem{example}[observation]{Example}
\newtheorem{remark}[observation]{Remark}

\newtheorem{proposition}[observation]{Proposition} 
\newtheorem{corollary}[observation]{Corollary}

\makeatletter

\newdimen\w@dth

\def\setw@dth#1#2{\setbox\z@\hbox{\scriptsize $#1$}\w@dth=\wd\z@
\setbox\@ne\hbox{\scriptsize $#2$}\ifnum\w@dth<\wd\@ne \w@dth=\wd\@ne \fi
\advance\w@dth by 1.2em}

\def\t@^#1_#2{\allowbreak\def\n@one{#1}\def\n@two{#2}\mathrel
{\setw@dth{#1}{#2}
\mathop{\hbox to \w@dth{\rightarrowfill}}\limits
\ifx\n@one\empty\else ^{\box\z@}\fi
\ifx\n@two\empty\else _{\box\@ne}\fi}}
\def\t@@^#1{\@ifnextchar_ {\t@^{#1}}{\t@^{#1}_{}}}

\def\t@left^#1_#2{\def\n@one{#1}\def\n@two{#2}\mathrel{\setw@dth{#1}{#2}
\mathop{\hbox to \w@dth{\leftarrowfill}}\limits
\ifx\n@one\empty\else ^{\box\z@}\fi
\ifx\n@two\empty\else _{\box\@ne}\fi}}
\def\t@@left^#1{\@ifnextchar_ {\t@left^{#1}}{\t@left^{#1}_{}}}

\def\two@^#1_#2{\def\n@one{#1}\def\n@two{#2}\mathrel{\setw@dth{#1}{#2}
\mathop{\vcenter{\hbox to \w@dth{\rightarrowfill}\kern-1.7ex
                 \hbox to \w@dth{\rightarrowfill}}%
       }\limits
\ifx\n@one\empty\else ^{\box\z@}\fi
\ifx\n@two\empty\else _{\box\@ne}\fi}}
\def\tw@@^#1{\@ifnextchar_ {\two@^{#1}}{\two@^{#1}_{}}}

\def\tofr@^#1_#2{\def\n@one{#1}\def\n@two{#2}\mathrel{\setw@dth{#1}{#2}
\mathop{\vcenter{\hbox to \w@dth{\rightarrowfill}\kern-1.7ex
                 \hbox to \w@dth{\leftarrowfill}}%
       }\limits
\ifx\n@one\empty\else ^{\box\z@}\fi
\ifx\n@two\empty\else _{\box\@ne}\fi}}
\def\t@fr@^#1{\@ifnextchar_ {\tofr@^{#1}}{\tofr@^{#1}_{}}}

\newdimen\W@dth
\def\setW@dth#1#2{\setbox\z@\hbox{$#1$}\W@dth=\wd\z@
\setbox\@ne\hbox{$#2$}\ifnum\W@dth<\wd\@ne \W@dth=\wd\@ne \fi
\advance\W@dth by 1.2em}

\def\T@^#1_#2{\allowbreak\def\N@one{#1}\def\N@two{#2}\mathrel
{\setW@dth{#1}{#2}
\mathop{\hbox to \W@dth{\rightarrowfill}}\limits
\ifx\N@one\empty\else ^{\box\z@}\fi
\ifx\N@two\empty\else _{\box\@ne}\fi}}
\def\T@@^#1{\@ifnextchar_ {\T@^{#1}}{\T@^{#1}_{}}}

\def\T@left^#1_#2{\def\N@one{#1}\def\N@two{#2}\mathrel{\setW@dth{#1}{#2}
\mathop{\hbox to \W@dth{\leftarrowfill}}\limits
\ifx\N@one\empty\else ^{\box\z@}\fi
\ifx\N@two\empty\else _{\box\@ne}\fi}}
\def\T@@left^#1{\@ifnextchar_ {\T@left^{#1}}{\T@left^{#1}_{}}}

\def\Tofr@^#1_#2{\def\N@one{#1}\def\N@two{#2}\mathrel{\setW@dth{#1}{#2}
\mathop{\vcenter{\hbox to \W@dth{\rightarrowfill}\kern-1.7ex
                 \hbox to \W@dth{\leftarrowfill}}%
       }\limits
\ifx\N@one\empty\else ^{\box\z@}\fi
\ifx\N@two\empty\else _{\box\@ne}\fi}}
\def\T@fr@^#1{\@ifnextchar_ {\Tofr@^{#1}}{\Tofr@^{#1}_{}}}

\def\Two@^#1_#2{\def\N@one{#1}\def\N@two{#2}\mathrel{\setW@dth{#1}{#2}
\mathop{\vcenter{\hbox to \W@dth{\rightarrowfill}\kern-1.7ex
                 \hbox to \W@dth{\rightarrowfill}}%
       }\limits
\ifx\N@one\empty\else ^{\box\z@}\fi
\ifx\N@two\empty\else _{\box\@ne}\fi}}
\def\Tw@@^#1{\@ifnextchar_ {\Two@^{#1}}{\Two@^{#1}_{}}}

\def\to{\@ifnextchar^ {\t@@}{\t@@^{}}}
\def\from{\@ifnextchar^ {\t@@left}{\t@@left^{}}}
\def\tofro{\@ifnextchar^ {\t@fr@}{\t@fr@^{}}}
\def\To{\@ifnextchar^ {\T@@}{\T@@^{}}}
\def\From{\@ifnextchar^ {\T@@left}{\T@@left^{}}}
\def\Two{\@ifnextchar^ {\Tw@@}{\Tw@@^{}}}
\def\Tofro{\@ifnextchar^ {\T@fr@}{\T@fr@^{}}}

\makeatother

\title[Diff. bundles in Comm. Algebra and Alg. Geometry]{Differential Bundles in Commutative Algebra \\ and Algebraic Geometry}
\author{G.S.H. Cruttwell and Jean-Simon Pacaud Lemay}
\thanks{The authors thank Geoff Vooys for his very useful algebraic geometry insights and references.  G.S.H. Cruttwell was partially supported by an NSERC Discovery grant, and Jean-Simon Pacaud Lemay was financially supported by a NSERC Postdoctoral Fellowship - Award \#: 456414649, a JSPS Postdoctoral Fellowship - Award \#: P21746, and an ARC DECRA, Award \#: DE230100303.}
\keywords{tangent categories, differential bundles, modules}
\amsclass{18F40, 13A99, 14A99}

\address{Department of Mathematics and Computer Science\\
 Mount Allison University, Sackville, Canada \\
 School of Mathematical and Physical Sciences \\
 Macquarie University, Sydney, Australia
}

\eaddress{gcruttwell@mta.ca \CR js.lemay@mq.edu.au}

\begin{document}
\allowdisplaybreaks

\maketitle

\begin{abstract} In this paper, we explain how the abstract notion of a differential bundle in a tangent category provides a new way of thinking about the category of modules over a commutative ring and its opposite category. MacAdam previously showed that differential bundles in the tangent category of smooth manifolds are precisely smooth vector bundles. Here we provide characterizations of differential bundles in the tangent categories of commutative rings and (affine) schemes. For commutative rings, the category of differential bundles over a commutative ring is equivalent to the category of modules over that ring. For affine schemes, the category of differential bundles over the Spec of a commutative ring is equivalent to the opposite category of modules over said ring. Finally, for schemes, the category of differential bundles over a scheme is equivalent to the opposite category of quasi-coherent sheaves of modules over that scheme.  
\end{abstract}

\tableofcontents

%%%%%%%%%%%%%%%%%%%%%%%%%%%%%%%%%%%%%%%%%%%%%%%%%%%%%%%%%%%%%%%%%%%%%

\section{Introduction}

What exactly is the relationship between differential geometry and algebraic geometry?  While there are many differences between these two subjects, one common thread is the use of ``differential'' methods. Indeed, discussions of tangent vectors, tangent spaces, and differentials are important in both subjects. A natural question to then ask is: can we precisely relate and contrast how differential geometry and algebraic geometry use these ideas? This paper gives one way to approach this question via the theory of \emph{tangent categories}, in particular through investigating \emph{differential bundles}.    

Tangent categories were first introduced by Rosický in \cite{rosicky1984abstract}, and later generalized and further developed by Cockett and Cruttwell in \cite{cockett2014differential}. A tangent category (Definition \ref{def:tancat}) is a category equipped with an endofunctor $\mathsf{T}$ which for every object $A$ associates an object $\mathsf{T}(A)$ that ``behaves like a tangent bundle'' for $A$.  More precisely, this behaviour is captured through various natural transformations related to the endofunctor $\mathsf{T}$, which encode basic properties such as linearity of the derivative and symmetry of mixed partial derivatives. The canonical example of a tangent category is the category of smooth manifolds, where the endofunctor is the tangent bundle functor.  However, there are also many other interesting examples of tangent categories. In fact, almost any category which has some form of ``differentiation'' for its morphisms can be given the structure of a tangent category. Examples of tangent categories include: 
\begin{itemize}
    \item Most generalizations of smooth manifolds form tangent categories. The category of ``convenient'' manifolds \cite{convenient}, the category of $C^\infty$ rings \cite{lawvereCInfinity}, and any model of synthetic differential geometry (SDG) \cite{sdg99} are all tangent categories.
    \item Any Cartesian differential category \cite{blute2009cartesian}, which formalizes differential calculus over Euclidean spaces, gives a tangent category. In particular, there are many examples of Cartesian differential categories from computer science, such as models of the differential lambda-calculus \cite{EHRHARD20031}.
    \item The category of commutative rings, and more generally categories of commutative algebras, are tangent categories, whose tangent bundle is given by dual numbers (Section \ref{sec:comringtan}).  
    \item The categories of affine schemes and schemes are tangent categories, where the tangent bundle is induced by the K\"{a}hler differentials (Section \ref{sec:comringoptan}).
    \item One can extend the definition of a tangent category to ``tangent infinity'' categories, which can be used to model ideas from Goodwillie functor calculus \cite{tangentInfinity}.
\end{itemize}

The theory of tangent categories is now well-established with a rich literature. In particular, in arbitrary tangent categories one can discuss about vector fields \cite{jacobi}, vector bundles \cite{cockett2018differential}, connections \cite{connections}, solutions to differential equations \cite{cockett2021differential}, and differential forms and de Rham cohomology \cite{deRham}. 

The main focus of this paper is differential bundles (Definition \ref{def:differentialbundle}) in the tangent categories found in commutative algebra and algebraic geometry. Differential bundles are a central structure in tangent categories, as they generalize smooth vector bundles in the category of smooth manifolds \cite{macadam2021vector}.  However, intriguingly, they are defined quite differently than vector bundles. The definition of a differential bundle contains no mention of either vector spaces, a base field, or local triviality. Instead, their central structure is the existence of a \emph{vertical lift}, which is a map from the total space to its tangent bundle satisfying a key universal property. That such a structure, when looked at in the category of smooth manifolds, gives exactly smooth vector bundles \cite{macadam2021vector}, is already interesting enough, as structures like the vector spaces in each fibre, and the local triviality, all come ``for free'' from the universality of the vertical lift. 

But what are differential bundles in the tangent categories of commutative rings, affine schemes, and schemes?  It is not immediately obvious what they should be. The main objective of this paper is to answer this question, and in so doing provide results which open up the possibility for many future investigations in this area. In summary, the main results of this paper are: 
\begin{itemize}
    \item Proposition \ref{prop:mod-diff-ring} and Theorem \ref{thm:mod-diff-ring-1}: In the tangent category of commutative rings, differential bundles over a commutative ring $R$ correspond to modules over $R$, and the category of differential bundles over $R$ is equivalent to the category of modules over $R$. 
    \item Proposition \ref{prop:mod-diff-aff} and Theorem \ref{thm:mod-diff-ring-aff-1}: In the tangent category of affine schemes (or equivalently the opposite category of commutative rings), differential bundles over a commutative ring $R$ correspond to modules over $R$, and the category of differential bundles over $R$ is equivalent to the \emph{opposite} category of modules over $R$. 
    \item Theorem \ref{thm:scheme_equivalence_1}: In the tangent category of schemes, differential bundles over a scheme $A$ correspond to quasicoherent sheaves of modules over $A$, and the category of differential bundles over $A$ is equivalent to the \emph{opposite} of the category of quasicoherent sheaves of modules over $A$.
\end{itemize}

These results are fascinating for several reasons.  For one, they show how diverse differential bundles can be. In the canonical tangent category example of smooth manifolds, differential bundles are exactly smooth vector bundles, whose definition includes the strict condition of local triviality.  However, in these ``algebraic'' tangent categories, differential bundles still give categories of central importance, categories of modules, but in which the objects have no sort of local triviality condition. Independently, these results are also interesting as they give a new \emph{characterization} of these categories. In particular, differential bundles provide a novel characterization of the opposite of the category of (quasicoherent sheaves of) modules. To the best of the authors' knowledge, there is no known previous characterization of the opposite of the category of modules for an arbitrary commutative ring (though there are some results in special cases, like characterizations of the opposite category of Abelian groups).  These results are thus interesting in and of themselves.  

Even more promising than the results themselves is what future results and ideas they can lead to. As described above, in any tangent category one can define and prove results about connections on such bundles; again, when applied to the tangent category of smooth manifolds, this recreates the usual notion.  But now via tangent categories, we get a notion of connection on modules - what do these look like?  What examples of them are there?  Do they recreate existing notions of connections in algebraic geometry?  We hope to explore these questions in future work (Section \ref{sec:future_work}), and continue to use these ideas to bridge the gap between differential geometry and algebraic geometry.  

\paragraph{Outline:} In Section \ref{sec:tan_cats}, we briefly review the definition of tangent categories and differential bundles.  We also review MacAdam's characterization of differential bundles in the tangent category of smooth manifolds \cite{macadam2021vector}. In Section \ref{sec:diff_bundles_cring}, we give our first major result: a characterization of differential bundles in the tangent category of commutative rings. Section \ref{diff_bundles_cringOp} contains our most important results: characterizations of differential bundles in the tangent categories of (affine) schemes. As mentioned above, as far as we know, these results provide new characterizations of the opposites of categories of (quasicoherent sheaves of) modules. Lastly, in Section \ref{sec:future_work}, we describe future work that we hope to pursue that builds on the ideas presented in this paper. 

\paragraph{Conventions:} 
In an arbitrary category, we denote identity maps as ${1_A: A \to A}$, and we use the classical notation for composition, $g \circ f$, as opposed to diagrammatic order which was used in other papers on tangent categories (such as in \cite{cockett2014differential,cockett2018differential} for example). For pullbacks and products, we use $\pi_j$ for the projections and $\langle -,- \rangle$ for the pairing operation which is induced by the universal property. By a commutative ring, we mean a commutative, unital, and associative ring. For a commutative ring $R$ and $a,b \in R$, we denote the addition by $a+b$, the zero by $0 \in R$, the negation by $-a$,  the multiplication by $ab$, and the unit by $1 \in R$. By a an $R$-module we mean a left $R$-module. For an $R$-module $M$, $a \in R$ and $m \in M$, we denote the action by $a \cdot m$ (unless otherwise specified). 

\section{Background}\label{sec:tan_cats}

We use this section to set notation, review and expand some results about differential bundles, and recall an alternative characterization of differential bundles due to MacAdam \cite{macadam2021vector}.  

\subsection{Tangent Categories and Differential Bundles}

We begin by briefly recalling the definitions of tangent categories and differential bundles in tangent categories; for full details see \cite[Definition 2.3]{cockett2014differential} and \cite[Definition 2.2]{connections}. 

\begin{definition}\label{def:tancat} A \textbf{tangent structure} on a category $\X$ is a sextuple $\T = (\mathsf{T}, \mathsf{p}, +, 0, \ell, \mathsf{c})$ consisting of:
\begin{enumerate}[(i)]
\item An endofunctor $\mathsf{T}: \mathbb{X} \to  \mathbb{X}$, called the \textbf{tangent bundle functor};
\item A natural transformation $\mathsf{p}_A: \mathsf{T}(A) \to A$, called the \textbf{projection}, such that for each $n\in \mathbb{N}$, the pullback of $n$ copies of $\mathsf{p}_A$ exists, which we denote as $\mathsf{T}_n(A)$ with $n$ projections $\pi_j: \mathsf{T}_n(A) \to \mathsf{T}(A)$; 
\item A natural transformation\footnote{Note that by the universal property of the pullback, it follows that we can define functors $\mathsf{T}_n: \mathbb{X} \to \mathbb{X}$.} $+_A: \mathsf{T}_2(A) \to \mathsf{T}(A)$, called the \textbf{sum};
\item A natural transformation $0_A: A \to \mathsf{T}(A)$, called the \textbf{zero};
\item A natural transformation $\ell_A: \mathsf{T}(A) \to \mathsf{T}^2(A)$, called the \textbf{vertical lift};
\item A natural transformation $\mathsf{c}_A: \mathsf{T}^2(A) \to \mathsf{T}^2(A)$, called the \textbf{canonical flip};
\end{enumerate}
satisfying various axioms (see \cite[Definition 2.3]{cockett2014differential}). A \textbf{tangent category} is a pair $(\X,\T)$ consisting of a category $\X$ and a tangent structure $\T$ on $\X$. 
\end{definition}

 Most of our tangent structures $\T$ will have inverses for the sum operation $+$, which is encoded by a (necessarily unique) natural transformation $-_A: \mathsf{T}(A) \to \mathsf{T}(A)$. We refer to such a tangent structure $\T$ as a \textbf{Rosický tangent structure}\footnote{Previously called a tangent structure with negatives.}, and say that $(\X,\T)$ is a \textbf{Rosický tangent category}. Lastly, if $\X$ has products and $\mathsf{T}$ preserves them, we say that $(\X,\T)$ is a \textbf{Cartesian (Rosický) tangent category}.

\begin{example} The category of smooth manifolds with $\mathsf{T}$ being the classical tangent bundle functor is a Cartesian Rosický tangent category.
\end{example}

There are many other examples of tangent categories.  Our focus in this paper is on the tangent categories of commutative rings, affine schemes, and schemes; the tangent structures on these categories will be reviewed in Sections \ref{sec:comringtan} and \ref{sec:comringoptan}. There are also many ways to make new tangent categories from existing ones; one of the most fundamental (assuming the existence of certain well-behaved limits) is by slicing. We will use this construction, in particular, to construct tangent categories of algebras from tangent categories of rings.  

\begin{proposition}\label{prop:slice_tan_cat} \cite[Pages 4--5]{rosicky1984abstract}
Suppose that $(\X,\T)$ is a tangent category, and $A$ is an object of $\X$.  Then the slice category $\X/A$ can be given the structure of a tangent category, where the tangent bundle of an object $f: X \to A$, $T_A(f) \to A$, is given by the left side of the pullback
\[ \xymatrix{T_A(f) \ar[r] \ar[d] & TX \ar[d]^{T(f)} \\ A \ar[r]_{0_A} & TA} \]
(assuming such pullbacks exist and are preserved by each $T^n$).  
\end{proposition}

The main focus of this paper is characterizing differential bundles in certain tangent categories.

\begin{definition} \label{def:differentialbundle}  In a tangent category $(\mathbb{X}, \mathbb{T})$, a \textbf{differential bundle} is a quadruple $\mathcal{E} = (\mathsf{q}, \sigma, \mathsf{z}, \lambda)$ consisting of: 
\begin{enumerate}[{\em (i)}]
\item Objects $E$ and $A$ of $\mathbb{X}$, called the \textbf{total} and \textbf{base} objects, respectively;
\item A map ${\mathsf{q}: E \to A}$ of $\mathbb{X}$, called the \textbf{projection}, such that for each $n \in \mathbb{N}$, the pullback of $n$ copies of $\mathsf{q}$ exists, which we denote as $E_n$, with $n$ projection maps $\pi_j: E_n \to E$;
\item A map $\sigma: E_2 \to E$ of $\mathbb{X}$, called the \textbf{sum};
\item A map $\mathsf{z}: A \to E$ of $\mathbb{X}$, called the \textbf{zero};
\item A map $\lambda: E \to \mathsf{T}(E)$ of $\mathbb{X}$, called the \textbf{lift};
\end{enumerate}
satisfying various axioms (for full details, see \cite[Definition 2.3]{cockett2018differential}).  We say the differential bundle \textbf{has negatives} if there is a map $\iota: E \to E$ which is an inverse for the addition operation $\sigma$.  
\end{definition}

Note that a differential bundle can have negatives in any arbitrary tangent category and that the negative is necessarily unique. As was shown in \cite{macadam2021vector}, in a Rosický tangent category, every differential bundle comes equipped with a (necessarily unique) negative (Proposition \ref{prop:Ben2}). Therefore in a Cartesian Rosický tangent category, differential bundles are the same as differential bundles with negatives.

 \begin{example}
 Differential bundles in the tangent category of smooth manifolds over a smooth manifold $M$ are precisely the same as smooth vector bundles over $M$ \cite[Theorem 4.2.7]{macadam2021vector}. Note that this is quite a surprising result, as the definition of differential bundle makes no explicit mention of vector space structure or local triviality!
 \end{example}

 Naturally, differential bundles over a terminal object (if one exists) are also a useful concept:

 \begin{definition} \label{def:diffobj} In a Cartesian tangent category $(\mathbb{X}, \mathbb{T})$, a \textbf{differential object} \cite[Proposition 3.4]{cockett2018differential} is a differential bundle over the terminal object $\ast$. 
\end{definition}

Alternatively, a differential object can be described as an object $A$ equipped with maps ${\hat{\mathsf{p}}: \mathsf{T}(A) \to A}$, $\sigma: A \times A \to A$, and $\mathsf{z}: \ast \to A$ such that $(A, \sigma, \mathsf{z})$ is a commutative monoid, $\mathsf{T}(A) \cong A \times A$ via $\mathsf{p}_A$ and $\hat{\mathsf{p}}$, and the diagrams from \cite[Definition 4.8]{cockett2014differential} commute. In a Cartesian Rosický tangent category, every differential object is automatically an Abelian group. 

\begin{example} \label{ex:smoothdiffobj} In the tangent category of smooth manifolds, the differential objects are precisely the Euclidean spaces.
\end{example}

The following result about differential bundles in a slice tangent category is easy to check, but will be useful for us to help characterize differential bundles in categories of algebras:

\begin{proposition}\label{prop:slice_diff_bundles}
If $(\mathbb{X}, \mathbb{T})$ is a tangent category with an object $A$ which satisfies the requirements of Proposition \ref{prop:slice_tan_cat}, then in the slice tangent category $\mathbb{X}/A$, a differential bundle over $f: X \to A$ is the same as a differential bundle over $X$ in $(\mathbb{X}, \mathbb{T})$.  
\end{proposition}

\subsection{Differential Bundles as Pre-Differential Bundles} \label{sec:prediff}

In this section, we review MacAdam's pre-differential bundles as introduced in \cite{macadam2021vector}. These allow for an alternative characterization of differential bundles, which in particular requires less data and axioms. Indeed, MacAdam cleverly observed that in the definition of a differential bundle, the sum (and negative), and any axioms involving it, can be replaced by a pullback square, called Rosický's universality diagram. From this special pullback, the sum (and negative) for the differential bundle can be constructed from the sum (and negative) of the tangent bundle. MacAdam then introduced pre-differential bundles, which are defined using only the projection, zero, and lift, and showed that differential bundles are precisely pre-differential bundles such that the $n$-fold pullbacks of the projection exist and Rosický's universality diagram holds. This pre-differential bundle approach to differential bundles is quite useful since it requires less data and fewer axioms to check when one wants to construct a differential bundle. This will be particularly useful when we will characterize differential bundles for commutative rings and (affine) schemes.  

The definition of a pre-differential bundle is what remains from the definition of a differential bundle after removing the sum (and negative) and any required pullback. 

\begin{definition} \label{def:prediffbun} In a tangent category $(\mathbb{X}, \mathbb{T})$, a \textbf{pre-differential bundle} \cite[Definition 10]{macadam2021vector} is a triple ${(\mathsf{q}: E \to A, \mathsf{z}: A \to E, \lambda: E \to \mathsf{T}(E))}$ consisting of objects $A$ and $E$ of $\mathbb{X}$, and maps ${\mathsf{q}: E \to A}$, ${\mathsf{z}: A \to E}$, and $\lambda: E \to \mathsf{T}(E)$ of $\mathbb{X}$ such that the following diagrams commute: 
 \begin{equation}\label{prediffbuneq}\begin{gathered} 
  \xymatrixcolsep{2.75pc}\xymatrix{ A \ar[r]^-{\mathsf{z}} \ar@{=}[dr]_-{}  & E \ar[d]^-{\mathsf{q}} & E \ar[r]^-{\lambda} \ar[d]_-{\mathsf{q}}  &  \mathsf{T}(E) \ar[d]^-{\mathsf{p}_E} & A \ar[r]^-{\mathsf{z}} \ar[d]_-{\mathsf{z}} & E \ar[d]^-{0_E} & E \ar[r]^-{\lambda} \ar[d]_-{\lambda}  & \mathsf{T}(E) \ar[d]^-{\mathsf{T}(\lambda)}  \\
& A & A \ar[r]_-{\mathsf{z}}  & E & E \ar[r]_-{\lambda} &  \mathsf{T}(E) &  \mathsf{T}(E) \ar[r]_-{\ell_E} &  \mathsf{T}^2(E) }  \end{gathered}\end{equation} 
If $(\mathsf{q}: E \to A, \mathsf{z}: A \to E, \lambda: E \to \mathsf{T}(E))$ is a pre-differential bundle, we say that it is a \textbf{pre-differential bundle over $A$}. When there is no confusion, pre-differential bundles will be denoted as $(\mathsf{q}: E \to A, \mathsf{z}, \lambda)$, and when the objects are specified simply as $(\mathsf{q}, \mathsf{z}, \lambda)$. 
\end{definition}

By definition of a differential bundle, the projection, zero, and lift of a differential bundle give a pre-differential bundle. On the other hand, a pre-differential bundle is a differential bundle precisely when the pullback of $n$ copies of the projection exists and certain squares are pullbacks \cite[Proposition 6]{macadam2021vector}. Since the main tangent categories of interest in this paper are Cartesian Rosický, we review when a pre-differential bundle is a differential bundle in this setting, where only one square is required to be a pullback \cite[Corollary 3]{macadam2021vector}. This pullback is called Rosický's universality diagram, and using the pullback universal property, we can construct the sum and negative for the differential bundle \cite[Lemma 5]{macadam2021vector}. 

\begin{proposition} \label{prop:prediff} \cite[Corollary 3]{macadam2021vector} Let $(\mathbb{X}, \mathbb{T})$ be a Cartesian Rosický tangent category, and let $(\mathsf{q}: E \to A, \mathsf{z}, \lambda)$ be a pre-differential bundle in $(\mathbb{X}, \mathbb{T})$ such that:
\begin{enumerate}[{\em (i)}]
\item \label{prediff-i} For each $n \in \mathbb{N}$, the pullback of $n$ copies of $\mathsf{q}$ exists, which we denote as $E_n$ with $n$ projection maps $\pi_j: E_n \to E$, for all $1 \leq j \leq n$, so $q \circ \pi_j = q \circ \pi_i$ for all $1 \leq i,j \leq n$, and for all $m \in \mathbb{N}$, $\mathsf{T}^m$ preserves these pullbacks;
\item \label{prediff-ii} The following commuting square is a pullback, called \textbf{Rosický's universality diagram}: 
 \begin{equation}\label{prediffbunpb}\begin{gathered} 
  \xymatrixcolsep{5pc}\xymatrix{ E \ar[r]^-{\lambda}  \ar[d]_-{\mathsf{q}} & \mathsf{T}(E) \ar[d]^-{\langle \mathsf{T}(\mathsf{q}), \mathsf{p}_E \rangle} \\
 A \ar[r]_-{\langle 0_A, \mathsf{z} \rangle}  & \mathsf{T}(A) \times E }  \end{gathered}\end{equation}
and for all $m \in \mathbb{N}$, $\mathsf{T}^m$ preserves this pullback. 
\end{enumerate}
Then define the maps $\sigma: E_2 \to E$ and $\iota: E \to E$ respectively as follows using the universal property of the above pullback: 
 \begin{equation}\label{sigmadef}\begin{gathered} 
  \xymatrixcolsep{4pc}\xymatrix{ E_2   \ar@{-->}[dr]_-{\sigma} \ar[dd]_-{\pi_j} \ar[rr]^-{\langle \lambda \circ \pi_1, \lambda \circ \pi_2 \rangle} & & \mathsf{T}_2(E) \ar[d]^-{+_E} & E   \ar@{-->}[dr]_-{\iota} \ar@/_1.5pc/[ddr]_-{\mathsf{q}} \ar[rr]^-{\lambda} & & \mathsf{T}(E) \ar[d]^-{-_E} \\  
   & E \ar[r]^-{\lambda}  \ar[d]_-{\mathsf{q}} & \mathsf{T}(E) \ar[d]^-{\langle \mathsf{T}(\mathsf{q}), \mathsf{p}_E \rangle} & & E \ar[r]^-{\lambda}  \ar[d]_-{\mathsf{q}} & \mathsf{T}(E) \ar[d]^-{\langle \mathsf{T}(\mathsf{q}), \mathsf{p}_E \rangle} \\
E  \ar[r]_-{\mathsf{q}}  & A \ar[r]_-{\langle 0_A, \mathsf{z} \rangle}  & \mathsf{T}(A) \times E &  & A \ar[r]_-{\langle 0_A, \mathsf{z} \rangle}  & \mathsf{T}(A) \times E }  \end{gathered}\end{equation}
Then $\mathcal{E}= (\mathsf{q}, \sigma, \mathsf{z}, \lambda, \iota)$ is a differential bundle with negatives over $A$. 
\end{proposition}

Conversely, if $\mathcal{E} = (\mathsf{q}: E \to A, \sigma, \mathsf{z}, \lambda)$ is a differential bundle in a Cartesian Rosický tangent category, then $(\mathsf{q}, \mathsf{z}, \lambda)$ is a pre-differential bundle which satisfies (\ref{prediff-i}) and (\ref{prediff-ii}) in Proposition \ref{prop:prediff}. Furthermore, the induced sum as constructed in Proposition \ref{prop:prediff} is precisely the sum $\sigma$ one started with, and so $(\mathsf{q}, \sigma, \mathsf{z}, \lambda, \iota)$ is a differential bundle with negatives. Similarly, if $(\mathsf{q}: E \to A, \sigma, \mathsf{z}, \lambda, \iota)$ is a differential bundle with negatives, then the induced negative as constructed in Proposition \ref{prop:prediff} is precisely the negative $\iota$ one started with. Therefore, in a Cartesian Rosický tangent category, every differential bundle is in fact a differential bundle with negatives. In conclusion, we have the following equivalence: 

\begin{proposition} \label{prop:Ben2} \cite[Proposition 6 \& Corollary 3]{macadam2021vector} In a Cartesian Rosický tangent category $(\mathbb{X}, \mathbb{T})$, the following are in bijective correspondence: 
\begin{enumerate}[{\em (i)}]
\item Differential bundles;
\item Differential bundles with negatives;
\item Pre-differential bundles that satisfy (\ref{prediff-i}) and (\ref{prediff-ii}) in Proposition \ref{prop:prediff}.
\end{enumerate}
\end{proposition}

\subsection{Morphisms and Categories of Differential Bundles} \label{sec:diffmorph}

In this section, we review morphisms between differential bundles. There are two possible kinds: one where the base objects can vary and one where the base object is fixed. The former is used as the maps in the category of all differential bundles of a tangent category, while the latter is used in the category of differential bundles over a specified object. In either case, a differential bundle morphism is asked to preserve the projections and the lifts of the differential bundles. 

\begin{definition} \label{def:differentialbundlemorph} \cite[Definion 2.3]{cockett2018differential} Let $(\mathbb{X}, \mathbb{T})$ be a (Rosický) tangent category.  
\begin{enumerate}[{\em (i)}]
\item Let $\mathcal{E} = (\mathsf{q}: E \to A, \sigma, \mathsf{z}, \lambda)$ and $\mathcal{E}^\prime = (\mathsf{q}^\prime: E^\prime \to A^\prime, \sigma^\prime, \mathsf{z}^\prime, \lambda^\prime)$ be differential bundles in $(\mathbb{X}, \mathbb{T})$. A \textbf{differential bundle morphism}\footnote{These were referred to as \emph{linear} differential bundle morphisms in \cite[Definition 2.3]{cockett2018differential}; however, since these morphisms are the ones of primary importance in this paper, here we simply refer to them as differential bundle morphisms.} $(f,g): \mathcal{E} \to \mathcal{E}^\prime$ is a pair of maps ${f: E \to E^\prime}$ and $g: A \to A^\prime$ such that the following diagram commutes: 
   \begin{equation}\label{diffbunmap}\begin{gathered} 
   \xymatrixcolsep{5pc}\xymatrix{ E \ar[r]^-{f} \ar[d]_-{\mathsf{q}} & E^\prime \ar[d]^-{\mathsf{q}^\prime} &  E \ar[d]_-{\lambda} \ar[r]^-{f}  & E^\prime  \ar[d]^-{\lambda^\prime}    \\
 A \ar[r]_-{g} & A^\prime & \mathsf{T}(E) \ar[r]_-{\mathsf{T}(f)}  & \mathsf{T}(E^\prime)}  \end{gathered}\end{equation} 
Let $\mathsf{DBun}\left[(\mathbb{X}, \mathbb{T}) \right]$ be the category whose objects are differential bundles in $(\mathbb{X}, \mathbb{T})$, maps are differential bundle morphisms between them, identity maps are pairs of identity maps $(1_E, 1_A): \mathcal{E} \to \mathcal{E}$, and composition is defined point-wise, that is, $(f,g) \circ (h,k) = (f \circ h, g \circ k)$. 
\item Let $A$ be an object in $\mathbb{X}$ and $\mathcal{E} = (\mathsf{q}: E \to A, \sigma, \mathsf{z}, \lambda)$ and $\mathcal{E}^\prime = (\mathsf{q}^\prime: E^\prime \to A, \sigma^\prime, \mathsf{z}^\prime, \lambda^\prime)$ be differential bundles over $A$ in $(\mathbb{X}, \mathbb{T})$. A \textbf{differential bundle morphism} ${f: \mathcal{E} \to \mathcal{E}^\prime}$ \textbf{ over $A$} is a map ${f: E \to E^\prime}$ such that $(f,1_A): \mathcal{E} \to \mathcal{E}^\prime$ is a differential bundle morphism. Explicitly, the following diagrams commute: 
   \begin{equation}\label{Adiffbunmap}\begin{gathered} 
   \xymatrixcolsep{5pc}\xymatrix{ E \ar[r]^-{f} \ar[dr]_-{\mathsf{q}} & E^\prime \ar[d]^-{\mathsf{q}^\prime} &  E \ar[d]_-{\lambda} \ar[r]^-{f}  & E^\prime  \ar[d]^-{\lambda^\prime}    \\
& A  & \mathsf{T}(E) \ar[r]_-{\mathsf{T}(f)}  & \mathsf{T}(E^\prime)}  \end{gathered}\end{equation} 
Let $\mathsf{DBun}_\mathbb{T}[A]$ be the category whose objects are differential bundles over $A$ in $(\mathbb{X}, \mathbb{T})$ and whose maps are differential bundle morphisms over $A$ between them, and where identity maps and composition are the same as in $\mathbb{X}$. 
\end{enumerate}
\end{definition}

Differential bundle morphisms automatically preserve the sum and zero, and also negatives if they exist:

\begin{lemma} \label{lem:diffbunmapadd} \cite[Proposition 2.16]{cockett2018differential} Let $(\mathbb{X}, \mathbb{T})$ be a tangent category, and let \\
$\mathcal{E} = (\mathsf{q}: E \to A, \sigma, \mathsf{z}, \lambda)$ and $\mathcal{E}^\prime = (\mathsf{q}^\prime: E^\prime \to A^\prime, \sigma^\prime, \mathsf{z}^\prime, \lambda^\prime )$ be differential bundles in $(\mathbb{X}, \mathbb{T})$, and let $(f,g): \mathcal{E} \to \mathcal{E}^\prime$ be a differential bundle morphism between them. Then $(f,g)$ is an additive bundle morphism \cite[Definition 2.2]{cockett2014differential}, that is, the following diagrams commute: 
   \begin{equation}\label{addbunmap}\begin{gathered} 
   \xymatrixcolsep{5pc}\xymatrix{   E_2  \ar[d]_-{\sigma}  \ar[r]^-{\langle f \circ \pi_1,  f \circ \pi_2 \rangle} & E^\prime_2  \ar[d]^-{\sigma^\prime} & A \ar[d]_-{\mathsf{z}} \ar[r]^-{g}  & A^\prime  \ar[d]^-{\mathsf{z}^\prime}    \\
 E  \ar[r]_-{f}  & E^\prime & E \ar[r]_-{f}  & E^\prime}  \end{gathered}\end{equation} 
Similarly, let $(\mathsf{q}: E \to A, \sigma, \mathsf{z}, \lambda, \iota)$ and $(\mathsf{q}^\prime: E^\prime \to A^\prime, \sigma^\prime, \mathsf{z}^\prime, \lambda^\prime, \iota^\prime)$ be differential bundles with negatives in $(\mathbb{X}, \mathbb{T})$, and let $(f,g): (\mathsf{q}, \sigma, \mathsf{z}, \lambda) \to (\mathsf{q}^\prime, \sigma^\prime, \mathsf{z}^\prime, \lambda^\prime)$ be a differential bundle morphism between the underlying differential bundles. Then $f$ preserves the negative, that is, the following diagram commutes: 
  \begin{equation}\label{groupmap}\begin{gathered} 
 \xymatrixcolsep{5pc}\xymatrix{ E \ar[d]_-{\iota}  \ar[r]^-{f} & E^\prime  \ar[d]^-{\iota^\prime}   \\
 E \ar[r]_-{f}  & E^\prime  }  \end{gathered}\end{equation} 
\end{lemma}

Other properties of differential bundle morphisms can be found in \cite[Section 2.5]{cockett2018differential}. 

Note that since differential bundle morphisms preserve negatives, the notion of a morphism between differential bundles with negatives is the same as a differential bundle morphism. Therefore for a Rosický tangent category, it follows from Proposition \ref{prop:Ben2} that its category of differential bundles is the same as its category of differential bundles with negatives. As such, abusing notation slightly, for a Rosický tangent category $(\mathbb{X}, \mathbb{T})$, we will consider $\mathsf{DBun}\left[(\mathbb{X}, \mathbb{T}) \right]$ and $\mathsf{DBun}_\mathbb{T}[A]$ to be the categories whose objects are differential bundles with negatives and whose maps are differential bundle morphisms. 

For a Cartesian (Rosický) tangent category, its category of differential objects is the category of differential bundles over the terminal object. Note that this is not the same as the Cartesian differential category of differential objects \cite[Theorem 4.11]{cockett2014differential}, since in that category the morphisms are not required to preserve the lift, sum, zero, or negative. 

\begin{definition} Let $(\mathbb{X}, \mathbb{T})$ be a Cartesian (Rosický) tangent category. Define\\
 $\mathsf{DIFF}\left[ (\mathbb{X}, \mathbb{T}) \right]$ to be the category of differential objects and differential bundle morphisms over $\ast$ between them, so $\mathsf{DIFF}\left[ (\mathbb{X}, \mathbb{T}) \right] = \mathsf{DBun}_\mathbb{T}[\ast]$. 
\end{definition}

We conclude this section by discussing differential bundle isomorphisms.  If $(\mathbb{X}, \mathbb{T})$ is a (Rosický) tangent category, then a differential bundle isomorphism is an isomorphism in the category $\mathsf{DBun}\left[(\mathbb{X}, \mathbb{T}) \right]$. Explicitly, this is a differential bundle morphism $(f,g)$ such that there exists a differential bundle morphism of opposite type $(f^{-1}, g^{-1})$ such that $(f,g) \circ (f^{-1}, g^{-1})= (1,1)$ and $(f^{-1}, g^{-1}) \circ (f,g) = (1,1)$. By definition of the composition in $\mathsf{DBun}\left[(\mathbb{X}, \mathbb{T}) \right]$, this is precisely the same as requiring that $f$ and $g$ are isomorphisms in $\mathbb{X}$. Similarly, for an object $A$, a differential bundle isomorphism over $A$ is an isomorphism in the category $\mathsf{DBun}_\mathbb{T}[A]$, which is a differential bundle morphism $f$ which is an isomorphism in $\mathbb{X}$ whose inverse $f^{-1}$ is also a differential bundle morphism. We will now prove the converse, that if the underlying maps of a differential bundle morphism are isomorphisms in the base category, then their inverses are also a differential bundle morphism. This will allow us to reduce the number of things to check when characterizing differential bundles in various tangent categories.  

\begin{lemma}\label{cor:diffbuniso} Let $(\mathbb{X}, \mathbb{T})$ be a tangent category.
\begin{enumerate}[{\em (i)}]
\item Let $\mathcal{E} = (\mathsf{q}: E \to A, \sigma, \mathsf{z}, \lambda)$ and $\mathcal{E}^\prime = (\mathsf{q}^\prime: E^\prime \to A^\prime, \sigma^\prime, \mathsf{z}^\prime, \lambda^\prime )$ be differential bundles in $(\mathbb{X}, \mathbb{T})$, and let $(f,g): \mathcal{E} \to \mathcal{E}^\prime$ be a differential bundle morphism between them. If $f: E \to E^\prime$ and $g: A \to A^\prime$ are isomorphisms in $\mathbb{X}$, then $(f^{-1}, g^{-1}): \mathcal{E}^\prime \to \mathcal{E}$ is a differential bundle morphism. Therefore, $(f,g)$ is a differential bundle isomorphism with inverse $(f^{-1}, g^{-1})$. 
\item Let $A$ an object in $\mathbb{X}$, and let $\mathcal{E} = (\mathsf{q}: E \to A, \sigma, \mathsf{z}, \lambda)$ and $\mathcal{E}^\prime = (\mathsf{q}^\prime: E^\prime \to A, \sigma^\prime, \mathsf{z}^\prime, \lambda^\prime )$ be differential bundles over $A$ in $(\mathbb{X}, \mathbb{T})$, and let $f: \mathcal{E} \to \mathcal{E}^\prime$ be a differential bundle morphism over $A$ between them. If $f: E \to E^\prime$ is an isomorphism in $\mathbb{X}$, then ${f^{-1}: \mathcal{E}^\prime \to \mathcal{E}}$ is a differential bundle morphism over $A$. Therefore $f$ is a differential bundle isomorphism with inverse $f^{-1}$. 
\end{enumerate}
\end{lemma}
\begin{proof} For $(i)$, we compute: 
\begin{align*}
    g^{-1} \circ \mathsf{q}^\prime = g^{-1} \circ \mathsf{q}^\prime \circ f \circ f^{-1} = g^{-1} \circ g \circ \mathsf{q}  \circ f^{-1} =  \mathsf{q} \circ f^{-1}
\end{align*}
The fact that $(f^{-1},g^{-1})$ is then an isomorphism in the category of differential bundles follows from \cite[Lemma 2.18.ii]{cockett2018differential}.  For $(ii)$, if $f$ is a differential bundle morphism over $A$, then $(f,1_A)$ is a differential bundle morphism. The identity is always an isomorphism, so if $f$ is also an isomorphism, it follows from $(i)$ that $(f^{-1}, 1_A)$ is a differential bundle morphism, which implies that $f^{-1}$ is a differential bundle morphism over $A$ as desired. 
\hfill \end{proof}

%%%%%%%%%%%%%%%%%%%%%%%%%%%%%%%%%%%%%%%%%%%%

\section{Differential Bundles for Commutative Rings}\label{sec:diff_bundles_cring}

In this section, we characterize differential bundles (with negatives) in the tangent category of commutative rings and prove that they correspond precisely to modules (Proposition \ref{prop:mod-diff-ring}). To go from a module to a differential bundle, we use a semi-direct product to build a sort of ring of dual numbers from said module (Lemma \ref{lem:mod-diff}). To go from a differential bundle to a module, we take the kernel of the projection (Lemma \ref{lem:diff-mod}). We then obtain that the category of differential bundles is equivalent to the category of modules (Theorem \ref{thm:mod-diff-ring-1} and Theorem \ref{thm:mod-diff-ring-2}). We will also explain how the only differential object is the zero ring (Corollary \ref{cor:diffobj-rin}). 

\subsection{Tangent Category of Commutative Rings}\label{sec:comringtan} Let $\mathsf{CRING}$ be the category whose objects are commutative rings and whose maps are ring morphisms. We begin by reviewing the canonical tangent structure on $\mathsf{CRING}$, whose tangent bundle is given by the ring of dual numbers. This was one of the main examples in Rosický's original paper \cite[Example 2]{rosicky1984abstract}. 

For a commutative ring $R$, its ring of dual numbers is the commutative ring $R[\varepsilon]$ defined as follows: 
\[ R[\varepsilon] = \lbrace a + b \varepsilon \vert~ \forall a,b \in R, \varepsilon^2 = 0 \rbrace \]
 where $a$ and $b \varepsilon$ will be used respectively as shorthand for $a + 0\varepsilon$ and $0 + b \varepsilon$. Then $R[\varepsilon]$ is a commutative ring with multiplication induced by $\varepsilon^2=0$. 
 
We define a Rosický tangent structure $\rotatebox[origin=c]{180}{$\mathsf{T}$} = (\rotatebox[origin=c]{180}{$\mathsf{T}$}, \mathsf{p}, +, 0, \ell, \mathsf{c}, - )$ on $\mathsf{CRING}$ using dual numbers as follows:  
\begin{enumerate}[{\em (i)}]
\item The endofunctor $\rotatebox[origin=c]{180}{$\mathsf{T}$}: \mathsf{CRING} \to \mathsf{CRING}$ maps a commutative ring $R$ to its ring of dual numbers $\rotatebox[origin=c]{180}{$\mathsf{T}$}(R) = R[\varepsilon]$ and a ring morphism $f: R \to S$ is sent to the ring morphism $\rotatebox[origin=c]{180}{$\mathsf{T}$}(f): R[\varepsilon] \to S[\varepsilon]$ defined as follows:
\[ \rotatebox[origin=c]{180}{$\mathsf{T}$}(f)(a+b \varepsilon) = f(a) + f(b) \varepsilon\]
\item The projection $\mathsf{p}_{R}: R[\varepsilon] \to R$ sends $\varepsilon$ to zero, and so is defined as projecting out the first component:
\[\mathsf{p}_{R}(a+b \varepsilon) = a \]
\end{enumerate}

To describe the pullbacks of the projection, first recall that $\mathsf{CRING}$ is a complete category, and therefore all pullbacks exist in $\mathsf{CRING}$. In particular, if $R$ and $R^\prime$ are commutative rings, then for any ring morphism $f: R^\prime \to R$, the general construction of a pullback of $n$ copies of $f$ in $\mathsf{CRING}$ is given by: 
\[R^\prime_n = \lbrace (x_1, \hdots, x_n) \vert ~ x_j \in R' \text{ s.t. } f(x_i) = f(x_j) \text{ for all } 1 \leq i,j \leq n \rbrace\]
and whose ring structure is given coordinate-wise. However for the projection of the ring of dual numbers, one can instead describe these pullbacks in terms of multivariable dual numbers. So for a commutative ring $R$, define $R[\varepsilon_1, \hdots, \varepsilon_n]$ as follows: 
\[R[\varepsilon_1, \hdots, \varepsilon_n] = \lbrace a + b_1 \varepsilon_1 + \hdots + b_n \varepsilon_n \vert~ \forall a,b_i \in R \text{ and } \varepsilon_i \varepsilon_j = 0 \rbrace \]
Then $R[\varepsilon_1, \hdots, \varepsilon_n]$ is a commutative ring whose structure is defined in the obvious way, so in particular the multiplication is induced by $\varepsilon_i \varepsilon_j = 0$. We leave it as an exercise for the reader to check for themselves that $R[\varepsilon_1, \hdots, \varepsilon_n]$ is indeed isomorphic to the pullback of $n$ copies of $\mathsf{p}_R$. So we can continue to describe the tangent structure as follows: 
\begin{enumerate}[{\em (i)}]
\setcounter{enumi}{2}
\item The pullback of $n$ copies of $\mathsf{p}_{R}$ is given by $\rotatebox[origin=c]{180}{$\mathsf{T}$}_n(R) = R[\varepsilon_1, \hdots, \varepsilon_n]$ and where the proejction $\pi_j: R[\varepsilon_1, \hdots, \varepsilon_n] \to R[\varepsilon]$ sends $\varepsilon_j$ to $\varepsilon$ and the other nilpotents to zero, that is, $\pi_j$ projects out the first component and $j$-th nilpotent component:
\[\pi_j(a + b_1 \varepsilon_1 + \hdots + b_n \varepsilon_n) = a + b_j \varepsilon \]
\item The sum $+_{R}: R[\varepsilon_1, \varepsilon_2] \to R[\varepsilon]$ maps both $\varepsilon_1$ and $\varepsilon_2$ to $\varepsilon$, which results in adding the nilpotent parts together:
\[ +_R( a + b\varepsilon_1 + c \varepsilon_2) = a + (b+c) \varepsilon \]
\item The zero $0_{R}:  R \to R[\varepsilon]$ is the injection of $R$ into its ring of dual numbers:
\[0_R(a) = a \]
\item The negative $-_{R}: R[\varepsilon] \to R[\varepsilon]$ maps $\varepsilon$ to $-\varepsilon$, which results in making the nilpotent part negative:
\[-_{R}(a + b \varepsilon) = a - b \varepsilon\]
\end{enumerate}

To describe the vertical lift and the canonical flip, let us first describe $\rotatebox[origin=c]{180}{$\mathsf{T}$}^2(R)$, the ring of dual numbers of the ring of dual numbers in terms of two nilpotent elements $\varepsilon$ and $\varepsilon^\prime$: 
\begin{align*}
 \rotatebox[origin=c]{180}{$\mathsf{T}$}^2(R) = R[\varepsilon][\varepsilon^\prime] = \lbrace a + b \varepsilon + c \varepsilon^\prime + d \varepsilon \varepsilon^\prime \vert~ \forall a,b,c,d \in R \text{ and } \varepsilon^2 = {\varepsilon^\prime}^2 = 0 \rbrace
\end{align*}
where the multiplication is induced by $\varepsilon^2 = {\varepsilon^\prime}^2 = 0$. So we define: 
\begin{enumerate}[{\em (i)}]
\setcounter{enumi}{6}
\item The vertical lift $\ell_{R}: R[\varepsilon] \to R[\varepsilon][\varepsilon^\prime]$ maps $\varepsilon$ to $\varepsilon^\prime$, and so maps the nilpotent component to the outer nilpotent component:  
\[\ell_{R}(a + b \varepsilon) = a + b \varepsilon\varepsilon^\prime\]
\item The canonical flip $\mathsf{c}_{R}: R[\varepsilon][\varepsilon^\prime] \to R[\varepsilon][\varepsilon^\prime]$ swaps $\varepsilon$ and $\varepsilon^\prime$, and so interchanges the middle nilpotent components: 
\[\mathsf{c}_{R}( a + b \varepsilon + c \varepsilon^\prime + d \varepsilon \varepsilon^\prime ) =  a + c \varepsilon + b \varepsilon^\prime + d \varepsilon \varepsilon^\prime\]
\end{enumerate}
So $\rotatebox[origin=c]{180}{$\mathbb{T}$} = (\rotatebox[origin=c]{180}{$\mathsf{T}$}, \mathsf{p}, +, 0, \ell, \mathsf{c}, - )$ is a Rosický tangent structure on $\mathsf{CRING}$. Also, $\mathsf{CRING}$ has finite products where the binary product is given by the Cartesian product of rings $R \times S$, where recall that the ring structure is given pointwise, and where the terminal object is the zero ring $0$. We also have that $(R \times S)[\varepsilon] \cong R[\varepsilon] \times S[\varepsilon]$ and $0[\varepsilon] \cong 0$. So we have that: 

\begin{lemma} $(\mathsf{CRING}, \rotatebox[origin=c]{180}{$\mathbb{T}$})$ is a Cartesian Rosický tangent category.
\end{lemma}

\begin{remark} This tangent category construction nicely generalizes to other settings:
\begin{itemize}
    \item Instead of commutative rings, we could have considered commutative unital semirings (also called rigs, for rings without negatives), which are of particular interest throughout all of computer science. So the category of commutative semirings will be a Cartesian tangent category via dual numbers, but not a Cartesian Rosický tangent category since we dropped negatives. 
    \item For any commutative (semi)ring $R$, the category of commutative $R$-algebras will also be a Cartesian tangent category; this follows from Proposition \ref{prop:slice_tan_cat}.
    \item The Eilenberg-Moore category of a codifferential category (or dually the opposite category of the coEilenberg-Moore category of a differential category) is a Cartesian tangent category \cite[Theorem 22]{cockett_et_al:LIPIcs:2020:11660}, whose tangent structure is indeed a generalization of the above dual numbers tangent structure. In fact, the tangent categories of commutative (semi)rings/algebras are precisely the Eilenberg-Moore categories of the appropriate polynomial models of codifferential categories.  
\end{itemize}
\end{remark}

\subsection{From Differential Bundles to Modules}

We begin by unpacking what a differential bundle with negatives would consist of in $(\mathsf{CRING}, \rotatebox[origin=c]{180}{$\mathbb{T}$})$. First recall that $(\mathsf{CRING}, \rotatebox[origin=c]{180}{$\mathbb{T}$})$ is a Cartesian Rosický tangent category, so by Proposition \ref{prop:Ben2}, differential bundles are the same thing as differential bundles with negatives. Also, as discussed in Section \ref{sec:comringtan}, $\mathsf{CRING}$ admits all pullbacks, so for any ring morphism ${\mathsf{q}: E \to R}$ between commutative rings, the general construction of a pullback of $n$ copies of $\mathsf{q}$ in $\mathsf{CRING}$ is given by: 
\[E_n = \lbrace (x_1, \hdots, x_n) \vert ~ x_j \in E \text{ s.t. } \mathsf{q}(x_i) = \mathsf{q}(x_j) \text{ for all } 1 \leq i,j \leq n \rbrace\]
and whose ring structure is given coordinate-wise. In particular, for the case $n=2$: 
\[E_2 = \lbrace (x,y) \vert ~ x,y \in E \text{ s.t. } \mathsf{q}(x) = \mathsf{q}(y) \rbrace\]
So for a commutative ring $R$, a differential bundle with negatives over $R$ in $(\mathsf{CRING}, \rotatebox[origin=c]{180}{$\mathbb{T}$})$ would consist of a commutative ring $E$ and five ring morphisms: $\mathsf{q}: E \to R$, $\sigma: E_2 \to E$, ${\mathsf{z}: R \to E}$, $\lambda: E \to E[\varepsilon]$, and $\iota: E \to E$. We will expand further upon on many of the equalities and properties these maps satisfy in the proof of Lemma \ref{lem:betaiso} below. To obtain an $R$-module, we take the kernel of the projection $\mathsf{q}$. 

\begin{lemma} \label{lem:diff-mod} Let $R$ be a commutative ring and $\mathcal{E} = (\mathsf{q}: E \to R, \sigma, \mathsf{z}, \lambda, \iota)$ be a differential bundle with negatives over $R$ in $(\mathsf{CRING}, \rotatebox[origin=c]{180}{$\mathbb{T}$})$. Then the kernel of the projection:
\[\mathsf{ker}(\mathsf{q}) = \lbrace x \vert~ \mathsf{q}(x) = 0 \rbrace\] 
is an $R$-module with action $a \cdot x = \mathsf{z}(a) x$. 
\end{lemma}
\begin{proof} Since $\mathsf{q}: E \to R$ is a ring morphism, this induces an $R$-module structure on $E$ with action $a \cdot e = \mathsf{z}(a) e$. Then viewing $R$ as an $R$-module with action given by multiplication $a \cdot b = ab$, this makes the projection $\mathsf{q}: E \to R$ an $R$-linear morphism. Therefore since the kernel of an $R$-linear morphism is always an $R$-module, we indeed have that $\mathsf{ker}(\mathsf{q})$, with the same action as $E$, is an $R$-module. 
\end{proof}

\subsection{From Modules to Differential Bundles}

We now construct a differential bundle from a module. For a commutative ring $R$ and an $R$-module $M$, define $M[\varepsilon]$ as follows: 
\[ M[\varepsilon] = \lbrace a + m \varepsilon \vert~ a \in R, m \in M \text{ and } \varepsilon^2 =0 \rbrace \] 
where $a$ and $m \varepsilon$ will be used respectively as shorthand for $a + 0\varepsilon$ and $0 + m \varepsilon$. Then $M[\varepsilon]$ is a commutative ring with multiplication induced by $\varepsilon^2=0$, that is, $(a + m \varepsilon) (b + n \varepsilon) = ab + (a \cdot n+ b \cdot m) \varepsilon$, and unit $1$. Note that when $M = R$ and the action is given by multiplication $a \cdot b = ab$, then this construction gives us the ring of dual numbers over $R$, or in other words, the tangent bundle $\rotatebox[origin=c]{180}{$\mathsf{T}$}(R) = R[\varepsilon]$.

 We now give $M[\varepsilon]$ the structure of a differential bundle over $R$.  
\begin{enumerate}[{\em (i)}]
\item The projection $\mathsf{q}_M: M[\varepsilon] \to R$ is defined as projecting out the $R$ component:
\[ \mathsf{q}_{M}(a+ m \varepsilon) = a \]
\end{enumerate}

As noted above, there is a general construction of pullbacks in $\mathsf{CRING}$. However for the projection $\mathsf{q}_M: M[\varepsilon] \to R$, we can instead describe these pullbacks in terms of multivariable dual numbers, like for the pullbacks of the tangent bundle. So define $M[\varepsilon_1, \hdots, \varepsilon_n]$ as follows: 
\[M[\varepsilon_1, \hdots, \varepsilon_n] = \lbrace a + m_1 \varepsilon_1 + \hdots + m_n \varepsilon_n \vert~ \forall a \in R, m_j \in M \text{ and } \varepsilon_i \varepsilon_j = 0 \rbrace \]
Then $M[\varepsilon_1, \hdots, \varepsilon_n]$ is a commutative ring whose structure is defined in the obvious way, so in particular the multiplication is induced by $\varepsilon_i \varepsilon_j = 0$. We leave it as an exercise for the reader to check for themselves that $M[\varepsilon_1, \hdots, \varepsilon_n]$ is the pullback of $n$ copies of $\mathsf{q}_M$. We can then describe the rest of the differential bundle structure as follows: 
\begin{enumerate}[{\em (i)}]
\setcounter{enumi}{1}
\item The pullback of $n$ copies of $\mathsf{p}_{R}$ is given by $M[\varepsilon]_n = M[\varepsilon_1, \hdots, \varepsilon_n]$ and where the pullback projection $\pi_j: M[\varepsilon_1, \hdots, \varepsilon_n] \to M[\varepsilon]$ sends $\varepsilon_j$ to $\varepsilon$ and the other nilpotents to zero, that is, $\pi_j$ projects out the $R$ component and $j$-th $M$ component:
\[\pi_j(a + m_1 \varepsilon_1 + \hdots + m_n \varepsilon_n) = a + m_j \varepsilon \]
\item The sum $\sigma: M[\varepsilon_1, \varepsilon_2] \to M[\varepsilon]$ maps both $\varepsilon_1$ and $\varepsilon_2$ to $\varepsilon$, which results in adding the $M$ components together:
\[ \sigma( a + m\varepsilon_1 + n \varepsilon_2) = a + (m+n) \varepsilon \]
\item The zero $\mathsf{z}:  R \to M[\varepsilon]$ is the injection of $R$ into the $R$ component:
\[\mathsf{z}(a) = a \]
\item The negative $\iota: M[\varepsilon] \to M[\varepsilon]$ maps $\varepsilon$ to $-\varepsilon$, which results in making the $M$ component negative:
\[\iota(a + m \varepsilon) = a - m \varepsilon\]
\end{enumerate}
To describe the lift, let us describe $\rotatebox[origin=c]{180}{$\mathsf{T}$}\left(M[\varepsilon]  \right)$, the ring of dual numbers of $M[\varepsilon]$ in terms of two nilpotent elements $\varepsilon$ and $\varepsilon^\prime$: 
\begin{align*}
\rotatebox[origin=c]{180}{$\mathsf{T}$}\left(M[\varepsilon]  \right) = M[\varepsilon][\varepsilon^\prime] = \lbrace a + m \varepsilon + b \varepsilon^\prime + n \varepsilon \varepsilon^\prime \vert~ \forall a,b \in R, m,n \in M \text{ and } \varepsilon^2 = {\varepsilon^\prime}^2 = 0 \rbrace
\end{align*}
where the multiplication is induced by $\varepsilon^2 = {\varepsilon^\prime}^2 = 0$. So we define: 
\begin{enumerate}[{\em (i)}]
\setcounter{enumi}{6}
\item The lift $\lambda: M[\varepsilon] \to M[\varepsilon][\varepsilon^\prime]$ maps $\varepsilon$ to $\varepsilon \varepsilon^\prime$, and so maps the $R$ component of $M[\varepsilon]$ to the first $R$ component of $M[\varepsilon][\varepsilon^\prime]$, and the $M$ component of $M[\varepsilon]$ to the second $M$ component of $M[\varepsilon][\varepsilon^\prime]$:
\[\lambda(a + m \varepsilon) = a + m \varepsilon\varepsilon^\prime\]
\end{enumerate}
We leave it as an exercise for the reader to check that these are all well-defined ring morphisms. 

\begin{lemma}\label{lem:mod-diff} For every commutative ring $R$ and $R$-module $M$, \[ \rotatebox[origin=c]{180}{$\mathsf{M}$}_R(M) := (\mathsf{q}_M, \sigma_M, \mathsf{z}_M, \lambda_M, \iota_M) \] is a differential bundle with negatives over $R$ in $(\mathsf{CRING}, \rotatebox[origin=c]{180}{$\mathbb{T}$})$. 
\end{lemma}
\begin{proof} To show that we have a differential bundle, we will instead show that we have a pre-differential bundle which satisfies (\ref{prediff-i}) and (\ref{prediff-ii}) in Proposition \ref{prop:prediff}. To show that $(\mathsf{q}_M, \mathsf{z}_M, \lambda_M)$ is a pre-differential bundle, we must show that the four equalities from Definition \ref{def:prediffbun} hold, but these all follow from straightforward computation, which we leave to the reader.   
    
    Next, we must show that this pre-differential bundle also satisfies the extra assumptions required to make it a differential bundle. Firstly, as mentioned above, $M[\varepsilon_1, \hdots, \varepsilon_n]$ is indeed the pullback of $n$ copies of the projection $\mathsf{q}_M$. Also, since $\rotatebox[origin=c]{180}{$\mathsf{T}$}$ is a right adjoint\footnote{This is a standard result in commutative algebra; see also the discussion after Lemma \ref{lemma:cringOpisTan}.}, it preserves all limits, and therefore all $\rotatebox[origin=c]{180}{$\mathsf{T}$}^n$ preserve these pullbacks. So $(\mathsf{q}_M, \mathsf{z}_M, \lambda_M)$ satisfies assumption (\ref{prediff-i}) of Proposition \ref{prop:prediff}. Next, we must show that the following square is a pullback: 
     \begin{equation}\label{}\begin{gathered} 
  \xymatrixcolsep{5pc}\xymatrix{ M[\varepsilon] \ar[r]^-{\lambda_M}  \ar[d]_-{\mathsf{q}_M} & M[\varepsilon][\varepsilon^\prime] \ar[d]^-{\langle \rotatebox[origin=c]{180}{$\mathsf{T}$}(\mathsf{q}_M), \mathsf{p}_{M[\varepsilon]} \rangle} \\
 R \ar[r]_-{\langle 0_R, \mathsf{z}_M \rangle}  & R[\varepsilon] \times M[\varepsilon] }  \end{gathered}\end{equation}
So suppose $S$ is a commutative ring, and we have ring morphisms $f: S \to M[\varepsilon][\varepsilon^\prime]$ and $g: S \to R$ such that $\langle \rotatebox[origin=c]{180}{$\mathsf{T}$}(\mathsf{q}_M), \mathsf{p}_{M[\varepsilon]} \rangle \circ f = \langle 0_R, \mathsf{z}_M \rangle \circ g$, that is, for every $x \in S$ the following equality holds: 
\begin{align*}
  \left( \rotatebox[origin=c]{180}{$\mathsf{T}$}(\mathsf{q}_M)(f(x)), \mathsf{p}_{M[\varepsilon]}(f(x)) \right)  = \left( g(x), g(x) \right) 
\end{align*}
    Now $f(x) \in M[\varepsilon][\varepsilon^\prime]$ is of the form: $f(x) = f_1(x) + f_2(x)\varepsilon + f_3(x)\varepsilon^\prime + f_4(x)\varepsilon\varepsilon^\prime$ for some $f_1(x),f_3(x) \in R$ and $f_2(x), f_4(x) \in M$. Then the above equality tells us that:
  \begin{align*}
& (g(x), g(x)) =  \left( \rotatebox[origin=c]{180}{$\mathsf{T}$}(\mathsf{q}_M)(f(x)), \mathsf{p}_{M[\varepsilon]}(f(x)) \right) \\
 &= \left( \mathsf{q}_M ( f_1(x) + f_2(x)\varepsilon) + \mathsf{q}_M( f_3(x) + f_4(x)\varepsilon ) \varepsilon , \mathsf{p}_{M[\varepsilon]}\left( f_1(x) + f_2(x)\varepsilon + f_3(x)\varepsilon^\prime + f_4(x)\varepsilon\varepsilon^\prime \right) \right) \\
 &= \left( f_1(x) + f_3(x) \varepsilon, f_1(x) + f_2(x) \varepsilon \right)
\end{align*}  
So this implies that $g(x) = f_1(x) + f_2(x) \varepsilon$ and $g(x) = f_1(x) + f_3(x) \varepsilon$. However, in both equalities, the left-hand side has no nilpotent component. Therefore, we have that $g(x) = f_1(x)$, $f_2(x) =0$, and $f_3(x) =0$. So $f(x) = g(x) + f_4(x) \varepsilon\varepsilon^\prime$. Then define $\langle f,g \rangle: S \to M[\varepsilon]$ to be $f$ but without $\varepsilon^\prime$, that is, as follows: 
\begin{align}
  \langle f, g \rangle(x) = g(x) + f_4(x) \varepsilon  
\end{align}
That $\langle f,g \rangle$ is a ring morphism essentially follows from the fact that $f$ is a ring morphism. Next we compute that $\langle f,g \rangle$ also satisfies the following:
\begin{align*}
    \lambda_M (\langle f,g \rangle(x) ) = \lambda_M ( g(x) + f_4(x) \varepsilon ) =  g(x) + f_4(x) \varepsilon\varepsilon^\prime = f(x) 
    \end{align*}
    \begin{align*}
   \mathsf{q}_M (\langle f,g \rangle(x) ) = \mathsf{q}_M ( g(x) + f_4(x) \varepsilon ) = g(x) 
\end{align*}
So $\lambda_M \circ \langle f,g \rangle = f$ and $\mathsf{q}_M \circ \langle f,g \rangle = g$ as desired. Lastly, it remains to show that $\langle f, g \rangle$ is the unique such ring morphism. So suppose we have a ring morphism $h: S \to M[\varepsilon]$ such that $\lambda_M \circ h = f$ and $\mathsf{q}_M \circ h = g$. Now $h(x) \in M[\varepsilon]$ is of the form $h(x) = h_1(x) + h_2(x) \varepsilon$ for some $h_1(x) \in R$ and $h_2(x) \in M$. By assumption, we have that:
 \begin{align*}
 g(x) + f_4(x) \varepsilon\varepsilon^\prime = f(x) =    \lambda_M (h(x) ) = \lambda_M ( h_1(x) + h_2(x) \varepsilon ) =  h_1(x) + h_2(x) \varepsilon\varepsilon^\prime 
\end{align*}
So $g(x) + f_4(x) \varepsilon\varepsilon^\prime=  h_1(x) + h_2(x) \varepsilon\varepsilon^\prime$, which implies that $h_1(x) = g(x)$ and $h_2(x) = f_4(x)$. Therefore, $h(x) = g(x) + f_4(x) \varepsilon = \langle f , g \rangle(x)$, and so $\langle f, g \rangle$ is unique. So we conclude that the above square is a pullback diagram. Furthermore, since $\rotatebox[origin=c]{180}{$\mathsf{T}$}$ is a right adjoint, we also have that $\rotatebox[origin=c]{180}{$\mathsf{T}$}^n$ preserves these pullbacks. Thus $(\mathsf{q}_M, \mathsf{z}_M, \lambda_M)$ satisfies assumption (\ref{prediff-ii}) of Proposition \ref{prop:prediff}. Therefore, $(\mathsf{q}_M, \mathsf{z}_M, \lambda_M)$ will induce a differential bundle with negatives. 

It remains to construct the sum and the negative as in Proposition \ref{prop:prediff}, and show that these are the same as the proposed $\sigma$ and $\iota$ above. The sum $\sigma$ will be given by: 
\[  \sigma = \left \langle +_{M[\varepsilon]} \circ \langle \lambda_M \circ \pi_1, \lambda_M \circ \pi_2 \rangle , \mathsf{q}_M \circ \pi_j \right \rangle  \]
We leave it to the reader to check for themselves that the following equalities hold: 
\begin{align*}
     +_{M[\varepsilon]}\left( \langle \lambda_M \circ \pi_1, \lambda_M \circ \pi_2 \rangle (a + m\varepsilon_1 + n\varepsilon_2) \right) = a + (m+n)\varepsilon\varepsilon^\prime %&& \mathsf{q}_M ( \pi_j (a + m\varepsilon_1 + n\varepsilon_2) ) = a
\end{align*}
Therefore by construction, we have that $\sigma( a + m\varepsilon_1 + n \varepsilon_2) = a + (m+n) \varepsilon$ as desired. The negative $\iota$ will be given by: 
\begin{align*}
\iota = \left \langle -_{M[\varepsilon]} \circ \lambda_M,  \mathsf{q}_M \right \rangle
\end{align*}
We then compute that: 
\[ -_{M[\varepsilon]}(\lambda_M(a+m \varepsilon)) = a - m \varepsilon\varepsilon^\prime \]
So by construction, we have that $\iota(a+ m\varepsilon) = a - m \varepsilon$. So we conclude that $\rotatebox[origin=c]{180}{$\mathsf{M}$}_R(M) = (\mathsf{q}_M, \sigma_M, \mathsf{z}_M, \lambda_M, \iota_M)$ is a differential bundle with negatives over $R$. 
\end{proof}

\subsection{Equivalence}

We will now show that the constructions of Lemma \ref{lem:diff-mod} and Lemma \ref{lem:mod-diff} are inverses of each other.

Beginning from the module side of things, let $R$ be a commutative ring and $M$ be an $R$-module. Consider $\mathsf{ker}(\mathsf{q}_M)$, the kernel of the projection of the induced differential bundle $\rotatebox[origin=c]{180}{$\mathsf{M}$}_R(M)$. However, $\mathsf{q}_M(a +m \varepsilon) =0$ implies that $a=0$. So the kernel of the projection consists solely of the $M$ component, that is, $\mathsf{ker}(\mathsf{q}_M) = \lbrace m \varepsilon \vert~ \forall m \in M \rbrace$, which is clearly isomorphic to $M$. Explicitly, $\alpha_M: M \to \mathsf{ker}(\mathsf{q}_M)$ is defined as $\alpha_M(m) = m \varepsilon$, and $\alpha^{-1}_M: \mathsf{ker}(\mathsf{q}_M) \to M$ is defined as $\alpha^{-1}_M(m \varepsilon) = m$. 

\begin{lemma}\label{lem:alphaiso} For every commutative ring $R$ and $R$-module $M$, $\alpha_M: M \to \mathsf{ker}(\mathsf{q}_M)$ is an $R$-linear isomorphism with inverse $\alpha^{-1}_M: \mathsf{ker}(\mathsf{q}_M) \to M$.    
\end{lemma}
\begin{proof} Clearly for every $R$-module $M$, $\alpha_M$ and $\alpha^{-1}_M$ are inverses of each other, that is, $\alpha^{-1}_M(\alpha_M(m)) = m$ and $\alpha_M(\alpha^{-1}_M(m \varepsilon)) = m \varepsilon$. However, we must explain why $\alpha_M$ and $\alpha^{-1}_M$ are also $R$-linear morphisms. Clearly, they both preserve the addition, so we must show that they preserve the action. We start by showing that $\alpha_M$ does, where recall that the action on $\mathsf{ker}(\mathsf{q}_M)$ is defined as $a \cdot (m\varepsilon) = \mathsf{z}_M(a) m\varepsilon$: 
\begin{align*}
\alpha_M( a \cdot m) =  (a \cdot m) \varepsilon = (a + 0\varepsilon) m \varepsilon = \mathsf{z}_M(a) m\varepsilon = a \cdot m\varepsilon =  a \cdot \alpha_M(m)
\end{align*}
So $\alpha_M$ is an $R$-module morphism. Since $\alpha_M$ and $\alpha^{-1}_M$ are inverses as functions, it then follows that $\alpha^{-1}_M$ will also be an $R$-module morphism. So we conclude that $\alpha_M$ and $\alpha^{-1}_M$ are inverse $R$-linear isomorphisms. 
\end{proof}

Let's now start instead from the differential bundle side of the story. So let $\mathcal{E} =({\mathsf{q}: E \to R}, \sigma, \mathsf{z}, \lambda, \iota)$ be a differential bundle with negatives over a commutative ring $R$ in $(\mathsf{CRING}, \rotatebox[origin=c]{180}{$\mathbb{T}$})$. To define differential bundle isomorphisms between $\mathcal{E}$ and $\rotatebox[origin=c]{180}{$\mathsf{M}$}_R\left(\mathsf{ker}(\mathsf{q}) \right)$, we will first need to define a ring isomorphism between $E$ and $\mathsf{ker}(\mathsf{q})[\varepsilon]$. To do so, we must first take a closer look at the lift $\lambda: E \to E[\varepsilon]$. Since the lift is a ring morphism whose codomain is a ring of dual numbers, it is well-known that it must be of the following form:
$\lambda(x) = \mathsf{p}_E(\lambda(x)) + \mathsf{D}_\lambda(x)\varepsilon$, where $\mathsf{D}_\lambda: E \to E$ is a derivation. Now by the second diagram of (\ref{prediffbuneq}), we have that $\mathsf{p}_E \circ \lambda = \mathsf{z} \circ \mathsf{q}$. This implies that the lift is in fact of the form:
\[ \lambda(x) = \mathsf{z}(\mathsf{q}(x)) + \mathsf{D}_\lambda(x)\varepsilon \]
and the product rule for the derivation $\mathsf{D}_\lambda$ is given by $\mathsf{D}_\lambda(xy) = \mathsf{z}(\mathsf{q}(x)) \mathsf{D}_\lambda(y) + \mathsf{z}(\mathsf{q}(y)) \mathsf{D}_\lambda(x)$. Then define the function $\beta_\mathcal{E}: E \to \mathsf{ker}(\mathsf{q})[\varepsilon]$ as follows: 
\begin{align}
   \beta_\mathcal{E}(x) = \mathsf{q}(x) + \mathsf{D}_\lambda(x) \varepsilon
\end{align}
To define its inverse $\beta^{-1}_\mathcal{E}: \mathsf{ker}(\mathsf{q})[\varepsilon] \to E$, we will need to make use of Rosický's universality diagram (\ref{prediffbunpb}). First, define the ring morphism $\zeta_\mathcal{E}: \mathsf{ker}(\mathsf{q})[\varepsilon] \to E[\varepsilon]$ as $\zeta_\mathcal{E}(a + x \varepsilon) = \mathsf{z}(a) + x \varepsilon$. By universality of the pullback, define $\beta^{-1}_\mathcal{E}: \mathsf{ker}(\mathsf{q})[\varepsilon] \to E$ as the unique ring morphism which makes the following diagram commute: 
 \begin{equation}\label{betainvdef}\begin{gathered} 
 \xymatrixcolsep{5pc}\xymatrix{ \mathsf{ker}(\mathsf{q})[\varepsilon] \ar@/_/[ddr]_-{\mathsf{q}_{\mathsf{ker}(\mathsf{q})}} \ar@/^/[drr]^-{\zeta_\mathcal{E}} \ar@{-->}[dr]^-{\beta^{-1}_\mathcal{E}} \\
& E \ar[d]_-{\mathsf{q}}  \ar[r]^-{\lambda} &  E[\varepsilon]\ar[d]^-{\langle \rotatebox[origin=c]{180}{$\mathsf{T}$}(\mathsf{q}), \mathsf{p}_E \rangle} \\
& R \ar[r]_-{\langle 0_R, \mathsf{z} \rangle} & R[\varepsilon] \times E }  \end{gathered}\end{equation} 
so $\beta^{-1}_\mathcal{E} = \langle \mathsf{q}_{\mathsf{ker}(\mathsf{q})}, \zeta_\mathcal{E} \rangle$. We will show below that $\beta^{-1}_\mathcal{E}$ is a differential bundle morphism; it then follows from the compatibility with the lift that $\beta^{-1}_\mathcal{E}(a + x \varepsilon) = \mathsf{z}(a) + x$. 

\begin{lemma} \label{lem:betaiso} For a commutative ring $R$ and a differential bundle with negatives $\mathcal{E} = (\mathsf{q}: E \to R, \sigma, \mathsf{z}, \lambda, \iota)$ over $R$ in $(\mathsf{CRING}, \rotatebox[origin=c]{180}{$\mathbb{T}$})$, ${\beta_\mathcal{E}: \mathcal{E} \to \rotatebox[origin=c]{180}{$\mathsf{M}$}_R\left(\mathsf{ker}(\mathsf{q}) \right)}$ is a differential bundle isomorphism over $R$ with inverse ${\beta^{-1}_\mathcal{E}: \rotatebox[origin=c]{180}{$\mathsf{M}$}_R\left(\mathsf{ker}(\mathsf{q}) \right) \to \mathcal{E}}$. 
\end{lemma}
\begin{proof} We first explain why $\beta_\mathcal{E}$ and $\beta^{-1}_\mathcal{E}$ are well-defined ring morphisms. Starting with $\beta_\mathcal{E}$, we must first explain why $\mathsf{D}_\lambda(x)$ is in the kernel of the projection $\mathsf{q}$. One of the differential bundle axioms tells us that $\rotatebox[origin=c]{180}{$\mathsf{T}$}(\mathsf{q}) \circ \lambda = 0_R \circ \mathsf{q}$. Therefore, for all $x \in E$, we have that $\mathsf{q}(\mathsf{z}(\mathsf{q}(x))) + \mathsf{q}(\mathsf{D}_\lambda(x)) \varepsilon = \mathsf{q}(x)$. Since the right-hand side has no nilpotent component, this implies that $\mathsf{q}(\mathsf{D}_\lambda(x))= 0$. So for all $x \in E$, $\mathsf{D}_\lambda(x) \in  \mathsf{ker}(\mathsf{q})$, and so $\beta_\mathcal{E}$ is well-defined. We leave it to the reader to check for themselves that $\beta_\mathcal{E}$ is indeed a ring morphism. 

Next we explain why $\beta^{-1}_\mathcal{E}$ is a well-defined. To do so, we must show that the outer diagram of (\ref{betainvdef}) commutes. First, note that by the first diagram of (\ref{prediffbuneq}), $\mathsf{q} \circ \mathsf{z} = 1_R$, so $\mathsf{q}( \mathsf{z}( a)) = a$ for all $a \in R$. Then for all $a \in R$ and $x \in \mathsf{ker}(\mathsf{q})$ we compute: 
\begin{align*}&\langle \rotatebox[origin=c]{180}{$\mathsf{T}$}(\mathsf{q}), \mathsf{p}_E \rangle \left( \zeta_\mathcal{E}( a+ x \varepsilon) \right) 
=
\langle \rotatebox[origin=c]{180}{$\mathsf{T}$}(\mathsf{q}), \mathsf{p}_E \rangle \left( \mathsf{z}( a) + x \varepsilon \right)
= \left( \rotatebox[origin=c]{180}{$\mathsf{T}$}(\mathsf{q})( \mathsf{z}( a) + x \varepsilon ) , \mathsf{p}_E( \mathsf{z}( a) + x \varepsilon )  \right) \\
&=
\left( \mathsf{q}( \mathsf{z}( a)) +  \mathsf{q}(x) \varepsilon ) , \mathsf{p}_E( \mathsf{z}( a) + x \varepsilon )  \right)
= 
\left( a, \mathsf{z}(a) \right)
= 
\left( 0_R(a), \mathsf{z}(a) \right) \\
& =
\langle 0_R, \mathsf{z} \rangle (a)
= 
\langle 0_R, \mathsf{z} \rangle \left( \mathsf{q}_{\mathsf{ker}(\mathsf{q})}(a + x \varepsilon) \right) 
\end{align*}
So $\langle \rotatebox[origin=c]{180}{$\mathsf{T}$}(\mathsf{q}), \mathsf{p}_E \rangle \circ \zeta_\mathcal{E} = \langle 0_R, \mathsf{z} \rangle \circ \mathsf{q}_{\mathsf{ker}(\mathsf{q})}$. Therefore, by the universal property of the pullback square, there exists a unique ring morphism $\beta^{-1}_\mathcal{E}: \mathsf{ker}(\mathsf{q})[\varepsilon] \to E$ such that $\lambda \circ \beta^{-1}_\mathcal{E} = \zeta_\mathcal{E}$ and $\mathsf{q} \circ \beta^{-1}_\mathcal{E} = \mathsf{q}_{\mathsf{ker}(\mathsf{q})}$. In particular, these imply that for every $a \in R$ and $x \in \mathsf{ker}(\mathsf{q})$ the following equalities hold: 
\begin{align*}
    \mathsf{q}(\beta^{-1}_\mathcal{E}(a + x \varepsilon) ) = a && \mathsf{D}_\lambda( \beta^{-1}_\mathcal{E}(a + x \varepsilon) ) = x 
\end{align*}

Next we show that $\beta_\mathcal{E}$ and $\beta^{-1}_\mathcal{E}$ are inverses of each other. To show that $\beta_\mathcal{E} \circ \beta^{-1}_\mathcal{E} = 1_{\mathsf{ker}(\mathsf{q})[\varepsilon]}$, we use the above identities:  
\begin{align*}
\beta_\mathcal{E}(\beta^{-1}_\mathcal{E}(a +x\varepsilon) ) = \mathsf{q}(\beta^{-1}_\mathcal{E}(a + x \varepsilon)) + \mathsf{D}_\lambda( \beta^{-1}_\mathcal{E}(a + x \varepsilon) ) \varepsilon = a + x \varepsilon
\end{align*}
To show that $\beta^{-1}_\mathcal{E} \circ \beta_\mathcal{E} = 1_E$, we will first show that $\mathsf{q}\circ \beta^{-1}_\mathcal{E} \circ \beta_\mathcal{E} = \mathsf{q}$ and $\lambda \circ \beta^{-1}_\mathcal{E} \circ \beta_\mathcal{E} = \lambda$: 
\begin{align*}
    \mathsf{q}(\beta^{-1}_\mathcal{E} (\beta_\mathcal{E} (x) ) ) =   \mathsf{q}(\beta_\mathcal{E} (x) ) =  \mathsf{q}((\mathsf{q}(x) +  \mathsf{D}_\lambda(x)\varepsilon ))  = \mathsf{q}(x) 
\end{align*}
\begin{align*}
    \lambda(\beta^{-1}_\mathcal{E} (\beta_\mathcal{E} (x) ) ) =  \zeta_\mathcal{E} (\beta_\mathcal{E} (x) ) =  \zeta_\mathcal{E} (\mathsf{q}(x) +  \mathsf{D}_\lambda(x)\varepsilon ) = \mathsf{z}(\mathsf{q}(x)) +\mathsf{D}_\lambda(x)\varepsilon = \lambda(x)  
\end{align*}
Therefore, by the universal property of the pullback, it follows that $\beta^{-1}_\mathcal{E} \circ \beta_\mathcal{E} = 1_E$. So $\beta_\mathcal{E}$ and $\beta^{-1}_\mathcal{E}$ are inverse ring isomorphisms. 

Lastly, we must show that $\beta_\mathcal{E}$ and $\beta^{-1}_\mathcal{E}$ are also differential bundle morphisms over $R$. To do so, we will need to know a bit more about $\mathsf{D}_\lambda$. The third diagram of (\ref{prediffbuneq}) is $0_E \circ \mathsf{z} = \lambda \circ \mathsf{z}$, which implies that for all $a \in R$, $\mathsf{z}(a) = \mathsf{z}(a) + \mathsf{D}_\lambda(\mathsf{z}(a)) \varepsilon$. Since the left-hand side has no nilpotent component, it follows that $\mathsf{D}_\lambda(\mathsf{z}(a)) =0$ for all $a \in R$. On the other hand, the last diagram of (\ref{prediffbuneq}) says that $\rotatebox[origin=c]{180}{$\mathsf{T}$}(\lambda) \circ \lambda = \ell_E \circ \lambda$. Then using that $\mathsf{q}(\mathsf{D}_\lambda(x))=0$, $\mathsf{q}(\mathsf{z}(a)) =a$, and $\mathsf{D}_\lambda(\mathsf{z}(a)) =0$, the last diagram of (\ref{prediffbuneq}) explicitly states that $\mathsf{q}(\mathsf{z}(x)) + \mathsf{D}_\lambda(\mathsf{D}_\lambda(x)) \varepsilon\varepsilon^\prime = \mathsf{q}(\mathsf{z}(x)) + \mathsf{D}_\lambda(x) \varepsilon\varepsilon^\prime$. This implies that $\mathsf{D}_\lambda(\mathsf{D}_\lambda(x)) = \mathsf{D}_\lambda(x)$ for all $x \in E$. With these identities, we can now show that $\beta_\mathcal{E}$ is a differential bundle morphism over $R$. So we show that the diagrams of (\ref{Adiffbunmap}) hold: 
\begin{enumerate}[{\em (i)}]
\item $ \mathsf{q}_{\mathsf{ker}(\mathsf{q})} \circ \beta_\mathcal{E} = \mathsf{q}$:
\begin{align*}
    \mathsf{q}_{\mathsf{ker}(\mathsf{q})}( \beta_\mathcal{E} (x) ) = \mathsf{q}_{\mathsf{ker}(\mathsf{q})}( \mathsf{q}(x) + \mathsf{D}_\lambda(x) \varepsilon ) = \mathsf{q}(x) 
\end{align*}
\item $\rotatebox[origin=c]{180}{$\mathsf{T}$}( \beta_\mathcal{E}) \circ \lambda = \lambda_{\mathsf{ker}(\mathsf{q})} \circ \beta_\mathcal{E}$:
\[ \rotatebox[origin=c]{180}{$\mathsf{T}$}( \beta_\mathcal{E} ) (\lambda(x) ) 
=
\rotatebox[origin=c]{180}{$\mathsf{T}$}( \beta_\mathcal{E} ) (\mathsf{z}(\mathsf{q}(x))+ \mathsf{D}_\lambda(x) \varepsilon )
=
\beta_\mathcal{E}(\mathsf{z}(\mathsf{q}(x)) + \beta_\mathcal{E}(\mathsf{D}_\lambda(x) \varepsilon) \varepsilon^\prime
\]
\[
=
\mathsf{q}(\mathsf{z}(\mathsf{q}(x)) )  + \mathsf{D}_\lambda(\mathsf{z}(\mathsf{q}(x)) ) \varepsilon + \mathsf{z}(\mathsf{q}(\mathsf{D}_\lambda(x))) \varepsilon^\prime + \mathsf{D}_\lambda(\mathsf{D}_\lambda(x)) \varepsilon\varepsilon^\prime
\]
\[
=
\mathsf{q}(x) + 0 \varepsilon + 0 \varepsilon^\prime + \mathsf{D}_\lambda(x) \varepsilon\varepsilon^\prime
=
\mathsf{q}(x) + \mathsf{D}_\lambda(x) \varepsilon\varepsilon^\prime
\]
\[
=
\lambda_{\mathsf{ker}(\mathsf{q})}(\mathsf{q}(x) + \mathsf{D}_\lambda(x) \varepsilon )
=
\lambda_{\mathsf{ker}(\mathsf{q})} ( \beta_\mathcal{E} (x) ) 
\]
\end{enumerate}
So $\beta_\mathcal{E}$ is a differential bundle morphism. Since $\beta_\mathcal{E}$ is a ring isomorphism, by Lemma \ref{cor:diffbuniso} it then follows that $\beta^{-1}_\mathcal{E}$ is also a differential bundle morphism. In particular, note that this implies that $\lambda \circ \beta^{-1}_\mathcal{E} = \rotatebox[origin=c]{180}{$\mathsf{T}$}( \beta^{-1}_\mathcal{E} ) \circ \lambda_{\mathsf{ker}(\mathsf{q})}$. But by definition, we have that $\lambda \circ \beta^{-1}_\mathcal{E} = \zeta_\mathcal{E}$, and so we also have that $\rotatebox[origin=c]{180}{$\mathsf{T}$}( \beta^{-1}_\mathcal{E} ) \circ \lambda_{\mathsf{ker}(\mathsf{q})} = \zeta_\mathcal{E}$. This implies that $\beta^{-1}_\mathcal{E}(a) + \beta^{-1}_\mathcal{E}(x \varepsilon) \varepsilon = \mathsf{z}(a) + x \varepsilon$, and so $\beta^{-1}_\mathcal{E}(a) = \mathsf{z}(a)$ and $\beta^{-1}_\mathcal{E}(x \varepsilon) = x$. Therefore, $\beta^{-1}_\mathcal{E}(a + x \varepsilon) = \mathsf{z}(a) + x$. So we conclude that  $\beta_\mathcal{E}$ and $\beta^{-1}_\mathcal{E}$ are differential bundle isomorphisms over $R$. 
\end{proof}

Therefore, the construction from a module to a differential bundle is the inverse of the construction from a differential bundle to a module. So we conclude that: 

\begin{proposition}\label{prop:mod-diff-ring} For a commutative ring $R$, there is a bijective correspondence between $R$-modules and differential bundles (with negatives) over $R$ in $(\mathsf{CRING}, \rotatebox[origin=c]{180}{$\mathbb{T}$})$. 
\end{proposition}

In $\mathsf{CRING}$, recall that the terminal object is the zero ring $0$. So differential objects in $(\mathsf{CRING}, \rotatebox[origin=c]{180}{$\mathbb{T}$})$ correspond precisely to $0$-modules. However, the only $0$-module is $0$. Therefore, there are no non-trivial differential objects in $(\mathsf{CRING}, \rotatebox[origin=c]{180}{$\mathbb{T}$})$.

\begin{corollary}\label{cor:diffobj-rin} The only differential object in $(\mathsf{CRING}, \rotatebox[origin=c]{180}{$\mathbb{T}$})$ is the zero ring $0$.
\end{corollary}

We now extend Proposition \ref{prop:mod-diff-ring} to an equivalence of categories. For a commutative ring $R$, let $\mathsf{MOD}_R$ be the category of $R$-modules and $R$-linear morphisms between them. We define an equivalence of categories between $\mathsf{MOD}_R$ from $\mathsf{DBUN}_{\rotatebox[origin=c]{180}{$\mathbb{T}$}}\left[ R\right]$ as follows: 
\begin{enumerate}[{\em (i)}]
\item Define the functor $\rotatebox[origin=c]{180}{$\mathsf{M}$}_R:\mathsf{MOD}_R \to \mathsf{DBUN}_{\rotatebox[origin=c]{180}{$\mathbb{T}$}}\left[ R\right]$ which sends an $R$-module $M$ to the differential bundle $\rotatebox[origin=c]{180}{$\mathsf{M}$}_R(M)$, and sends an $R$-linear morphism $f: M \to M^\prime$ to the differential bundle morphism over $R$ $\rotatebox[origin=c]{180}{$\mathsf{M}$}_R(f): \rotatebox[origin=c]{180}{$\mathsf{M}$}_R(M) \to \rotatebox[origin=c]{180}{$\mathsf{M}$}_R(M^\prime)$ where ${\rotatebox[origin=c]{180}{$\mathsf{M}$}_R(f): M[\varepsilon] \to M^\prime[\varepsilon]}$ is defined as: 
\[ \rotatebox[origin=c]{180}{$\mathsf{M}$}_R(f)(a + m\varepsilon) = a + f(m)\varepsilon \]
\item Define the functor $\rotatebox[origin=c]{180}{$\mathsf{M}$}^\circ_R: \mathsf{DBUN}_{\rotatebox[origin=c]{180}{$\mathbb{T}$}}\left[ R\right] \to \mathsf{MOD}_R$ which sends a differential bundle with negatives over $R$, $\mathcal{E} =(\mathsf{q}: E \to R, \sigma, \mathsf{z}, \lambda, \iota)$ to the $R$-module $\rotatebox[origin=c]{180}{$\mathsf{M}$}^\circ_R(\mathcal{E}) = \mathsf{ker}(\mathsf{q})$, and sends a differential bundle morphism $f: \mathcal{E} =(\mathsf{q}: E \to R, \sigma, \mathsf{z}, \lambda, \iota) \to \mathcal{E}^\prime =({\mathsf{q}^\prime: E^\prime \to R}, \sigma^\prime, \mathsf{z}^\prime, \lambda^\prime, \iota^\prime)$ over $R$ to the $R$-linear morphism $\rotatebox[origin=c]{180}{$\mathsf{M}$}^\circ_R(f) :  \mathsf{ker}(\mathsf{q}) \to  \mathsf{ker}(\mathsf{q}^\prime)$ defined as:
\[ \rotatebox[origin=c]{180}{$\mathsf{M}$}^\circ_R(f)(x) = f(x) \]
\item Define the natural isomorphism $\alpha: \mathsf{1}_{\mathsf{MOD}_R} \Rightarrow \rotatebox[origin=c]{180}{$\mathsf{M}$}_R^\circ \circ \rotatebox[origin=c]{180}{$\mathsf{M}$}_R$ with inverse $\alpha^{-1}: \rotatebox[origin=c]{180}{$\mathsf{M}$}_R^\circ \circ \rotatebox[origin=c]{180}{$\mathsf{M}$}_R \Rightarrow \mathsf{1}_{\mathsf{MOD}_R}$ as $\alpha_M$,  $\alpha^{-1}_M$ defined in Lemma \ref{lem:alphaiso}.
\item Define the natural isomorphism $\beta: \mathsf{1}_{\mathsf{DBUN}_{\rotatebox[origin=c]{180}{$\mathbb{T}$}}\left[ R\right]} \Rightarrow \rotatebox[origin=c]{180}{$\mathsf{M}$}_R \circ \rotatebox[origin=c]{180}{$\mathsf{M}$}^\circ_R$ with inverse $\beta^{-1}: \rotatebox[origin=c]{180}{$\mathsf{M}$}_R^\circ \circ \rotatebox[origin=c]{180}{$\mathsf{M}$}_R \Rightarrow \mathsf{1}_{\mathsf{DBUN}_{\rotatebox[origin=c]{180}{$\mathbb{T}$}}\left[ R\right]}$ as $\beta_{\mathcal{E}}$, $\beta^{-1}_{\mathcal{E}}$ defined in Lemma \ref{lem:betaiso}.
\end{enumerate}

\begin{theorem} \label{thm:mod-diff-ring-1} For a commutative ring $R$, we have an equivalence of categories: 
\[\mathsf{MOD}_R \simeq \mathsf{DBUN}_{\rotatebox[origin=c]{180}{$\mathbb{T}$}}\left[ R\right]\]
\end{theorem}
\begin{proof} We must first explain why $\rotatebox[origin=c]{180}{$\mathsf{M}$}_R$ and $\rotatebox[origin=c]{180}{$\mathsf{M}$}^\circ_R$ are well-defined on morphisms. So given an $R$-linear morphism $f: M \to M^\prime$, we must show that $\rotatebox[origin=c]{180}{$\mathsf{M}$}_R(f)$ is a differential bundle morphism over $R$. We leave it to the reader to check for themselves that $\rotatebox[origin=c]{180}{$\mathsf{M}$}_R(f)$ is a ring morphism. So it remains to show that the diagrams of (\ref{Adiffbunmap}) also hold: 
\begin{enumerate}[{\em (i)}]
\item $ \mathsf{q}_{M^\prime} \circ \rotatebox[origin=c]{180}{$\mathsf{M}$}_R(f) = \mathsf{q}_M$:
\begin{align*}
    \mathsf{q}_{M^\prime}\left( \rotatebox[origin=c]{180}{$\mathsf{M}$}_R(f)( a + m \varepsilon ) \right) = \mathsf{q}_{M^\prime}( a + f(m) \varepsilon ) = a =  \mathsf{q}_M(a + m\varepsilon) 
\end{align*}
\item $\rotatebox[origin=c]{180}{$\mathsf{T}$}\left( \rotatebox[origin=c]{180}{$\mathsf{M}$}_R(f) \right) \circ \lambda_M = \lambda_{M^\prime} \circ \rotatebox[origin=c]{180}{$\mathsf{M}$}_R(f)$:
\[
\rotatebox[origin=c]{180}{$\mathsf{T}$}\left( \rotatebox[origin=c]{180}{$\mathsf{M}$}_R(f) \right) \left( \lambda_M(a +m\varepsilon) \right)
=
\rotatebox[origin=c]{180}{$\mathsf{T}$}\left( \rotatebox[origin=c]{180}{$\mathsf{M}$}_R(f) \right) \left( a +m\varepsilon\varepsilon^\prime  \right)
=
\rotatebox[origin=c]{180}{$\mathsf{M}$}_R(f)(a) + \rotatebox[origin=c]{180}{$\mathsf{M}$}_R(f)(m \varepsilon)\varepsilon^\prime
\]
\[
= 
a + f(m)\varepsilon\varepsilon^\prime
=
\lambda_{M^\prime}\left( a + f(m)\varepsilon \right)
=
\lambda_{M^\prime}\left( \rotatebox[origin=c]{180}{$\mathsf{M}$}_R(f)(a + m\varepsilon) \right)
\]
\end{enumerate}
So we conclude that $\rotatebox[origin=c]{180}{$\mathsf{M}$}_R(f)$ is a differential bundle morphism over $R$. On the other hand, given a differential bundle morphism $f: \mathcal{E}\to \mathcal{E}^\prime$ over $R$, we must first explain why if $x \in \mathsf{ker}(\mathsf{q})$ then $\rotatebox[origin=c]{180}{$\mathsf{M}$}^\circ_R(f)(x) \in \mathsf{ker}(\mathsf{q}^\prime)$. Note that since $f$ is a differential bundle morphism over $R$, by definition this means that for all $x \in E$, $\mathsf{q}^\prime(f(x)) = \mathsf{q}(x)$. So it follows that if $x \in \mathsf{ker}(\mathsf{q})$, we have that: 
\[ \mathsf{q}^\prime\left( \rotatebox[origin=c]{180}{$\mathsf{M}$}^\circ_R(f)(x) \right) = \mathsf{q}^\prime(f(x)) = \mathsf{q}(x) = 0 \]
and so $\rotatebox[origin=c]{180}{$\mathsf{M}$}^\circ_R(f)(x) \in \mathsf{ker}(\mathsf{q}^\prime)$. Thus $\rotatebox[origin=c]{180}{$\mathsf{M}$}^\circ_R(f) :  \mathsf{ker}(\mathsf{q}) \to  \mathsf{ker}(\mathsf{q}^\prime)$ is well-defined. To show that $\rotatebox[origin=c]{180}{$\mathsf{M}$}^\circ_R(f)$ is $R$-linear, clearly since $f$ preserves the addition, $\rotatebox[origin=c]{180}{$\mathsf{M}$}^\circ_R(f)$ will also, therefore it remains to show $\rotatebox[origin=c]{180}{$\mathsf{M}$}^\circ_R(f)$ preserves the action. Since $f$ is a differential bundle morphisms over $R$, by Lemma \ref{lem:diffbunmapadd}, we have that $f$ preserves the zero, that is, $f(\mathsf{z}(a)) = \mathsf{z}^\prime(a)$ for all $a \in R$. So we compute: 
\[ a \cdot \rotatebox[origin=c]{180}{$\mathsf{M}$}^\circ_R(f)(x) = a \cdot f(x) = \mathsf{z}^\prime(a) f(x) = f(\mathsf{z}(a) x) = f(a \cdot x) =  \rotatebox[origin=c]{180}{$\mathsf{M}$}^\circ_R(f)(a \cdot x) \]
So we conclude that $\rotatebox[origin=c]{180}{$\mathsf{M}$}^\circ_R(f)$ is an $R$-linear morphism. So $\rotatebox[origin=c]{180}{$\mathsf{M}$}_R$ and $\rotatebox[origin=c]{180}{$\mathsf{M}$}^\circ_R$ are well-defined, and it is straightforward to see that they also preserve composition and identities, so $\rotatebox[origin=c]{180}{$\mathsf{M}$}_R$ and $\rotatebox[origin=c]{180}{$\mathsf{M}$}^\circ_R$ are indeed functors. 

Next we explain why $\alpha$, $\alpha^{-1}$, $\beta$, and $\beta^{-1}$ are natural transformations. In fact, it suffices to explain why $\alpha$ and $\beta^{-1}$ are natural, and it will then follow that $\alpha^{-1}$ and $\beta$ are as well since we have already shown they are isomorphisms on each object. So for an $R$-linear morphism $f: M \to M^\prime$, we compute: 
\[ \rotatebox[origin=c]{180}{$\mathsf{M}$}_R^\circ\left( \rotatebox[origin=c]{180}{$\mathsf{M}$}_R (f) \right) \left( \alpha_M(m) \right) = \rotatebox[origin=c]{180}{$\mathsf{M}$}_R^\circ\left( \rotatebox[origin=c]{180}{$\mathsf{M}$}_R (f) \right) \left( m \varepsilon \right) = \rotatebox[origin=c]{180}{$\mathsf{M}$}_R (f)(m \varepsilon) = f(m) \varepsilon = \alpha_{M^\prime}(f(m)) \]
So $\alpha$ is indeed a natural transformation, and so $\alpha^{-1}$ will also be a natural transformation. Therefore, $\alpha$ and $\alpha^{-1}$ are inverse natural isomorphisms. On the other hand, for a differential bundle morphism $f: \mathcal{E}\to \mathcal{E}^\prime$ over $R$, we compute: 
\[
\beta^{-1}_{\mathcal{E}^\prime}\left( \rotatebox[origin=c]{180}{$\mathsf{M}$}_R\left( \rotatebox[origin=c]{180}{$\mathsf{M}$}^\circ_R (f) \right) \left( a + x\varepsilon \right) \right)
=
\beta^{-1}_{\mathcal{E}^\prime}\left( a + \rotatebox[origin=c]{180}{$\mathsf{M}$}^\circ_R (f)(x)\varepsilon  \right)
=
\beta^{-1}_{\mathcal{E}^\prime}\left( a + f(x)\varepsilon  \right)
\]
\[
=
\mathsf{z}^\prime(a) + f(x)\varepsilon
=
f(\mathsf{z}(a)) + f(x) = f( \mathsf{z}(a) + x)
=
f\left( \beta^{-1}_{\mathcal{E}} (a + x \varepsilon) \right) 
\]
So $\beta^{-1}$ is indeed a natural transformation, and so $\beta$ will also be a natural transformation. Therefore, $\beta$ and $\beta^{-1}$ are inverse natural isomorphisms. So we conclude that we have an equivalence of categories, and so $\mathsf{MOD}_R \simeq \mathsf{DBUN}_{\rotatebox[origin=c]{180}{$\mathbb{T}$}}\left[ R\right]$.
\end{proof}

We also obtain an equivalence of categories between the category of all differential bundles and the category of modules. Let $\mathsf{MOD}$ be the category whose objects are pairs $(R,M)$ consisting of a commutative ring $R$ and an $R$-module $M$, and where a map is a pair ${(g,f): (R,M) \to (R^\prime, M^\prime)}$ consisting of a ring morphism $g: R \to R^\prime$ and an $R$-linear map $f: M \to M^\prime$, where $M^\prime$ is an $R$-module via the action $a \bullet m = g(a) \cdot m$, so explicitly, $f(a \cdot m) = g(a) \cdot f(m)$. Composition is defined as $(g^\prime, f^\prime) \circ (g,f) = (g^\prime \circ g, f^\prime \circ f)$ and identities are $(1_R, 1_M)$. 

We define an equivalence of categories between $\mathsf{MOD}$ and $\mathsf{DBUN}\left[ (\mathsf{CRING}, \rotatebox[origin=c]{180}{$\mathbb{T}$})\right]$ as follows: 
\begin{enumerate}[{\em (i)}]
\item Define the functor $\rotatebox[origin=c]{180}{$\mathsf{M}$}:\mathsf{MOD} \to \mathsf{DBUN}\left[ (\mathsf{CRING}, \rotatebox[origin=c]{180}{$\mathbb{T}$})\right]$ which sends an object $(R,M)$ to the differential bundle $\rotatebox[origin=c]{180}{$\mathsf{M}$}(R,M) = \rotatebox[origin=c]{180}{$\mathsf{M}$}_R(M)$, and sends a map $(g,f): (R,M) \to (R^\prime, M^\prime)$ to the differential bundle morphism $\rotatebox[origin=c]{180}{$\mathsf{M}$}(g,f): \rotatebox[origin=c]{180}{$\mathsf{M}$}_R(M) \to \rotatebox[origin=c]{180}{$\mathsf{M}$}_{R^\prime}(M^\prime)$ defined as: 
\[ \rotatebox[origin=c]{180}{$\mathsf{M}$}(g,f)(a + m\varepsilon) = g(a) + f(m)\varepsilon \]
\item Define the functor $\rotatebox[origin=c]{180}{$\mathsf{M}$}^\circ: \mathsf{DBUN}\left[ (\mathsf{CRING}, \rotatebox[origin=c]{180}{$\mathbb{T}$})\right] \to \mathsf{MOD}$ which sends a differential bundle with negatives $\mathcal{E} =(\mathsf{q}: E \to R, \sigma, \mathsf{z}, \lambda, \iota)$ to the pair $\rotatebox[origin=c]{180}{$\mathsf{M}$}^\circ(\mathcal{E}) = (R, \mathsf{ker}(\mathsf{q}))$, and sends a differential bundle morphism $(f,g): \mathcal{E} =(\mathsf{q}: E \to R, \sigma, \mathsf{z}, \lambda, \iota) \to \mathcal{E}^\prime =({\mathsf{q}^\prime: E^\prime \to R^\prime}, \sigma^\prime, \mathsf{z}^\prime, \lambda^\prime, \iota^\prime)$ to the pair $\rotatebox[origin=c]{180}{$\mathsf{M}$}^\circ(f,g) = (g, \rotatebox[origin=c]{180}{$\mathsf{M}$}^\circ_{R}(f))$. 
\item Define the natural isomorphism $\overline{\alpha}: \mathsf{1}_{\mathsf{MOD}} \Rightarrow \rotatebox[origin=c]{180}{$\mathsf{M}$}^\circ \circ \rotatebox[origin=c]{180}{$\mathsf{M}$}$ as $\overline{\alpha}_{(R,M)} = (1_R, \alpha_M)$, with inverse natural isomorphism $\overline{\alpha}^{-1}: \rotatebox[origin=c]{180}{$\mathsf{M}$}^\circ \circ \rotatebox[origin=c]{180}{$\mathsf{M}$} \Rightarrow \mathsf{1}_{\mathsf{MOD}}$ defined as $\overline{\alpha}^{-1}_{(R,M)} = (1_R, \alpha^{-1}_M)$.
\item Define the natural isomorphism $\overline{\beta}: \mathsf{1}_{\mathsf{DBUN}\left[ (\mathsf{CRING}, \rotatebox[origin=c]{180}{$\mathbb{T}$})\right]} \Rightarrow \rotatebox[origin=c]{180}{$\mathsf{M}$} \circ \rotatebox[origin=c]{180}{$\mathsf{M}$}^\circ$ as $\overline{\beta}_\mathcal{E} = (1, \beta_\mathcal{E})$, with inverse natural isomorphism $\overline{\beta}^{-1}: \rotatebox[origin=c]{180}{$\mathsf{M}$}_R^\circ \circ \rotatebox[origin=c]{180}{$\mathsf{M}$}_R \Rightarrow \mathsf{1}_{\mathsf{DBUN}\left[ (\mathsf{CRING}, \rotatebox[origin=c]{180}{$\mathbb{T}$})\right]}$ as $\overline{\beta}^{-1}_\mathcal{E} = (1, \beta^{-1}_\mathcal{E})$.
\end{enumerate}

\begin{theorem} \label{thm:mod-diff-ring-2} We have an equivalence of categories: 
\[\mathsf{MOD} \simeq \mathsf{DBUN}\left[ (\mathsf{CRING}, \rotatebox[origin=c]{180}{$\mathbb{T}$})\right]\] 
\end{theorem}
\begin{proof} That $\rotatebox[origin=c]{180}{$\mathsf{M}$}$ and $\rotatebox[origin=c]{180}{$\mathsf{M}$}^\circ$ are well-defined on morphisms is similar to the proofs that $\rotatebox[origin=c]{180}{$\mathsf{M}$}_R$ and $\rotatebox[origin=c]{180}{$\mathsf{M}$}^\circ_R$ were well-defined on morphisms in the proof of Theorem \ref{thm:mod-diff-ring-1}. So $\rotatebox[origin=c]{180}{$\mathsf{M}$}$ and $\rotatebox[origin=c]{180}{$\mathsf{M}$}^\circ$ are indeed functors. Next, since $\alpha_M$ and $\alpha^{-1}_M$ are $R$-linear morphisms, it follows that $\overline{\alpha}_{(R,M)} = (1_R, \alpha_M)$ and $\overline{\alpha}^{-1}_{(R,M)} = (1_R, \alpha^{-1}_M)$ are indeed maps in $\mathsf{MOD}$, so $\overline{\alpha}$ and $\overline{\alpha}^{-1}$ are well-defined. On the other hand, since $\beta_\mathcal{E}$ and $\beta^{-1}_\mathcal{E}$ are differential bundle morphisms over the base commutative ring, it follows by definition that $\overline{\beta}_\mathcal{E} = (1, \beta_\mathcal{E})$ and $\overline{\beta}^{-1}_\mathcal{E} = (1, \beta^{-1}_\mathcal{E})$ are differential bundle morphisms, so $\overline{\beta}$ and $\overline{\beta}^{-1}$ are well-defined. Lastly, that $\overline{\alpha}$, $\overline{\alpha}^{-1}$, $\overline{\beta}$, and $\overline{\beta}^{-1}$ are natural isomorphisms follows directly from the fact that $\alpha$, $\alpha^{-1}$, $\beta$, and $\beta^{-1}$ are natural isomorphisms. So we conclude that we have an equivalence of categories: $\mathsf{MOD} \simeq \mathsf{DBUN}\left[ (\mathsf{CRING}, \rotatebox[origin=c]{180}{$\mathbb{T}$})\right]$.
\end{proof}

\begin{remark} The equivalence between modules and differential bundles is also true in more general settings. Indeed, both for the tangent category of commutative semirings and the tangent category of commutative algebras over a (semi)ring, differential bundles correspond precisely to modules via the above constructions. However, in a setting where one does not have negatives, we would have also needed to check an extra pullback square, since this is also required to make a pre-differential bundle a differential bundle in a Cartesian tangent category without negatives \cite[Proposition 6]{macadam2021vector}. Even more generally, in a codifferential category, every module of an algebra of the monad\footnote{Note that any algebra over a monad in a codifferential category acquires the structure of a commutative algebra with respect to the categorical monoidal structure, and hence we can talk about modules over it; see \cite[Theorem 3.12]{blute2015derivations}.} will induce a differential bundle in the Eilenberg-Moore category via \cite[Theorem 5.1]{blute2015derivations} and a generalization of Lemma \ref{lem:mod-diff}. If a codifferential category has kernels, then every differential bundle induces a module by generalizing Lemma \ref{lem:diff-mod}, and so in the presence of kernels, differential bundles in the Eilenberg-Moore category also correspond precisely to modules. However, since a codifferential category may not have all kernels, there may be differential bundles which are not induced by modules.
\end{remark}

\section{Differential Bundles for (Affine) Schemes}\label{diff_bundles_cringOp}

In this section, we characterize differential bundles (with negatives) in the tangent category of affine schemes and prove that they also correspond to modules (Proposition \ref{prop:mod-diff-aff}).  However, the constructions are quite different in this case.  To go from a module to a differential bundle, we take the free symmetric algebra over said module (Lemma \ref{lem:mod-diff-aff}). To go from a differential bundle to a module, we take the image of the derivation induced by the lift (Lemma \ref{lem:diff-mod-aff}). Moreover, in contrast to the previous section, in this case we obtain that the category of differential bundles is equivalent to the \emph{opposite} category of modules (Theorem \ref{thm:mod-diff-ring-aff-1} and Theorem \ref{thm:mod-diff-ring-aff-2}). To the best of our understanding, there is no general reason why the fact that differential bundles in commutative rings are equivalent to the category of modules would also imply that differential bundles in the opposite category of commutative rings are equivalent to the opposite category of modules. We conclude the section by generalizing these results to the category of schemes, where differential bundles are equivalent to the opposite category of quasicoherent sheaves of modules.  

\subsection{Tangent Category of (Affine) Schemes}\label{sec:comringoptan}

In this section, we discuss the tangent categories of affine schemes and schemes, where the tangent structure is induced by Kähler differentials. In this paper, by the category of affine schemes we mean the opposite category of commutative rings $\mathsf{CRING}^{op}$. As such, we will be working directly with $\mathsf{CRING}^{op}$, so we will write in terms of commutative rings $R$ instead of affine schemes $\mathsf{Spec}(R)$. While $\mathsf{CRING}^{op}$ has been mentioned as an example of a tangent category in other papers \cite{cockett2014differential,GARNER2018668}, a full explicit description of its tangent structure has not previously been given in the literature.  We provide such a description here as it will both be useful for the story of this paper and for future work on applications of tangent category theory in algebraic geometry. 

To give a Rosický tangent structure on $\mathsf{CRING}^{op}$, we must give a ``co-Rosický tangent structure'' on $\mathsf{CRING}$. Explicitly, this means giving a functor $\mathsf{T}: \mathsf{CRING} \to \mathsf{CRING}$ and natural transformations, and so in particular ring morphisms, of type: $\mathsf{p}_{R}: R \to \mathsf{T}(R)$, $+_{R}: \mathsf{T}(R) \to \mathsf{T}_2(R)$, $0_{R}: \mathsf{T}(R) \to R$, $\ell_{R}: \mathsf{T}^2(R) \to \mathsf{T}(R)$, $\mathsf{c}_{R}: \mathsf{T}^2(R) \to \mathsf{T}^2(R)$, and $-_R: \mathsf{T}(R) \to \mathsf{T}(R)$. 

For a commutative ring $R$, its tangent bundle $\mathsf{T}(R)$ is the free symmetric $R$-algebra over its modules of Kähler differentials $\Omega(R)$: 
\[ \mathsf{T}(R) := \mathsf{Sym}_R \left( \Omega(R) \right) = \bigoplus \limits_{n=0}^{\infty} \Omega(R)^{{\otimes^s_R}^n} = R \oplus \Omega(R) \oplus \left( \Omega(R) \otimes^s_R \Omega(R) \right) \oplus \hdots \]
where $\otimes^s_R$ is the symmetrized tensor product over $R$. In \cite[Definition 16.5.12.I]{grothendieck1966elements}, Grothen-dieck calls $\mathsf{T}(R)$ the ``fibré tangente'' (French for tangent bundle) of $R$, while in \cite[Section 2.6]{jubin2014tangent}, Jubin calls $\mathsf{T}(R)$ the tangent algebra of $R$. 

For the story of this paper, it will be useful to have a more explicit description of $\mathsf{T}(R)$. So equivalently, $\mathsf{T}(R)$ is the $R$-algebra generated by the set $\lbrace \mathsf{d}(a) \vert~ a \in R \rbrace$ modulo the equations:
\begin{align*}
  \mathsf{d}(1) = 0 && \mathsf{d}(a+b) = \mathsf{d}(a) + \mathsf{d}(b) && \mathsf{d}(ab) = a \mathsf{d}(b) + b \mathsf{d}(a)
\end{align*}
which are the same equations that are modded out to construct the module of Kähler differentials of $R$. Thus, an arbitrary element of $\mathsf{T}(R)$ is a finite sum of monomials of the form $a \mathsf{d}(b_1) \hdots \mathsf{d}(b_n)$. So the ring structure of $\mathsf{T}(R)$ is essentially the same as that of polynomials. Furthermore, $\mathsf{T}(R)$ also has a universal property similar to that of the module of Kähler differentials, but instead for algebras. For a commutative $R$-algebra $A$, a derivation evaluated in $A$ is a linear map $\mathsf{D}: R \to A$ which satisfies the product rule $\mathsf{D}(ab)= a \cdot \mathsf{D}(b) + b \cdot \mathsf{D}(a)$. Now $\mathsf{T}(R)$ is a commutative $R$-algebra $A$ via the $R$-module action given by multiplication, $a \cdot w = aw$ for $a \in R$ and $w \in \mathsf{T}(R)$. Then the map $\mathsf{d}: R \to \mathsf{T}(R)$, which maps $a$ to $\mathsf{d}(a)$, is a derivation and is universal in the sense that for any commutative $R$-algebra $A$ equipped with a derivation $\mathsf{D}: R \to A$, there exists a unique $R$-algebra morphism $\mathsf{D}^\flat: \mathsf{T}(R) \to A$ such that $\mathsf{D}^\flat(\mathsf{d}(a))= \mathsf{D}(a)$. 

\begin{example} It may be useful to work out some basic examples of tangent bundles: 
\begin{enumerate}[{\em (i)}]
\item For the ring of integers $\mathbb{Z}$, its tangent bundle is itself: $\mathsf{T}(\mathbb{Z}) = \mathbb{Z}$
\item For the polynomial ring in $n$-variables $\mathbb{Z}[x_1, \hdots, x_n]$, its tangent bundle is the polynomial ring in $2n$-variables: 
\[\mathsf{T}(\mathbb{Z}[x_1, \hdots, x_n]) = \mathbb{Z}[x_1, \hdots, x_n, \mathsf{d}(x_1), \hdots, \mathsf{d}(x_n)]\] 
with no added assumptions on the variables $\mathsf{d}(x_i)$;
\item For coordinate rings of varieties, that is, the polynomial rings quotiented by some finitely generated ideal $\mathbb{Z}[x_1,\hdots, x_n]/\langle p(\vec x), \hdots, q(\vec x) \rangle$, its tangent bundle is the polynomial ring's tangent bundle quotiented by the ideal generated by the same polynomials and their total derivatives, so that $\mathsf{T}\left(\mathbb{Z}[x_1, \hdots, x_n]/\langle p(\vec x), \hdots, q(\vec x) \rangle \right)$ is 
\[ \mathbb{Z}[x_1, \hdots, x_n, \mathsf{d}(x_1), \hdots, \mathsf{d}(x_n)]/\langle p(\vec x), \hdots, q(\vec x), \mathsf{d}(p)(\vec x), \hdots, \mathsf{d}(q)(\vec x) \rangle \]
For example, for $\mathbb{Z}[x,y]/\langle x^2 - xy^2 \rangle$, its tangent bundle is: 
\[ \mathbb{Z}[x,y, \mathsf{d}(x), \mathsf{d}(y)]/\langle x^2 - xy^2, 2x\mathsf{d}(x) - y^2\mathsf{d}(x) - 2xy\mathsf{d}(y) \rangle \]
\end{enumerate}
\end{example}

To define the necessary ring morphisms for the tangent structure, note that $\mathsf{T}(R)$ is generated by $a$ and $\mathsf{d}(a)$, for all $a \in R$. Therefore, to define ring morphisms with domain $\mathsf{T}(R)$, it suffices to define them on generators $a$ and $\mathsf{d}(a)$. Using this to our advantage, we can define a tangent structure on $\mathsf{CRING}^{op}$. 

\begin{enumerate}[{\em (i)}]
\item The endofunctor $\mathsf{T}: \mathsf{CRING} \to \mathsf{CRING}$ maps a commutative ring $R$ to its tangent bundle $\mathsf{T}(R)$ as defined above, and a ring morphism $f: R \to S$ is sent to the ring morphism $\mathsf{T}(f): \mathsf{T}(R) \to \mathsf{T}(S)$ defined as on generators as follows: 
\begin{align*}
    \mathsf{T}(f)(a) = f(a) && \mathsf{T}(f)(\mathsf{d}(a)) = \mathsf{d}(f(a))
\end{align*}
\item The projection $\mathsf{p}_{R}: R \to \mathsf{T}(R)$ is defined as the injection of $R$ into $\mathsf{T}(R)$:
\[\mathsf{p}_{R}(a) = a\]
\end{enumerate}

We also need the pullback of $n$ copies of $\mathsf{p}_{R}$ in $\mathsf{CRING}^{op}$, which means that we need the pushout of $n$ copies of $\mathsf{p}_{R}$ in $\mathsf{CRING}$. Recall that $\mathsf{CRING}$ is cocomplete, and therefore admits all pushouts. To describe the desired pushout, note that for commutative rings $R$ and $R^\prime$, any ring morphism $f: R \to R^\prime$ induces an $R$-algebra structure on $R^\prime$ via the $R$-module action $a \cdot x = f(a)x$ for all $a \in R$ and $x \in R^\prime$. Therefore, the pushout of $n$-copies of ${f: R \to R^\prime}$ is given by taking the tensor product over $R$ of $n$ copies of $R^\prime$ viewed as an $R$-module: $R^\prime_n := {R^\prime}^{\otimes^n_R}$, where $\otimes_R$ is the tensor product over $R$ of $R$-modules. The induced $R$-algebra structure on $\mathsf{T}(R)$ via $\mathsf{p}_{R}$ is precisely given by multiplication, as described above. 
\begin{enumerate}[{\em (i)}]
\setcounter{enumi}{2}
\item The pushouts of $n$ copies of $\mathsf{p}_{R}$ is given by $\mathsf{T}_n(R) := {\mathsf{T}(R)}^{\otimes^n_R}$ and where the pushout injections ${\pi_j: \mathsf{T}(R) \to \mathsf{T}_n(R)}$ injects $\mathsf{T}(R)$ into the $j$-th component:
\[\pi_j(w) = 1 \otimes_R \hdots \otimes_R 1 \otimes_R w \otimes_R 1 \otimes_R \hdots \otimes_R 1 \]
\end{enumerate}

To describe the sum, zero, and negative, let us first explain what additive bundles \cite[Section 2.1]{cockett2014differential} and Abelian group bundles \cite[Section 1]{rosicky1984abstract} are in $\mathsf{CRING}^{op}$. An additive bundle over $R$ in $\mathsf{CRING}^{op}$ corresponds precisely to a commutative $R$-bialgebra over the tensor product $\otimes_R$. The sum and zero of the additive bundle are the comultiplication and counit respectively of the $R$-coalgebra structure. The fact that they are ring morphisms and they commute with the additive bundle's projection will further imply that we obtain a commutative $R$-bialgebra. An Abelian group bundle over $R$ in $\mathsf{CRING}^{op}$ corresponds precisely to a commutative $R$-Hopf algebra over the tensor product $\otimes_R$. The negative of the Abelian group bundle gives the antipode for the $R$-Hopf algebra structure. So to give the sum, zero, and negative for our tangent structure, we must give a $R$-Hopf algebra structure on the tangent bundle $\mathsf{T}(R)$. Luckily, free symmetric $R$-algebras have a canonical commutative $R$-Hopf algebra. 

\begin{enumerate}[{\em (i)}]
\setcounter{enumi}{3}
\item The sum $+_{R}: \mathsf{T}(R) \to \mathsf{T}(R) \otimes_R \mathsf{T}(R)$ is given by the comultiplication of the canonical $R$-coalgebra structure of free symmetric $R$-algebras, that is, defined on generators as follows:
\begin{align*}
    +_R(a) = a \otimes_R 1 = 1 \otimes_R a && +_R(\mathsf{d}(a)) = \mathsf{d}(a) \otimes_R 1 + 1 \otimes_R \mathsf{d}(a)
\end{align*}
\item The zero $0_{R}: \mathsf{T}(R) \to R$ is the counit of the canonical $R$-coalgebra structure of free symmetric $R$-algebras, that is, defined on generators as follows: 
\begin{align*}
    0_R(a) = a && 0_R(\mathsf{d}(a)) = 0
\end{align*}
\item The negative $-_{R}: \mathsf{T}(R) \to \mathsf{T}(R)$ is the antipode of the canonical $R$-Hopf algebra structure of free symmetric $R$-algebras, that is, defined on generators as follows:
\begin{align*}
-_R(a) = a && -_R(\mathsf{d}(a))= - \mathsf{d}(a)
\end{align*}
\end{enumerate}

To describe the vertical lift and the canonical flip, let us first describe $\mathsf{T}^2(R)$ as the $R$-algebra generated by the set $\lbrace \mathsf{d}(a) \vert ~ a \in R \rbrace \cup \lbrace \mathsf{d}^\prime(a) \vert ~ a \in R \rbrace \cup \lbrace \mathsf{d}^\prime\mathsf{d}(a) \vert ~ a \in R \rbrace$, modulo the relations: 
\begin{gather*}
  \mathsf{d}(1) = 0 \qquad \mathsf{d}(a+b) = \mathsf{d}(a) + \mathsf{d}(b) \qquad \mathsf{d}(ab) = a \mathsf{d}(b) + b \mathsf{d}(a) \\
  \mathsf{d}^\prime(1) = 0 \qquad \mathsf{d}^\prime(a+b) = \mathsf{d}^\prime(a) + \mathsf{d}^\prime(b) \qquad \mathsf{d}^\prime(ab) = a \mathsf{d}^\prime(b) + b \mathsf{d}^\prime(a) \\
 \mathsf{d}^\prime\mathsf{d}(1)= 0 \qquad \mathsf{d}^\prime\mathsf{d}(a+b) =   \mathsf{d}^\prime\mathsf{d}(a) +  \mathsf{d}^\prime\mathsf{d}(b) \\
 \mathsf{d}^\prime\mathsf{d}(ab) = \mathsf{d}(b)\mathsf{d}^\prime(a)  + a \mathsf{d}^\prime\mathsf{d}(b) +  \mathsf{d}(a)\mathsf{d}^\prime(b)  + b \mathsf{d}^\prime\mathsf{d}(a)
 \end{gather*}
These relations say that $\mathsf{d}$ and $\mathsf{d}^\prime$ are derivations, and that $\mathsf{d}^\prime\mathsf{d}$ is the composite of derivations. Therefore, to define a ring morphism with domain $\mathsf{T}^2(R)$, it suffices to define it on the four types of generators $a$, $\mathsf{d}(a)$, $\mathsf{d}^\prime(a)$, and $ \mathsf{d}^\prime\mathsf{d}(a)$ for $a \in R$. 
\begin{enumerate}[{\em (i)}]
\setcounter{enumi}{6}
\item The vertical lift $\ell_{R}: \mathsf{T}^2(R) \to \mathsf{T}(R)$ is defined on generators as follows: 
\begin{align*}
   \ell_R(a) =a &&  \ell_R(\mathsf{d}(a)) =0 && \ell_R(\mathsf{d}^\prime(a)) =0 && \ell_R( \mathsf{d}^\prime\mathsf{d}(a)) = \mathsf{d}(a)
\end{align*}
\item The canonical flip $\mathsf{c}_{R}: \mathsf{T}^2(R) \to \mathsf{T}^2(R)$ is defined on generators as follows: 
\begin{align*}
   \mathsf{c}_R(a) =a &&  \mathsf{c}_R(\mathsf{d}(a)) =\mathsf{d}^\prime(a) && \mathsf{c}_R(\mathsf{d}^\prime(a)) =\mathsf{d}(a) && \mathsf{c}_R( \mathsf{d}^\prime\mathsf{d}(a)) = \mathsf{d}^\prime\mathsf{d}(a)
\end{align*}
\end{enumerate}
So $\mathbb{T} = (\mathsf{T}, \mathsf{p}, +, 0, \ell, \mathsf{c}, - )$ is a Rosický tangent structure on $\mathsf{CRING}^{op}$. Also, $\mathsf{CRING}$ has finite coproducts where the binary coproduct is given by the tensor product of rings $R \otimes S$ and where the initial object is the ring of integers $\mathbb{Z}$. Thus $\mathsf{CRING}^{op}$ has finite products. Since one has that $\Omega(R \otimes S) \cong R \otimes \Omega(S) \oplus S \otimes \Omega(R)$ and $\Omega(\mathbb{Z}) \cong 0$, it follows that that $\mathsf{T}(R \otimes S) \cong \mathsf{T}(R) \otimes \mathsf{T}(S)$ and $\mathsf{T}(\mathbb{Z}) \cong \mathbb{Z}$. So we have that: 

\begin{lemma}\label{lemma:cringOpisTan} $(\mathsf{CRING}^{op}, \mathbb{T})$ is a Cartesian Rosický tangent category. 
\end{lemma}

It is worth mentioning that the tangent structures for commutative rings and affine schemes are related to one another via the adjoint tangent structure theorem. Per \cite[Proposition 5.17]{cockett2014differential}, if the tangent bundle of a tangent category has a left adjoint, then this induces a tangent structure on the opposite category where the left adjoint is the tangent bundle. This is precisely what is happening between the tangent categories $(\mathsf{CRING}, \rotatebox[origin=c]{180}{$\mathbb{T}$})$ and $(\mathsf{CRING}^{op}, \mathbb{T})$. Indeed, $\mathsf{T}: \mathsf{CRING} \to \mathsf{CRING}$ is a left adjoint to ${\rotatebox[origin=c]{180}{$\mathsf{T}$}: \mathsf{CRING} \to \mathsf{CRING}}$, so we have a natural bijective correspondence between ring morphisms of type ${R \to R^\prime[\varepsilon]}$ and ${\mathsf{T}(R) \to R^\prime}$. Explicitly, given a ring morphism $f: R \to R^\prime[\varepsilon]$, which is of the form ${f(a) = f_1(a) + f_2(a) \varepsilon}$, define the ring morphism $f^\sharp: \mathsf{T}(R) \to R^\prime$ on generators as $f^\sharp(a) = f_1(a)$ and ${f^\sharp(\mathsf{d}(a)) = f_2(a)}$, and conversely, given a ring morphism $g: \mathsf{T}(R) \to R^\prime$, define the ring morphism ${g^\flat: R \to R^\prime[\varepsilon]}$ as: ${g^\flat(a) = g(a) + g(\mathsf{d}(a)) \varepsilon}$. 

Furthermore, $(\mathsf{CRING}^{op}, \mathbb{T})$ is not only a Cartesian Rosický tangent category but also a \emph{representable} tangent category \cite[Section 5.2]{cockett2014differential}. Briefly, a representable tangent category is a Cartesian category whose tangent bundle functor $\mathsf{T}$ is a representable functor $\mathsf{T} \cong (-)^D$ for some object $D$, that is, $\mathsf{T}$ is a right adjoint for the functor $\_ \times D$. The object $D$ is called the infinitesimal object \cite[Definition 5.6]{cockett2014differential}, and note that the opposite category of a representable category is a tangent category with tangent bundle functor $\_ \times D$ (where $\times$ becomes a coproduct in the opposite category). $(\mathsf{CRING}^{op}, \mathbb{T})$ is a representable tangent category where the infinitesimal object is the ring of dual numbers for the integers, $\mathbb{Z}[\varepsilon]$. So we have that $\mathbb{T}(R) \cong R^{\mathbb{Z}[\varepsilon]}$ in $\mathsf{CRING}^{op}$, and $\rotatebox[origin=c]{180}{$\mathsf{T}$}(R) \cong R \otimes \mathbb{Z}[\varepsilon]$ in $\mathsf{CRING}$.

Using Proposition \ref{prop:slice_tan_cat}, we then get that for each commutative ring $R$, the slice category $\mathsf{CRING}^{op}/R$ is also a tangent category.  But as is well-known, this slice category is equal to the (opposite of) the category of commutative $R$-algebras.  Thus we have:

\begin{corollary}\label{cor:tan_algebras}
For any commutative ring $R$, the opposite of the category of algebras over $R$, $(\mathsf{CALG}_R)^{op}$ is a Cartesian Rosický tangent category, with tangent functor given as in Proposition \ref{prop:slice_tan_cat}.  
\end{corollary}
In particular, this tangent structure on objects is given by (the symmetric algebra of) the  ``relative'' K\"{a}hler differentials: this is the same construction as seen earlier in this section, except with $d(r) = 0$ for all $r \in R$.  

\begin{remark} There are also other ways to generalize the tangent category structure of $\mathsf{CRING}^{op}$:
\begin{itemize}
    \item The opposite category of commutative semirings and the opposite category of commutative algebras over a commutative (semi)ring will be representable tangent categories via Kähler differentials in a similar fashion.
    \item The coEilenberg-Moore category of a differential category (or dually the opposite category of the Eilenberg-Moore category of a codifferential category) is a (representable) tangent category \cite[Theorem 26]{cockett_et_al:LIPIcs:2020:11660}, and these tangent categories of opposite categories of commutative (semi)rings/algebras are precisely the coEilenberg-Moore categories of the appropriate polynomial models of differential categories.
\end{itemize}
\end{remark}

The category of schemes $\mathsf{SCH}$ is also a Cartesian Rosický tangent category, whose tangent structure is built up from the tangent structure for affine schemes.  Indeed, recall that a scheme is by definition the gluing of affine schemes. So the tangent bundle of a scheme can be defined as the gluing of the tangent bundles of each affine piece of said scheme. For a different approach of describing the tangent structure of schemes, see \cite[Example 2.(iii)]{GARNER2018668}.

\begin{proposition} $(\mathsf{SCH}, \mathbb{T})$ is a Cartesian Rosický tangent category. 
\end{proposition}

As with affine schemes, we can apply Proposition \ref{prop:slice_tan_cat} to tangent structure on each category of relative schemes:

\begin{corollary}\label{cor:slice_schemes}
For each scheme $A$, the slice category $\mathsf{SCH}/A$ has the structure of a Cartesian Rosický tangent category.
\end{corollary}

\subsection{From Differential Bundles to Modules}

Let us begin by unravelling what a differential bundle with negatives would be in $(\mathsf{CRING}^{op}, \mathbb{T})$. First recall that $(\mathsf{CRING}^{op}, \mathbb{T})$ is a Cartesian Rosický tangent category, so by Proposition \ref{prop:Ben2}, differential bundles are the same thing as differential bundles with negatives. Also, as discussed in Section \ref{sec:comringoptan}, $\mathsf{CRING}^{op}$ admits all pullbacks since $\mathsf{CRING}$ admits all pushouts. For any ring morphism ${\mathsf{q}: R \to E}$ between commutative rings, recall that $E$ becomes a commutative $R$-algebra, so, in particular, an $R$-module, with action $a \cdot x = \mathsf{q}(a)x$. Then the pushout of $n$ copies of $\mathsf{q}$ in $\mathsf{CRING}$ is given by taking the tensor product over $R$ of $n$ copies of $E$: $E_n = E^{\otimes^n_R}$. Then a differential bundle with negatives over a commutative ring $R$ in $(\mathsf{CRING}^{op}, \mathbb{T})$ viewed in $\mathsf{CRING}$ would consist of a commutative ring $E$ and five ring morphisms: $\mathsf{q}: R \to E$, ${\sigma: E \to E \otimes_R E}$, ${\mathsf{z}: E \to R}$, $\lambda: \mathsf{T}(E) \to E$, and $\iota: E \to E$. The axioms of a differential bundle imply that $E$ is a commutative $R$-Hopf algebra, where the sum $\sigma$ is the comultiplication, the zero $\mathsf{z}$ is the counit, and the negative $\iota$ is the antipode. 

To obtain an $R$-module from a differential bundle $E$, we take the image of the map ${\mathsf{D}_\lambda: E \to E}$ defined as $\mathsf{D}_\lambda(x) = \lambda(\mathsf{d}(x))$, which is a derivation whose product rule is $\mathsf{D}_\lambda(ab) = \lambda(a) \mathsf{D}_\lambda(b) + \lambda(b) \mathsf{D}_\lambda(a)$. 

\begin{lemma}\label{lem:diff-mod-aff} Let $R$ be a commutative ring, and let $\mathcal{E} = (\mathsf{q}: E \to R, \sigma, \mathsf{z}, \lambda, \iota)$ be a differential bundle (with negatives) over $R$ in $(\mathsf{CRING}^{op}, \mathbb{T})$. Then the image of the derivation:
\[\mathsf{im}(\mathsf{D}_\lambda) = \lbrace \mathsf{D}_\lambda(x) = \lambda(\mathsf{d}(x)) \vert~ \forall x \in E \rbrace\] 
is an $R$-module with action $a \cdot \mathsf{D}_\lambda(x) = \mathsf{D}_\lambda(\mathsf{q}(a)x)$. 
\end{lemma}
\begin{proof} Recall that for any $R$-linear map $f: M \to N$, the image $\mathsf{im}(f) = \lbrace f(m) \vert~ \forall m\in M \rbrace$ is an $R$-module with action $a \cdot f(m) = f(a \cdot m)$. Therefore, to prove that $\mathsf{im}(\mathsf{D}_\lambda)$ is an $R$-module, it suffices to show that $\mathsf{D}_\lambda$ is an $R$-linear map. Clearly $\mathsf{D}_\lambda$ is additive, so it remains to show that $\mathsf{D}_\lambda$ also preserves the action. First note that the dual of one of the differential bundle axioms tells us that $\lambda \circ \mathsf{T}(q) = \mathsf{q} \circ 0_R$. In particular this implies that $\lambda(\mathsf{q}(a)) = \mathsf{q}(a)$ and $\lambda(\mathsf{d}(\mathsf{q}(a))) = 0$ for all $a \in R$. Note that the second equality can be rewritten as $ \mathsf{D}_\lambda( \mathsf{q}(a)  ) = 0$ for all $a \in R$. So we compute: 
\begin{align*}
\mathsf{D}_\lambda(a \cdot x) = \mathsf{D}_\lambda( \mathsf{q}(a) x ) =\lambda(\mathsf{q}(a)) \mathsf{D}_\lambda( x ) + \lambda(x) \mathsf{D}_\lambda( \mathsf{q}(a)  ) = \mathsf{q}(a) \mathsf{D}_\lambda(x) + 0 =   a \cdot \mathsf{D}_\lambda(x)
\end{align*}
So $\mathsf{D}_\lambda$ is $R$-linear and we conclude that $\mathsf{im}(\mathsf{D}_\lambda)$ is an $R$-module. 
\end{proof}

\subsection{From Modules to Differential Bundles}

We now construct a differential bundle from a module. For a commutative ring $R$ and an $R$-module $M$, let $\mathsf{Sym}_R(M)$ be the free symmetric $R$-algebra over $M$, that is: 
\[ \mathsf{Sym}_R \left( M \right) = \bigoplus \limits_{n=0}^{\infty} M^{{\otimes^s_R}^n} = R \oplus M \oplus \left( M \otimes^s_R M \right) \oplus \hdots \]
where $\otimes^s_R$ is the symmetrized tensor product over $R$. Note that as a commutative ring, $\mathsf{Sym}_R(M)$ is generated by all $a \in R$ and $m \in M$. Therefore, to define ring morphisms with domain $\mathsf{Sym}_R(M)$, it suffices to define them on generators $a$ and $m$. Using this to our advantage, we define a differential bundle with negatives over $R$ structure on $\mathsf{Sym}_R(M)$ viewed in $\mathsf{CRING}$ (so the differential bundle structure maps will all be backwards) as follows: 
\begin{enumerate}[{\em (i)}]
\item The projection $\mathsf{q}_{M}: R \to \mathsf{Sym}_R(M)$ is defined as the injection of $R$ into $\mathsf{Sym}_R(M)$:
\[ \mathsf{q}_{M}(a) = a \]
\item The pushouts (which recall are pullbacks in $\mathsf{CRING}^{op}$) are given by taking the tensor product over $R$ of $n$ copies of $\mathsf{Sym}_R(M)$, so we have that $\mathsf{Sym}_R(M)_n := {\mathsf{Sym}_R(M)}^{\otimes^n_R}$, where the $j$th injection ${\pi_j: \mathsf{Sym}_R(M) \to \mathsf{Sym}_R(M)_n}$ injects $\mathsf{Sym}_R(M)$ into the $j$-th component:
\[ \pi_j(w) = 1 \otimes_R \hdots \otimes_R 1 \otimes_R w \otimes_R 1 \otimes_R \hdots \otimes_R 1\]
\item The sum $\sigma_{M}: \mathsf{Sym}_R(M) \to \mathsf{Sym}_R(M) \otimes_R \mathsf{Sym}_R(M)$ is the canonical comultiplication of the free symmetric $R$-algebras, that is, defined on generators as follows:
\begin{align*}
  \sigma_M(a) = a \otimes_R 1 = 1 \otimes_R a && \sigma_M(m) = m \otimes_R 1 + 1 \otimes_R m  
\end{align*}
\item The zero $\mathsf{z}_{R}: \mathsf{Sym}_R(M) \to R$ is the canonical counit of the free symmetric $R$-algebras, that is, defined on generators as follows: 
\begin{align*}
\mathsf{z}_M(a) = a && \mathsf{z}_M(m) = 0
\end{align*}
\item The negative $\iota_M: \mathsf{Sym}_R(M) \to \mathsf{Sym}_R(M)$ is the canonical antipode of the free symmetric $R$-algebras, that is, defined on generators as follows:
\begin{align*}
    \iota_M(a) = a && \iota_M(m) = -m
\end{align*}
\end{enumerate}
To describe the lift, note that $\mathsf{T}(\mathsf{Sym}_R(M))$ as a commutative ring is generated by $a$, $m$, $\mathsf{d}(a)$, and $\mathsf{d}(m)$ for all $a \in R$ and $m \in M$ (and modulo the appropriate equations). 
\begin{enumerate}[{\em (i)}]
\setcounter{enumi}{6}
\item The lift $\lambda_{M}: \mathsf{T}(\mathsf{Sym}_R(M)) \to \mathsf{Sym}_R(M)$ is defined on generators as follows:
\begin{align*}
    \lambda_M(a) =a && \lambda_M(m) =0 && \lambda_M(\mathsf{d}(a)) =0 && \lambda_M(\mathsf{d}(m)) = m
\end{align*}
\end{enumerate}

\begin{lemma}\label{lem:mod-diff-aff} For every commutative ring $R$ and $R$-module $M$, 
\[\mathsf{M}_R(M) := (\mathsf{q}_M, \sigma_M, \mathsf{z}_M, \lambda_M, \iota_M)\] 
is a differential bundle with negatives over $R$ in $(\mathsf{CRING}^{op}, \mathbb{T})$. 
\end{lemma}
\begin{proof} To show that we have a differential bundle, we will instead show that we have a pre-differential bundle which satisfies (\ref{prediff-i}) and (\ref{prediff-ii}) in Proposition \ref{prop:prediff}. So to show that $(\mathsf{q}_M, \mathsf{z}_M, \lambda_M)$ is a pre-differential bundle in $(\mathsf{CRING}^{op}, \mathbb{T})$, we must show that the dual of the four equalities from Definition \ref{def:prediffbun} hold in $\mathsf{CRING}$. To do so, we show that these hold on the generators.  
\begin{enumerate}[{\em (i)}]
\item $\mathsf{z}_M \circ \mathsf{q}_M = 1_R$ 
\begin{align*}
        \mathsf{z}_M(\mathsf{q}_M( a) ) = \mathsf{z}_M(a) = a
\end{align*}
\item $\lambda_M \circ \mathsf{p}_{\mathsf{Sym}_R(M)} =  \mathsf{q}_M \circ \mathsf{z}_M$
\begin{gather*}
\lambda_M(\mathsf{p}_{\mathsf{Sym}_R(M)}(a)) = \lambda_M(a) = a = \mathsf{q}_M(a) = \mathsf{q}_M(\mathsf{z}_M(a)) \\ \\ 
\lambda_M(\mathsf{p}_{\mathsf{Sym}_R(M)}(m)) = \lambda_M(m) = 0 = \mathsf{q}_M(0) = \mathsf{q}_M(\mathsf{z}_M(m))
\end{gather*}
\item $\mathsf{z}_M  \circ 0_{\mathsf{Sym}_R(M)} = \mathsf{z}_M \circ \lambda_M$
\begin{gather*}
\mathsf{z}_M \left( 0_{\mathsf{Sym}_R(M)} (a) \right) = \mathsf{z}_M(a) = \mathsf{z}_M(\lambda_M(a)) \\ \\
\mathsf{z}_M \left( 0_{\mathsf{Sym}_R(M)} (m) \right) = \mathsf{z}_M(m) = 0 =  \mathsf{z}_M(0) = \mathsf{z}_M(\lambda_M(0))
\end{gather*}
\item $\lambda_M \circ \mathsf{T}(\lambda_M)= \lambda_M \circ \ell_{\mathsf{Sym}_R(M)}$: Note that $\mathsf{T}^2(\mathsf{Sym}_R(M))$ has eight kinds of generators, $a$, $m$, $\mathsf{d}(a)$, $\mathsf{d}(m)$, $\mathsf{d}^\prime(a)$, $\mathsf{d}^\prime(m)$, $\mathsf{d}^\prime\mathsf{d}(a)$, and $\mathsf{d}^\prime\mathsf{d}(m)$ for all $a \in R$ and $m\in M$. 
\begin{gather*}
\lambda_M(\mathsf{T}(\lambda_M)(a)) = \lambda_M(\lambda_M(a)) = \lambda_M(a) = \lambda_M( \ell_{\mathsf{Sym}_R(M)}(a)) \\ \\
\lambda_M(\mathsf{T}(\lambda_M)(m)) = \lambda_M(\lambda_M(m)) = \lambda_M(0) = 0 = \lambda_M(m)= \lambda_M( \ell_{\mathsf{Sym}_R(M)}(m)) \\ \\ 
\lambda_M(\mathsf{T}(\lambda_M)(\mathsf{d}(a))) = \lambda_M(\lambda_M(\mathsf{d}(a))) = \lambda_M(0) = \lambda_M( \ell_{\mathsf{Sym}_R(M)}(\mathsf{d}(a))) \\ \\ 
\lambda_M(\mathsf{T}(\lambda_M)(\mathsf{d}(m))) = \lambda_M(\lambda_M(\mathsf{d}(m))) = \lambda_M(0) = \lambda_M( \ell_{\mathsf{Sym}_R(M)}(\mathsf{d}(m))) \\ \\
\lambda_M(\mathsf{T}(\lambda_M)(\mathsf{d}^\prime(a))) = \lambda_M(\mathsf{d}(\lambda_M(a))) = \lambda_M(\mathsf{d}(a)) = 0 = \lambda_M(0) = \lambda_M( \ell_{\mathsf{Sym}_R(M)}(\mathsf{d}^\prime(a))) \\  \\ 
\lambda_M(\mathsf{T}(\lambda_M)(\mathsf{d}^\prime(m))) = \lambda_M(\mathsf{d}(\lambda_M(m))) = \lambda_M(\mathsf{d}(0)) = \lambda_M(0) = \lambda_M( \ell_{\mathsf{Sym}_R(M)}(\mathsf{d}^\prime(m))) \\ \\
\lambda_M(\mathsf{T}(\lambda_M)(\mathsf{d}^\prime\mathsf{d}(a))) \!= \!\lambda_M(\mathsf{d}(\lambda_M(\mathsf{d}(a)))) = \lambda_M(\mathsf{d}(0)) \\ 
~~~~~~~~= \lambda_M(0) = 0 = \lambda_M(\mathsf{d}(a)) \!=\! \lambda_M( \ell_{\mathsf{Sym}_R(M)}(\mathsf{d}^\prime\mathsf{d}(a))) \\ \\
\lambda_M(\mathsf{T}(\lambda_M)(\mathsf{d}^\prime\mathsf{d}(m))) = \lambda_M(\mathsf{d}(\lambda_M(\mathsf{d}(m)))) = \lambda_M(\mathsf{d}(m)) = \lambda_M( \ell_{\mathsf{Sym}_R(M)}(\mathsf{d}^\prime\mathsf{d}(m))) 
\end{gather*}
    \end{enumerate}
So the desired equalities hold and we conclude that $(\mathsf{q}_M, \mathsf{z}_M, \lambda_M)$ is a pre-differential bundle in $(\mathsf{CRING}^{op}, \mathbb{T})$. 
    
Next, we must show that this pre-differential bundle also satisfies the extra assumptions required to make it a differential bundle, or rather that the dual of the assumptions hold in $\mathsf{CRING}$. As explained above, the pushout of $n$ copies of the projection $\mathsf{q}_M$ exists, chosen here to be $\mathsf{Sym}_R(M)_n$, and since $\mathsf{T}$ is a left adjoint, it preserves all colimits, so $\mathsf{T}^n$ preserves these pushouts. Dualizing this, we conclude that $(\mathsf{q}_M, \mathsf{z}_M, \lambda_M)$ satisfies assumption (\ref{prediff-i}) of Proposition \ref{prop:prediff} in $\mathsf{CRING}^{op}$. 

Next, we must show that the dual of (\ref{prediff-ii}) of Proposition \ref{prop:prediff} also holds, that is, we must show that the following square is a pushout in $\mathsf{CRING}$: 
     \begin{equation}\label{}\begin{gathered} 
  \xymatrixcolsep{5pc}\xymatrix{ \mathsf{T}(R) \otimes \mathsf{Sym}_R(M) \ar[d]_-{[\mathsf{T}(\mathsf{q}_M), \mathsf{p}_{\mathsf{Sym}_R(M)}]} \ar[r]^-{[0_R, \mathsf{z}_M]} & R \ar[d]^-{\mathsf{q}_M} \\
\mathsf{T}(\mathsf{Sym}_R(M)) \ar[r]_-{\lambda_M} &\mathsf{Sym}_R(M)  }  \end{gathered}\end{equation}
where $[-,-]$ is the copairing operation of the coproduct, which recall in $\mathsf{CRING}$ is given by the tensor product. Now suppose that $S$ is a commutative ring, and we have ring morphisms $f: \mathsf{T}(\mathsf{Sym}_R(M)) \to S$ and ${g: R \to S}$ such that $f \circ  [\mathsf{T}(\mathsf{q}_M), \mathsf{p}_{\mathsf{Sym}_R(M)}] = g \circ [0_R, \mathsf{z}_M]$. In particular, this implies that for every $a \in R$ and $m \in M$,the following equalities hold:  
\begin{align*}
f(a) = g(a) && f(\mathsf{d}(a)) = 0 && f(m) = 0
\end{align*}
Then define the map $[f,g]: \mathsf{Sym}_R(M) \to S$ as the ring morphism defined on generators as follows: 
\begin{align}
 [f,g](a) = g(a) && [f,g](m) = f(\mathsf{d}(m))    
\end{align}
Next, we compute the following on generators: 
\begin{gather*}
[f,g](\mathsf{q}_M(a)) = [f,g](a) = g(a) \\ \\
  [f,g](\lambda_M(a)) = [f,g](a) = g(a) = f(a) \\ \\
  [f,g](\lambda_M(m)) = [f,g](0) = 0 = f(m) \\ \\ 
    [f,g](\lambda_M(\mathsf{d}(a))) = [f,g](0) = 0 = f(\mathsf{d}(a)) \\ \\ 
  [f,g]  (\lambda_M(\mathsf{d}(m))) = [f,g](m) = f(\mathsf{d}(m))
\end{gather*}
Thus it follows that $[f,g] \circ \lambda_M = f$ and $[f,g] \circ \mathsf{q}_M  = g$ as desired. Lastly, it remains to show that $[f,g]$ is the unique such ring morphism. So suppose we have a ring morphism $h: \mathsf{Sym}_R(M) \to S$ such that $h \circ \lambda_M = f$ and $h \circ \mathsf{q}_M= g$. Then on generators, we compute that: 
\begin{gather*}
 h(a) = h(\mathsf{q}_M(a)) = g(a) = [f,g](a) \\ \\
 h(m) = h(\lambda_M(\mathsf{d}(m))) = f(\mathsf{d}(m)) = [f,g](m)
\end{gather*}
Since $h$ and $[f,g]$ are ring morphisms that are equal on generators, it follows that $h = [f,g]$. Thus $[f,g]$ is unique. Thus we conclude the above diagram is a pushout in $\mathsf{CRING}$. Furthermore, since $\mathsf{T}$ is a left adjoint in $\mathsf{CRING}$, we also have that $\mathsf{T}^n$ preserves these pushouts. Dualizing this, it follows that $(\mathsf{q}_M, \mathsf{z}_M, \lambda_M)$ satisfies assumption (\ref{prediff-ii}) of Proposition \ref{prop:prediff} in $\mathsf{CRING}^{op}$. Therefore by Proposition \ref{prop:prediff}, the pre-differential bundle $(\mathsf{q}_M, \mathsf{z}_M, \lambda_M)$ will induce a differential bundle with negatives in $(\mathsf{CRING}^{op}, \mathbb{T})$.

It remains to construct the sum and the negative as in Proposition \ref{prop:prediff}, and show that these are the same as the proposed $\sigma$ and $\iota$ above. By dualizing the construction, the sum $\sigma$ is: 
\[ \sigma_M = \left [ [\pi_1 \circ \lambda_M, \pi_2 \circ \lambda_M] \circ +_{\mathsf{Sym}_R(M)}, \pi_j \circ \mathsf{q}_M \right ] \]
On generators, we compute: 
\begin{align*}
 \sigma_M(a) \!=\! \left [ [\pi_1 \circ \lambda_M, \pi_2 \circ \lambda_M] \circ +_{\mathsf{Sym}_R(M)}, \pi_j \circ \mathsf{q}_M \right ] \! (a) =  \pi_j(\mathsf{q}_M(a)) = \pi_j(a) = a \otimes_R 1 = 1 \otimes_R a
  \end{align*}
 \begin{align*}
 &\sigma_M(m) \!=\! \left [ [\pi_1 \circ \lambda_M, \pi_2 \circ \lambda_M] \circ +_{\mathsf{Sym}_R(M)}, \pi_j \circ \mathsf{q}_M \right ]\!(m) = [\pi_1 \circ \lambda_M, \pi_2 \circ \lambda_M](+_{\mathsf{Sym}_R(M)}(\mathsf{d}(m))) \\
 &= [\pi_1 \circ \lambda_M, \pi_2 \circ \lambda_M](\mathsf{d}(m) \otimes_R 1) + [\pi_1 \circ \lambda_M, \pi_2 \circ \lambda_M](1 \otimes_R \mathsf{d}(m)) \\
 &= \pi_1(\lambda_M(\mathsf{d}(m))) +\pi_2(\lambda_M(\mathsf{d}(m))) = \pi_1(m) + \pi_2(m) = m \otimes_R 1 + 1 \otimes_R m 
\end{align*}
Thus on generators, $\sigma_M(a)=a \otimes_R 1 = 1 \otimes_R a$ and $\sigma(m)=m \otimes_R 1 + 1 \otimes_R m$, as defined above. On the other hand, the negative $\iota$ is: 
\begin{align*}
 \iota_M = \left [\lambda_M \circ -_{\mathsf{Sym}_R(M)} ,  \mathsf{q}_M \right ]
\end{align*}
On generators, we compute: 
\begin{gather*}
    \iota_M(a) = \left [\lambda_M \circ -_{\mathsf{Sym}_R(M)} ,  \mathsf{q}_M \right ](a) = \mathsf{q}_M(a) = a \\ \\ 
      \iota_M(m) = \left [\lambda_M \circ -_{\mathsf{Sym}_R(M)} ,  \mathsf{q}_M \right ](m) = \lambda_M(-_{\mathsf{Sym}_R(M)}(\mathsf{d}(m)) 
      = - \lambda_M(\mathsf{d}(m)) = - m 
\end{gather*}
So on generators $\iota_M(a)=a$, and $\iota_M(m) = -m$ as desired. So we conclude that $\mathsf{M}_R(M) = (\mathsf{q}_M, \sigma_M, \mathsf{z}_M, $ $ \lambda_M, \iota_M)$ is a differential bundle with negatives over $R$ in $(\mathsf{CRING}^{op}, \mathbb{T})$. 
\end{proof}

\subsection{Equivalence}

We will now show that the constructions of Lemma \ref{lem:diff-mod-aff} and Lemma \ref{lem:mod-diff-aff} are inverses of each other. Starting from the module side of things, let $R$ be a commutative ring, $M$ an $R$-module, and consider the induced derivation ${\mathsf{D}_{\lambda_M}: \mathsf{Sym}_R(M) \to \mathsf{Sym}_R(M)}$. We will show that the image of the derivation is precisely $M$. 

\begin{lemma}\label{lem:miso} For every commutative ring $R$ and $R$-module $M$, $\mathsf{im}(\mathsf{D}_{\lambda_M})=M$ as $R$-modules.    
\end{lemma}
\begin{proof} Let us compute what this derivation does on pure symmetrized tensors. For degree $0$, that is, for $a \in R$ we have that: 
\begin{align*}
    \mathsf{D}_{\lambda_M}(a) = \lambda_M(\mathsf{d}(a)) = 0 
\end{align*}
so $\mathsf{D}_{\lambda_M}(a) =0$. For degree $1$, that is, for $m \in M$ we have that: 
\begin{align*}
    \mathsf{D}_{\lambda_M}(m) = \lambda_M(\mathsf{d}(m)) = m 
\end{align*}
so $\mathsf{D}_{\lambda_M}(m) =m$. For degree $2$, that is, for $m,n \in M$ using the product rule, we have that: 
\begin{align*}
    \mathsf{D}_{\lambda_M}(mn) = \lambda_M(m) \mathsf{D}_{\lambda_M}(n) + \lambda_M(n) \mathsf{D}_{\lambda_M}(m) = 0 + 0 = 0 
\end{align*}
(where we are writing $mn$ for the product of $m$ and $n$ in $\mathsf{Sym}_R(M)$).  And similarly for degree $n \geq 2$, again by using the product rule, we have that $\mathsf{D}_{\lambda_M}(m_1 m_2 \hdots m_n) =0$. So it follows that $\mathsf{im}(\mathsf{D}_{\lambda_M}) = \lbrace m \vert~ \forall m \in M \rbrace$, so $\mathsf{im}(\mathsf{D}_{\lambda_M})=M$. Furthermore, note that the multiplication of $a$ and $m$ in $\mathsf{Sym}_R(M)$ is precisely the module action, $am = a \cdot m$. Thus the induced action on $\mathsf{im}(\mathsf{D}_{\lambda_M})$ from Lemma \ref{lem:diff-mod-aff} is given by:
\[ a \cdot  \mathsf{D}_{\lambda_M}(m) = \mathsf{D}_{\lambda_M}(\mathsf{q}(a)m) = \mathsf{D}_{\lambda_M}(am) = \mathsf{D}_{\lambda_M}(a \cdot m) = a\cdot m  \]
So $\mathsf{im}(\mathsf{D}_{\lambda_M})=M$ as $R$-modules.
\end{proof}

Conversely, let us start from a differential bundle, so let $\mathcal{E} =(\mathsf{q}: E \to R, \sigma, \mathsf{z}, \lambda, \iota)$ be a differential bundle with negatives over a commutative ring $R$ in $(\mathsf{CRING}^{op}, \mathbb{T})$. To define a differential bundle isomorphism between $\mathcal{E}$ and $\mathsf{M}(\mathsf{im}(\mathsf{D}_\lambda)$, we first need to define ring isomorphisms between $E$ and $\mathsf{Sym}_R\left(\mathsf{im}(\mathsf{D}_\lambda) \right)$. Define the ring morphism ${\psi_\mathcal{E}: \mathsf{Sym}_R\left(\mathsf{im}(\mathsf{D}_\lambda) \right) \to E}$ on generators $a \in R$ and $x \in E$ as follows: 
\begin{align}
  \psi_\mathcal{E}(a) = \mathsf{q}(a) && \psi_\mathcal{E}\left( \mathsf{D}_\lambda(x) \right) =  \mathsf{D}_\lambda(x)  
\end{align}
Note that $\psi_\mathcal{E}$ can also be defined by the universal property of the free symmetric $R$-algebra, that is, it is the unique $R$-algebra morphism induced by the inclusion $\mathsf{im}(\mathsf{D}_\lambda) \to E$. To define the inverse we will need to use the dual of Rosický's universality diagram, which in this case asks that the following diagram be a pushout:
 \begin{equation}\label{}\begin{gathered} 
 \xymatrixcolsep{5pc}\xymatrix{ \mathsf{T}(R) \otimes E \ar[r]^-{[\mathsf{T}(\mathsf{q}), \mathsf{p}_E]} \ar[d]_-{[0_R, \mathsf{z}]} & \mathsf{T}(E) \ar[d]^-{\lambda}  \\ 
 R \ar[r]_-{\mathsf{q}}   & E  }  \end{gathered}\end{equation} 
So define the ring morphism $\delta_\mathcal{E}: \mathsf{T}(E) \to  \mathsf{Sym}_R\left(\mathsf{im}(\mathsf{D}_\lambda) \right)$ on generators $x \in E$ as follows: 
\begin{align}
  \delta_\mathcal{E}(x) = \mathsf{z}(x) && \delta_\mathcal{E}(\mathsf{d}(x)) = \mathsf{D}_\lambda(x) 
\end{align}
By universality of the pushout, define $\psi^{-1}_\mathcal{E}: \mathsf{ker}(\mathsf{q})[\varepsilon] \to E$ as the unique ring morphism which makes the following diagram commute: 
 \begin{equation}\label{pushout1}\begin{gathered} 
 \xymatrixcolsep{5pc}\xymatrix{ \mathsf{T}(R) \otimes E \ar[r]^-{[\mathsf{T}(\mathsf{q}), \mathsf{p}_E]} \ar[d]_-{[0_R, \mathsf{z}]} & \mathsf{T}(E) \ar[d]^-{\lambda} \ar@/^/[ddr]^-{\delta_{\mathcal{E}}} \\ 
 R \ar[r]_-{\mathsf{q}} \ar@/_/[drr]_-{\mathsf{q}_{\mathsf{im}(\mathsf{D}_\lambda)}}  & E \ar@{-->}[dr]_-{\psi^{-1}_\mathcal{E}}  \\
 & &  \mathsf{Sym}_R\left(\mathsf{im}(\mathsf{D}_\lambda) \right)  }  \end{gathered}\end{equation} 
so $\psi^{-1}_\mathcal{E} = \left[ \mathsf{q}_{\mathsf{im}(\mathsf{D}_\lambda)}, \delta_{\mathcal{E}} \right]$. 

\begin{lemma}\label{lem:psiiso} For a commutative ring $R$ and a differential bundle with negatives $\mathcal{E} = ({\mathsf{q}: E \to R}, \sigma, \mathsf{z}, \lambda,\iota)$ over $R$ in $(\mathsf{CRING}^{op}, \mathbb{T})$, $\psi_\mathcal{E}: \mathcal{E} \to \mathsf{M}(\mathsf{im}(\mathsf{D}_\lambda))$ is a differential bundle isomorphism over $R$ in $(\mathsf{CRING}^{op}, \mathbb{T})$ with inverse $\psi^{-1}_\mathcal{E}: \mathsf{M}(\mathsf{im}(\mathsf{D}_\lambda)) \to \mathcal{E}$. 
\end{lemma}
\begin{proof} We first explain why $\psi_\mathcal{E}$ and $\psi^{-1}_\mathcal{E}$ are well-defined ring morphisms. Clearly, $\psi_\mathcal{E}$ is well-defined by construction. On the other hand, to explain why $\psi^{-1}_\mathcal{E}$ is well-defined, we must show that the outer diagram of (\ref{pushout1}) commutes. First, note that by the dual of the first pre-differential bundle axiom in (\ref{prediffbuneq}), $\mathsf{z}(\mathsf{q}(a)) = a$ for all $a \in R$, and recall that $\mathsf{D}_\lambda(\mathsf{q}(a)) =0$ for all $a \in R$ as well. Then on generators $a \in R$ and $x \in E$ we compute:
\[
\delta_{\mathcal{E}}\left( [\mathsf{T}(\mathsf{q}), \mathsf{p}_E] (a \otimes x) \right)
\!=
\delta_{\mathcal{E}}\left( (\mathsf{T}(\mathsf{q})(a)\mathsf{p}_E(x)) \right)
\!=
\delta_{\mathcal{E}}\left( \mathsf{q}(a) x \right)
\!=
\delta_{\mathcal{E}}\left( \mathsf{q}(a) \right) \delta_{\mathcal{E}}\left( x \right)
\!=
\]
\[
\mathsf{q}_{\mathsf{im}(\mathsf{D}_\lambda)}\left(\mathsf{z}(\mathsf{q}(a)) \right) \mathsf{q}_{\mathsf{im}(\mathsf{D}_\lambda)}\left(\mathsf{z}(x) \right)
= 
\mathsf{q}_{\mathsf{im}(\mathsf{D}_\lambda)}\left( a \right) \mathsf{q}_{\mathsf{im}(\mathsf{D}_\lambda)}\left(\mathsf{z}(x) \right)
=
\mathsf{q}_{\mathsf{im}(\mathsf{D}_\lambda)}(a\mathsf{z}(x))
=
\]
\[
\mathsf{q}_{\mathsf{im}(\mathsf{D}_\lambda)}\left( 0_R(a) \mathsf{z}(x) \right)
=
\mathsf{q}_{\mathsf{im}(\mathsf{D}_\lambda)}\left( [0_R, \mathsf{z}] (a \otimes x) \right)
\]
and
\[
\delta_{\mathcal{E}}\left( [\mathsf{T}(\mathsf{q}), \mathsf{p}_E] (\mathsf{d}(a) \otimes x) \right)
=
\delta_{\mathcal{E}}\left( \mathsf{T}(\mathsf{q})(\mathsf{d}(a))\mathsf{p}_E(x) \right)
=
\delta_{\mathcal{E}}\left( \mathsf{d}(\mathsf{q}(a)) x \right)
=
\delta_{\mathcal{E}}\left( \mathsf{d}(\mathsf{q}(a))  \right) \delta_{\mathcal{E}}\left( x \right)
\]
\[
=
\mathsf{D}_\lambda(\mathsf{q}(a)) x
=
0
=
\mathsf{q}_{\mathsf{im}(\mathsf{D}_\lambda)}(0)
=
\mathsf{q}_{\mathsf{im}(\mathsf{D}_\lambda)}(0 \mathsf{z}(x))
\]
\[
=
\mathsf{q}_{\mathsf{im}(\mathsf{D}_\lambda)}\left( 0_R(\mathsf{d}(a)) \mathsf{z}(x) \right)
=
\mathsf{q}_{\mathsf{im}(\mathsf{D}_\lambda)}\left( [0_R, \mathsf{z}] (\mathsf{d}(a) \otimes x) \right)
\]
So $\delta_{\mathcal{E}} \circ [\mathsf{T}(\mathsf{q}), \mathsf{p}_E] = \mathsf{q}_{\mathsf{im}(\mathsf{D}_\lambda)} \circ [0_R, \mathsf{z}]$. Therefore, by the universal property of the pushout square, there exists a unique ring morphism $\psi^{-1}_\mathcal{E}: E \to \mathsf{Sym}_R\left(\mathsf{im}(\mathsf{D}_\lambda) \right)$ such that $\psi^{-1}_\mathcal{E} \circ \lambda = \delta_\mathcal{E}$ and $\psi^{-1}_\mathcal{E} \circ \mathsf{q} = \mathsf{q}_{\mathsf{im}(\mathsf{D}_\lambda)}$. In particular, these imply that for every $a \in R$ and $x \in E$ the following equalities hold: 
\begin{align*}
   \psi^{-1}_\mathcal{E} (\mathsf{q}(a)) = a && \psi^{-1}_\mathcal{E} (\mathsf{D}_\lambda(x)) = \mathsf{D}_\lambda(x) 
\end{align*}
Next we show that $\psi_\mathcal{E}$ and $\psi^{-1}_\mathcal{E}$ are inverses of each other. To show that $\psi^{-1}_\mathcal{E} \circ \psi_\mathcal{E} = 1_{\mathsf{Sym}_R\left(\mathsf{im}(\mathsf{D}_\lambda) \right)}$, we use the above identities and compute the following on generators $a \in R$ and $x \in E$:  
\begin{gather*}
\psi^{-1}_\mathcal{E}(\psi_\mathcal{E}(a) ) = \psi^{-1}_\mathcal{E}(\mathsf{q}(a)) = a \\ \\ 
\psi^{-1}_\mathcal{E}\left(\psi_\mathcal{E}\left(\mathsf{D}_\lambda(x) \right) \right) = \psi^{-1}_\mathcal{E}( \mathsf{D}_\lambda(x) ) = \mathsf{D}_\lambda(x) 
\end{gather*}
So $\psi^{-1}_\mathcal{E} \circ \psi_\mathcal{E} = 1_{\mathsf{Sym}_R\left(\mathsf{im}(\mathsf{D}_\lambda) \right)}$. On the other hand, to show that $\psi_\mathcal{E} \circ \psi^{-1}_\mathcal{E} = 1_E$, we will first show that $\psi_\mathcal{E} \circ \psi^{-1}_\mathcal{E} \circ \mathsf{q} = \mathsf{q}$ and $\psi_\mathcal{E} \circ \psi^{-1}_\mathcal{E} \circ \lambda = \lambda$. So on generators $a \in R$ and $x \in E$, we compute: 
\begin{gather*}
\psi_\mathcal{E}( \psi^{-1}_\mathcal{E}( \mathsf{q}(a) )) =  \psi_\mathcal{E}( a) = \mathsf{q}(a)  \\ \\ 
 \psi_\mathcal{E}( \psi^{-1}_\mathcal{E}( \lambda(x) )) =  \psi_\mathcal{E}( \delta_\mathcal{E}(x) ) =  \psi_\mathcal{E}( \mathsf{z}(x) ) = \mathsf{q}(\mathsf{z}(x)) = \lambda\left( \mathsf{p}_E(x) \right) = \lambda(x) \\ \\
 \psi_\mathcal{E}( \psi^{-1}_\mathcal{E}( \lambda(\mathsf{d}(x) ) )) =  \psi_\mathcal{E}( \delta_\mathcal{E}(\mathsf{d}(x) ) ) =  \psi_\mathcal{E}( \mathsf{D}_\lambda(x) ) = \mathsf{D}_\lambda(x) = \lambda(\mathsf{d}(x) )
\end{gather*}
Therefore, by the universal property of the pushout, it follows that $\psi_\mathcal{E} \circ \psi^{-1}_\mathcal{E} = 1_E$. So $\psi_\mathcal{E}$ and $\psi^{-1}_\mathcal{E}$ are inverse ring isomorphisms. 

Lastly, we must show that $\psi_\mathcal{E}$ and $\psi^{-1}_\mathcal{E}$ are also differential bundle morphisms over $R$ in $(\mathsf{CRING}^{op}, \mathbb{T})$, that is, we must show the dual of the axioms in Definition \ref{def:differentialbundlemorph} hold. We will first show that $\psi_\mathcal{E}$ is a differential bundle morphism. To do so, first recall that $\lambda(\mathsf{q}(a)) = \mathsf{q}(a)$ and $\mathsf{D}_\lambda(\mathsf{q}(a)) =0$, and that the dual of the last pre-differential bundle axiom in (\ref{prediffbuneq}) states that $\lambda \circ \mathsf{T}(\lambda) = \lambda \circ \ell_E$. So we show that the desired equalities hold by computing the following on generators:  
\begin{enumerate}[{\em (i)}]
\item $ \psi_\mathcal{E} \circ  \mathsf{q}_{\mathsf{im}(\mathsf{D}_\lambda)} = \mathsf{q}$:
\begin{align*}
   \psi_\mathcal{E}(\mathsf{q}_{\mathsf{im}(\mathsf{D}_\lambda)}(a)) =  \psi_\mathcal{E}(a) = \mathsf{q}(a) 
\end{align*}
\item $\lambda \circ \mathsf{T}(\beta_\mathcal{E}) = \psi_\mathcal{E} \circ \lambda_{\mathsf{im}(\mathsf{D}_\lambda)}$, on $a$:
\[ 
\lambda\left( \mathsf{T}(\beta_\mathcal{E})(a) \right)= 
\lambda\left( \beta_\mathcal{E}(a) \right)
=
\lambda( \mathsf{q}(a) )
=
\mathsf{q}(a)
=
\psi_\mathcal{E}(a)
=
\psi_\mathcal{E}( \lambda_{\mathsf{im}(\mathsf{D}_\lambda)}(a))
\]
on $\mathsf{D}_\lambda(x)$:
\[
\lambda\left( \mathsf{T}(\beta_\mathcal{E})\left(\mathsf{D}_\lambda(x) \right) \right) 
=
\lambda\left( \beta_\mathcal{E}\left( \mathsf{D}_\lambda(x) \right) \right)
=
\lambda\left( \mathsf{D}_\lambda(x) \right)
=
\lambda \left( \lambda (\mathsf{d}(x) ) \right)
=
\lambda \left( \mathsf{T}(\lambda) (\mathsf{d}(x) ) \right)
\]
\[
=
\lambda \left( \ell_E (\mathsf{d}(x) ) \right)
=
\lambda(0)
=
0
=
\psi_\mathcal{E}(0)
=
\psi_\mathcal{E}\left( \lambda_{\mathsf{im}(\mathsf{D}_\lambda)}\left( \mathsf{D}_\lambda(x) \right) \right) 
\]
on $\mathsf{d}(a)$:
\[
\lambda\left( \mathsf{T}(\beta_\mathcal{E})\left(\mathsf{d}(a) \right) \right)
= 
\lambda\left( \mathsf{d}\left(  \beta_\mathcal{E}(a)\right) \right)
=
\]
\[
\lambda\left( \mathsf{d}\left( \mathsf{q}(a) \right) \right)
=
\mathsf{D}_\lambda( \mathsf{q}(a) )
= 
0
=
\psi_\mathcal{E}(0) \\
=
\psi_\mathcal{E}\left( \lambda_{\mathsf{im}(\mathsf{D}_\lambda)}\left( \mathsf{d}(a) \right) \right) 
\]
and finally on $\mathsf{d}\left( \mathsf{D}_\lambda(x)\right)$:
\[
\lambda\left( \mathsf{T}(\beta_\mathcal{E})\left(\mathsf{d}\left( \mathsf{D}_\lambda(x)\right)  \right) \right)
=
\lambda\left( \mathsf{d}\left(  \beta_\mathcal{E}\left(\mathsf{D}_\lambda(x) \right)\right) \right)
=
\lambda\left( \mathsf{d}\left( \mathsf{D}_\lambda(x) \right) \right)
=
\]
\[
\lambda\left( \mathsf{d}\left( \lambda (\mathsf{d}(x) ) \right) \right)
=
\lambda \left( \mathsf{T}(\lambda) (\mathsf{d}^\prime\mathsf{d}(x) ) \right)
=
\lambda \left(\ell_E (\mathsf{d}^\prime\mathsf{d}(x) ) \right)
=
\lambda \left(\mathsf{d}(x) \right)
=
\]
\[
\mathsf{D}_\lambda(x)
=
\beta_\mathcal{E}\left(\mathsf{D}_\lambda(x) \right)
=
\psi_\mathcal{E}\left( \lambda_{\mathsf{im}(\mathsf{D}_\lambda)}\left( \mathsf{d}\left( \mathsf{D}_\lambda(x)\right)  \right) \right)   
\]
\end{enumerate}
So it follows that $\psi_\mathcal{E}$ is a differential bundle morphism over $R$ in $(\mathsf{CRING}^{op}, \mathbb{T})$. By Lemma \ref{cor:diffbuniso} it then follows that $\psi^{-1}_\mathcal{E}$ is also a differential bundle morphism over $R$. So we conclude that $\psi_\mathcal{E}$ and $\psi^{-1}_\mathcal{E}$ are differential bundle isomorphisms over $R$ in $(\mathsf{CRING}^{op}, \mathbb{T})$. 
\end{proof}

Thus, the construction from a module to a differential bundle is the inverse of the construction from a differential bundle to a module. So we conclude that: 

\begin{proposition}\label{prop:mod-diff-aff} For a commutative ring $R$, there is a bijective correspondence between $R$-modules and differential bundles (with negatives) over $R$ in $(\mathsf{CRING}^{op}, \mathbb{T})$. 
\end{proposition}

In $\mathsf{CRING}$, recall that the initial object is $\mathbb{Z}$, which means that $\mathbb{Z}$ is the terminal object in $\mathsf{CRING}^{op}$. So differential objects in $(\mathsf{CRING}^{op}, \mathbb{T})$ correspond precisely to $\mathbb{Z}$-modules, which are precisely Abelian groups. 

\begin{corollary}\label{cor:diffobj-aff} There is a bijective correspondence between $\mathbb{Z}$-modules/Abelian groups and differential objects in $(\mathsf{CRING}^{op}, \mathbb{T})$.
\end{corollary}

We now extend Proposition \ref{prop:mod-diff-aff} to an equivalence of categories. For a commutative ring $R$, we define an equivalence of categories between $\mathsf{MOD}^{op}_R$ and $\mathsf{DBUN}_{\mathbb{T}}\left[ R\right]$ as follows: 
\begin{enumerate}[{\em (i)}]
\item Define the functor $\mathsf{M}_R:\mathsf{MOD}^{op}_R \to \mathsf{DBUN}_{\mathbb{T}}\left[ R\right]$ which sends an $R$-module $M$ to the differential bundle $\mathsf{M}_R(M)$, and sends an $R$-linear morphism $f: M \to M^\prime$ to the differential bundle morphism over $R$ $\mathsf{M}_R(f): \mathsf{M}_R(M^\prime) \to \mathsf{M}_R(M)$ defined to be the ring morphism $\mathsf{M}_R(f): \mathsf{Sym}_R(M) \to \mathsf{Sym}_R(M^\prime)$ defined on generators as follows: 
\begin{align*}
    \mathsf{M}_R(f)(a) = a && \mathsf{M}_R(f)(m) = f(m) 
\end{align*}
\item Define the functor $\mathsf{M}^\circ_R: \mathsf{DBUN}_{\mathbb{T}}\left[ R\right] \to \mathsf{MOD}^{op}_R$ which sends a differential bundle with negatives over $R$, ${\mathcal{E} =(\mathsf{q}: E \to R, \sigma, \mathsf{z}, \lambda, \iota)}$ to the $R$-module $\mathsf{M}^\circ_R(\mathcal{E}) =\mathsf{im}(\mathsf{D}_\lambda)$, and sends a differential bundle morphism $f: \mathcal{E} =(\mathsf{q}: E \to R, \sigma, \mathsf{z}, \lambda, \iota) \to \mathcal{E}^\prime =({\mathsf{q}^\prime: E^\prime \to R}, \sigma^\prime, \mathsf{z}^\prime, \lambda^\prime, \iota^\prime)$ over $R$ to the $R$-linear morphism $\mathsf{M}^\circ_R(f) :  \mathsf{im}(\mathsf{D}_{\lambda^\prime}) \to \mathsf{im}(\mathsf{D}_\lambda)$ defined as:
\[ \mathsf{M}^\circ_R(f)( \mathsf{D}_{\lambda^\prime}(x)) = \mathsf{D}_{\lambda}(f(x))  \]
\item Observe that $\mathsf{M}^\circ_R \circ \mathsf{M}_R= \mathsf{1}_{\mathsf{MOD}^{op}_R}$. 
\item Define the natural isomorphism $\psi: \mathsf{1}_{\mathsf{DBUN}_{\mathbb{T}}\left[ R\right]} \Rightarrow \mathsf{M}_R \circ \mathsf{M}^\circ_R$ with inverse $\psi^{-1}:  \mathsf{M}_R \circ \mathsf{M}^\circ_R \Rightarrow \mathsf{1}_{\mathsf{DBUN}_{\mathbb{T}}\left[ R\right]}$ as $\psi_{\mathcal{E}}$ and $\psi^{-1}_{\mathcal{E}}$ as defined in Lemma \ref{lem:psiiso}.
\end{enumerate}

\begin{theorem} \label{thm:mod-diff-ring-aff-1} For a commutative ring $R$, we have an equivalence of categories: 
\[\mathsf{MOD}^{op}_R \simeq \mathsf{DBUN}_{\mathbb{T}}\left[ R\right]\] 
\end{theorem}
\begin{proof} We must first explain why $\mathsf{M}_R$ and $\mathsf{M}^\circ_R$ are well-defined. Clearly, $\mathsf{M}_R$ is well-defined on objects and maps and preserves composition and identities. So $\mathsf{M}_R$ is indeed a functor. On the other hand, let ${f: \mathcal{E} \to \mathcal{E}^\prime}$ be a differential bundle morphism over $R$ in $(\mathsf{CRING}^{op}, \mathbb{T})$. This implies that $f: E^\prime \to E$ is a ring morphism and also that $f(\mathsf{q}^\prime(a)) = \mathsf{q}(a)$ for all $a \in R$. Since $\mathsf{M}_R(f)$ is clearly linear, we show that it also preserves the action:
\[
\mathsf{M}^\circ_R(f)\left( a \cdot \mathsf{D}_{\lambda^\prime}(x) \right)
=
\mathsf{M}^\circ_R(f)\left( \mathsf{D}_{\lambda^\prime}(\mathsf{q}(a)x) \right)
=
\mathsf{D}_{\lambda}\left(f(\mathsf{q}(a)x) \right)
\]
\[
=
\mathsf{D}_{\lambda}\left(f(\mathsf{q}^\prime(a))f(x) \right)
=
\mathsf{D}_{\lambda}\left(\mathsf{q}(a)f(x) \right)
=
a \cdot  \mathsf{D}_{\lambda}(f(x))
=
a \cdot  \mathsf{M}^\circ_R(f)( \mathsf{D}_{\lambda^\prime}(x))
\]
So we have that $\mathsf{M}_R(f)$ is an $R$-linear morphism, and so $\mathsf{M}^\circ_R$ is well-defined. Clearly, $\mathsf{M}^\circ_R$ also preserves composition and identities, so $\mathsf{M}^\circ_R$ is also a functor. Furthermore, we also have that $\mathsf{M}^\circ_R \circ \mathsf{M}_R= \mathsf{1}_{\mathsf{MOD}^{op}_R}$. Next, $\psi$ and $\psi^{-1}$ are well-defined component-wise and are inverses at each component. Therefore, it suffices to show that $\psi$ is natural and then it will follow that $\psi^{-1}$ is also natural.  If $f: \mathcal{E} \to \mathcal{E}^\prime$ is a differential bundle morphism over $R$ in $(\mathsf{CRING}^{op}, \mathbb{T})$, then $f \circ \lambda^\prime = \lambda \circ \mathsf{T}(f)$. In particular, this means that $f(\lambda^\prime(\mathsf{d}(x))) = \lambda(\mathsf{d}(f(x)))$. However, we can rewrite this as $f\left( \mathsf{D}_{\lambda^\prime}(x) \right) = \mathsf{D}_{\lambda}\left( f(x)  \right)$. Therefore, we compute on generators that: 
\begin{align*}
 \psi_{\mathcal{E}} \left( \mathsf{M}_R\left( \mathsf{M}^\circ_R(f) \right) (a) \right) &= \psi_{\mathcal{E}}(a) = \mathsf{q}(a) = f(\mathsf{q}^\prime(a)) = f(\psi_{\mathcal{E}^\prime}(a))
\end{align*}
and
\[
\psi_{\mathcal{E}} \left( \mathsf{M}_R\left( \mathsf{M}^\circ_R(f) \right) \left( \mathsf{D}_{\lambda^\prime}(x) \right) \right)
=
\psi_{\mathcal{E}}\left( \mathsf{M}^\circ_R(f) \left( \mathsf{D}_{\lambda^\prime}\left(x \right) \right)  \right)
\]
\[
=
\psi_{\mathcal{E}}\left( \mathsf{D}_{\lambda}\left( f(x)  \right)  \right)
= 
\mathsf{D}_{\lambda}\left( f(x)  \right)
= 
\left( \mathsf{D}_{\lambda^\prime}(x) \right)
=
f\left(\psi_{\mathcal{E}^\prime}\left( \mathsf{D}_{\lambda^\prime}(x) \right) \right) 
\]
    So $ \psi_{\mathcal{E}} \circ \mathsf{M}_R\left( \mathsf{M}^\circ_R(f) \right) = f \circ \psi_{\mathcal{E}^\prime}$ in $\mathsf{CRING}$. Therefore $\psi$ is a natural transformation, and it follows that so is $\psi^{-1}$. Thus, $\psi$ and $\psi^{-1}$ are natural isomorphisms. So we conclude that we have an equivalence of categories: $\mathsf{MOD}^{op}_R \simeq \mathsf{DBUN}_{\mathbb{T}}\left[ R\right]$. 
\end{proof}

It then follows that we have an equivalence between the category of differential objects and the opposite category of Abelian groups. So let $\mathsf{Ab}$ be the category whose objects are Abelian groups and whose morphisms are group morphisms.   

\begin{corollary} There is an equivalence of categories: \[ \mathsf{DBUN}\left[ (\mathsf{CRING}^{op}, \mathbb{T}) \right] \simeq (\mathsf{MOD}_\mathbb{Z})^{op} \simeq \mathsf{Ab}^{op}. \]
\end{corollary}

We now define an equivalence of categories between $\mathsf{MOD}^{op}$ and $\mathsf{DBUN}\left[ (\mathsf{CRING}^{op}, \mathbb{T}) \right]$ as follows: 
\begin{enumerate}[{\em (i)}]
\item Define the functor $\mathsf{M}:\mathsf{MOD}^{op} \to \mathsf{DBUN}_{\mathbb{T}}\left[ R\right]$ which sends an object $(R,M)$ to the differential bundle $\mathsf{M}(R,M) = \mathsf{M}_R(M)$, and sends a map $(g,f): (R,M) \to (R^\prime, M^\prime)$ in $\mathsf{MOD}$ to the differential bundle morphism $\mathsf{M}(f) = (\mathsf{M}_R(f), g):\mathsf{M}_{R^\prime}(M^\prime) \to  \mathsf{M}_R(M)$. 
\item Define the functor $\mathsf{M}^\circ: \mathsf{DBUN}\left[ (\mathsf{CRING}^{op}, \mathbb{T}) \right] \to \mathsf{MOD}^{op}$ which sends a differential bundle with negatives ${\mathcal{E} =(\mathsf{q}: E \to R, \sigma, \mathsf{z}, \lambda, \iota)}$ to the 
pair $\mathsf{M}^\circ(\mathcal{E}) = (R, \mathsf{im}(\mathsf{D}_\lambda))$, and sends a differential bundle morphism $(f,g): \mathcal{E} =(\mathsf{q}: E \to R, \sigma, \mathsf{z}, \lambda, \iota) \to \mathcal{E}^\prime =({\mathsf{q}^\prime: E^\prime \to R^\prime}, \sigma^\prime, \mathsf{z}^\prime, \lambda^\prime, \iota^\prime)$ to the pair $\mathsf{M}^\circ(f,g) = (g, \mathsf{M}^\circ_R(f))$. 
\item Observe that $\mathsf{M}^\circ \circ \mathsf{M}= \mathsf{1}_{\mathsf{MOD}^{op}}$. 
\item Define the natural isomorphism $\overline{\psi}: \mathsf{1}_{ \mathsf{DBUN}\left[ (\mathsf{CRING}^{op}, \mathbb{T}) \right]} \Rightarrow \mathsf{M} \circ \mathsf{M}^\circ$ as $\overline{\psi}_{\mathcal{E}} = (1, \psi_{\mathcal{E}})$, with inverse natural isomorphism $\overline{\psi}^{-1}: \mathsf{M} \circ \mathsf{M}^\circ \Rightarrow \mathsf{1}_{ \mathsf{DBUN}\left[ (\mathsf{CRING}^{op}, \mathbb{T}) \right]}$ as $\overline{\psi}^{-1}_{\mathcal{E}} = (1, \psi^{-1}_{\mathcal{E}})$. 
\end{enumerate}

\begin{theorem} \label{thm:mod-diff-ring-aff-2} We have an equivalence of categories: 
\[\mathsf{MOD}^{op} \simeq \mathsf{DBUN}\left[ (\mathsf{CRING}^{op}, \mathbb{T}) \right]\]
\end{theorem}
\begin{proof} The proof that $\mathsf{M}$ and $\mathsf{M}^\circ$ are well-defined functors is similar to the proof that $\mathsf{M}_R$ and $\mathsf{M}_R^\circ$ are well-defined in the proof of Theorem \ref{thm:mod-diff-ring-aff-1}. Furthermore, it also follows that $\mathsf{M}^\circ \circ \mathsf{M}= \mathsf{1}_{\mathsf{MOD}^{op}}$. On the other hand, since $\psi_\mathcal{E}$ and $\psi^{-1}_\mathcal{E}$ are differential bundle morphisms over the base commutative ring, it follows by definition that $\overline{\psi}_\mathcal{E} = (1, \psi_\mathcal{E})$ and $\overline{\psi}^{-1}_\mathcal{E} = (1, \psi^{-1}_\mathcal{E})$ are differential bundle morphisms, so $\overline{\psi}$ and $\overline{\psi}^{-1}$ are well-defined. Lastly, that $\overline{\psi}$, and $\overline{\psi}^{-1}$ are natural isomorphisms follows directly from the fact that $\psi$, and $\psi^{-1}$ are natural isomorphisms. So we conclude that we indeed have an equivalence of categories: $\mathsf{MOD}^{op} \simeq \mathsf{DBUN}\left[ (\mathsf{CRING}^{op}, \mathbb{T}) \right]$. 
\end{proof}

\begin{remark} The equivalence between modules and differential bundles is also true in more general settings. Indeed, both for the opposite category of commutative semirings and the opposite category of commutative algebras over a (semi)ring, differential bundles correspond precisely to modules via the above constructions (where the latter follows from Corollary \ref{cor:tan_algebras} and Proposition \ref{prop:slice_diff_bundles}). As explained before, in a setting where one does not have negatives, we would again also need to prove an extra pushout square. On the other hand, it is unclear if this result always generalizes to the coEilenberg-Moore category of a differential category. If the differential category has enough limits and colimits, then it is possible to generalize the constructions of Lemma \ref{lem:diff-mod-aff} and Lemma \ref{lem:mod-diff-aff}, and then we obtain a bijective correspondence between differential bundles and comodules of the coalgebras of the comonad of said differential category. However, in general, a differential category need not have all limits or colimits. In future work, it would be interesting to characterize differential bundles in arbitrary differential categories and understand what assumptions are needed so that they correspond to (co)modules. 
\end{remark}

\subsection{Differential bundles in schemes}

In this section, we show how the characterization of differential bundles in affine schemes can be extended to characterize differential bundles in the larger category of schemes.  Since schemes are the gluing of affine schemes, this follows relatively straightforwardly from the results of the previous sections, so here we merely sketch the proofs. Our first goal is to show that for any differential bundle ${\mathsf{q}: E \to A}$ in schemes, the projection $\mathsf{q}$ is an affine map. Let us first quickly recall the definition of affine morphisms and equivalent characterizations \cite[Section 29.11]{stacks-project}.

\begin{definition} \cite[Definition 29.11.1]{stacks-project}
A morphism of schemes $f: X \to Y$ is \textbf{affine} if for all affine opens $U$ of $Y$, the inverse image $f^{-1}(U)$, that is, the following pullback: 
    \[ \xymatrix{f^{-1}(U) \ar[r] \ar[d] & X \ar[d] \\ U \ar[r] & Y} \]
is itself affine.
\end{definition}

\begin{proposition}\label{prop:affine} \cite[Lemma 29.11.3]{stacks-project} For a scheme morphism $f: X \to Y$, the following are equivalent:
\begin{enumerate}[{\em (i)}]
    \item $f$ is affine; 
    \item $Y$ has a covering by affine opens $\{U_i\}_{i \in I}$ such that for all $i \in I$, $f^{-1}(U_i)$ is affine; 
    \item $X = \mathsf{Spec}(A)$ for some quasicoherent sheaf of algebras $A$ on the sheaf ${\cal O}_Y$.
\end{enumerate}
\end{proposition}

The following is a general result about affine morphisms which will be useful below:

\begin{lemma}\label{lem:affineretract}
Affine morphisms are closed under retract, that is, if we have scheme morphisms

\[ \xymatrix{X_1 \ar@/^-0.5pc/[rr]_{s} \ar[dr]_{f_1} & & X_2 \ar@/^-0.5pc/[ll]_{r} \ar[dl]^{f_2} \\ & Y &} \]

with $(s,r)$ a section/retraction pair in the category of schemes over $Y$ and $f_2$ is affine, then so is $f_1$. 
\end{lemma}
\begin{proof}
Let $U$ be an affine open subset of $Y$. Then we can define a section/retraction pair $(s_U, r_U)$ between $f_1^{-1}(U)$ and $f_2^{-1}(U)$ with both defined by pullback.  For example, here is the defining diagram for $s_U$:
\[ \xymatrix{f_1^{-1}(U) \ar[rr] \ar@{-->}[dr]^{s_U} \ar@/^-0.7pc/[ddr] & & X_1 \ar[dr]^{s} & \\
& f_2^{-1}(U) \ar[rr] \ar[d] & & X_2 \ar[d]^{f_2} \\
& U \ar[rr] && Y} \]
Thus $f_{1}^{-1}(U)$ is a retract of a representable element in the presheaf category $[\mathsf{CRING},\mathsf{SET}]$ (where $\mathsf{SET}$ is the category of sets and arbitrary functions between them). But so long as a category $\X$ has split idempotents, then representables in the functor category $[\X^{op},\mathsf{SET}]$ are closed under retract \cite[Lemma 6.5.6]{bourceux}. So $f_1^{-1}(U)$ is itself representable, and so by definition $f_1$ is affine, as required. 
\end{proof}

We may now prove that for a differential bundle in the category of schemes, the projection is affine. 

\begin{proposition}
In the category of schemes, if $\mathsf{q}: E \to A$ is a differential bundle, then $\mathsf{q}$ is affine.
\end{proposition}
\begin{proof}
By \cite[Corollary 3.1.4]{macadam2021vector}, $\mathsf{q}$ is a retract of a pullback of a tangent bundle projection $\mathsf{p}_A: \mathsf{T}(A) \to A$.  By definition, $\mathsf{T}(A)$ is $\mathsf{Spec}$ of a quasicoherent sheaf of algebras on ${\cal O}_A$, so by Lemma \ref{prop:affine}, $\mathsf{p}_A$ is affine.  But affines are closed under pullback \cite[Lemma 29.11.8]{stacks-project} and retracts (Lemma \ref{lem:affineretract}), so $\mathsf{q}$ is affine.  
\end{proof}

We may now prove that every differential bundle is a $\mathsf{Spec}$ of $\mathsf{Sym}$ of a quasicoherent sheaf of modules. 

\begin{proposition}
If $\mathsf{q}: E \to A$ is a differential bundle in the category of schemes, then $E$ is $\mathsf{Spec}$ of $\mathsf{Sym}$ of a quasicoherent sheaf of modules.
\end{proposition}
\begin{proof}
Cover $A$ by affines $U_i$, and since $\mathsf{q}$ is affine, each pullback $q^{-1}(U_i)$: 
\[ \xymatrix{\mathsf{q}^{-1}(U_i) \ar[r] \ar[d] & E \ar[d]^{\mathsf{q}} \\ U_i \ar[r] & A} \]
is also affine. Moreover, by \cite[Lemma 2.7]{cockett2018differential}, differential bundles are closed under pullback.  Thus each map $\mathsf{q}^{-1}(U_i) \to U_i$ is a differential bundle in the category of affine schemes, and hence by Proposition \ref{prop:mod-diff-ring}, each $\mathsf{q}^{-1}(U_i)$ is $\mathsf{Sym}$ of a module on $U_i$.  Thus as $E$ is the gluing of these, $E$ is itself $\mathsf{Spec}$ of $\mathsf{Sym}$ of a quasicoherent sheaf of modules \cite[pg. 379]{vakil}.   
\end{proof}

Conversely, we now prove that every quasicoherent sheaf of modules induces a differential bundle. 

\begin{proposition}
If $M$ is a quasicoherent sheaf of modules on a scheme $A$, then $\mathsf{Spec}$ of $\mathsf{Sym}$ of $M$ is a differential bundle over $A$.  
\end{proposition}
\begin{proof}
Suppose that $A$ is covered by affines $U_i$.  Then by \cite[pg.379]{vakil}, if $M$ is a quasicoherent sheaf of modules on $A$, then $M$ is the gluing of modules $M_i$ over the $U_i$, and $\mathsf{Spec}$ of $\mathsf{Sym}$ of $M$ is the gluing of $\mathsf{Spec}$ of $\mathsf{Sym}$ of the $M_i$'s.  Thus it suffices to show that such a gluing is a differential bundle over $A$. But by Lemma \ref{lem:mod-diff-aff}, $\mathsf{Spec}$ of $\mathsf{Sym}$ of each $M_i$ is a differential bundle over $A_i$.  It then follows that since the tangent functor on schemes preserves gluings (for an abstract proof of this, see \cite[Prop. 6.15.ii]{cockett2014differential}), the lifts of each such differential bundle $\lambda_i$ glue together to give a lift $\lambda$ for $\mathsf{Spec}$ of $\mathsf{Sym}$ of $M$, and it follows through straightforward calculations that this satisfies the required conditions to be a differential bundle.    
\end{proof}

The results on morphisms follow similarly, and therefore we obtain: 

\begin{theorem}\label{thm:scheme_equivalence_1}
For a scheme $A$, there is an equivalence of categories between differential bundles over $A$ in the tangent category $\mathsf{SCH}$ and the opposite category of quasicoherent sheaves of modules over $A$.  
\end{theorem}

\begin{remark} 
By Corollary \ref{cor:slice_schemes} 
and Proposition \ref{prop:slice_diff_bundles}, for any scheme $A$, there is a similar result for the tangent category of schemes over $A$.    
\end{remark}

\section{Future Work}\label{sec:future_work}

Understanding differential bundles in the tangent categories of commutative rings, affine schemes, and schemes is just the beginning of applying tangent category theory to commutative algebra and algebraic geometry. There are many possible future avenues for exploration based on this work, such as: 
\begin{itemize}
    \item The most immediate next step is to understand how connections in tangent categories apply to these examples.  They seem closely related to connections on modules \cite[Section 8.2]{connections_cqft}, but more work needs to be done to understand the precise relationship between the two notions. 
    \item Tangent categories have a notion of differential forms and de Rham cohomology \cite{deRham}.  Initial investigation with this idea suggests 
    that for affine schemes over $R$, when the coefficient object is taken to be the polynomial ring $R[x]$, then this tangent category cohomology recreates algebraic de Rham cohomology.  However, more work is required to prove this.  Moreover, \cite{deRham} also develops a second notion of cohomology: sector form cohomology.  It is not clear what this should give in the algebraic geometry setting.  
    \item In \cite{joyce}, Dominic Joyce develops algebraic geometry in the setting of $C^{\infty}$-rings.  It seems likely that the categories involved are tangent categories, and so hopefully tangent category theory applied to this particular example will recreate the corresponding notions Joyce has developed. 
    \item A key idea in algebraic geometry is that of a \emph{smooth} morphism or object.  It would be interesting to see if such a notion could be generalized to arbitrary tangent categories (in such a way, that, for example, all objects in the tangent category of smooth manifolds are smooth).
    \item Finally, the Serre-Swan theorem provides a very different way to compare vector bundles to modules.  It would be interesting to see if one could give a proof of the Serre-Swan theorem using the results of this paper. 
\end{itemize}

Thus, we hope that this paper will serve as inspiration for future work in this area.  

\bibliographystyle{plain}      % mathematics and physical sciences
\bibliography{tac_refs}   % name your BibTeX data base

\begin{thebibliography}{10}

\bibitem{tangentInfinity}
K.~Bauer, M.~Burke, and M.~Ching.
\newblock Tangent infinity-categories and {G}oodwillie calculus.
\newblock {\em arXiv.org:2101.07819}, 2022.

\bibitem{blute2009cartesian}
R.~F. Blute, J.~R.~B. Cockett, and R.~A.~G. Seely.
\newblock Cartesian differential categories.
\newblock {\em Theory and Applications of Categories}, 22(23):622--672, 2009.

\bibitem{blute2015derivations}
R.~F. Blute, R.~B.~B. Lucyshyn-Wright, and K.~O'Neill.
\newblock Derivations in codifferential categories.
\newblock {\em Cahiers de Topologie et G{\'e}om{\'e}trie Diff{\'e}rentielle
  Cat{\'e}goriques}, 57:243--280, 2016.

\bibitem{bourceux}
F.~Bourceux.
\newblock {\em Handbook of Categorical Algebra, Volume I}.
\newblock Cambridge University Press, 1994.

\bibitem{cockett2014differential}
J.~R.~B. Cockett and G.~S.~H. Cruttwell.
\newblock Differential structure, tangent structure, and {SDG}.
\newblock {\em Applied Categorical Structures}, 22(2):331--417, 2014.

\bibitem{jacobi}
J.~R.~B. Cockett and G.~S.~H. Cruttwell.
\newblock The {J}acobi identity for tangent categories.
\newblock {\em Cahiers de Topologie et G{\'e}om{\'e}trie Diff{\'e}rentielle
  Cat{\'e}goriques}, LVI(4):301--316, 2015.

\bibitem{connections}
J.~R.~B. Cockett and G.~S.~H. Cruttwell.
\newblock Connections in tangent categories.
\newblock {\em Theory and applications of categories}, 32(26):835--888, 2017.

\bibitem{cockett2018differential}
J.~R.~B. Cockett and G.~S.~H. Cruttwell.
\newblock Differential bundles and fibrations for tangent categories.
\newblock {\em Cahiers de Topologie et G{\'e}om{\'e}trie Diff{\'e}rentielle
  Cat{\'e}goriques}, LIX:10--92, 2018.

\bibitem{cockett2021differential}
J.~R.~B. Cockett, G.S.H. Cruttwell, and J.-S.P. Lemay.
\newblock Differential equations in a tangent category {I}: Complete vector
  fields, flows, and exponentials.
\newblock {\em Applied Categorical Structures}, 29(5):773--825, 2021.

\bibitem{cockett_et_al:LIPIcs:2020:11660}
J.~R.~B. Cockett, J.-S.~P. Lemay, and R.~B.~B. Lucyshyn-Wright.
\newblock {Tangent Categories from the Coalgebras of Differential Categories}.
\newblock In M.~Fern{\'a}ndez and A.~Muscholl, editors, {\em 28th EACSL Annual
  Conference on Computer Science Logic (CSL 2020)}, volume 152 of {\em Leibniz
  International Proceedings in Informatics (LIPIcs)}, pages 17:1--17:17,
  Dagstuhl, Germany, 2020. Schloss Dagstuhl--Leibniz-Zentrum fuer Informatik.

\bibitem{deRham}
G.~S.~H. Cruttwell and R.B.B. Lucyshyn-Wright.
\newblock A simplicial foundation for differential and sector forms in tangent
  categories.
\newblock {\em Journal of Homotopy and Related Structures volume}, 13:867--925,
  2018.

\bibitem{EHRHARD20031}
T.~Ehrhard and L.~Regnier.
\newblock The differential lambda-calculus.
\newblock {\em Theoretical Computer Science}, 309(1):1--41, 2003.

\bibitem{GARNER2018668}
R.~Garner.
\newblock An embedding theorem for tangent categories.
\newblock {\em Advances in Mathematics}, 323:668--687, 2018.

\bibitem{grothendieck1966elements}
A.~Grothendieck.
\newblock {\'E}l{\'e}ments de g{\'e}om{\'e}trie alg{\'e}brique {IV}.
\newblock {\em Publications Math{\'e}matiques de l'IH{\'E}S}, 28:5--255, 1966.

\bibitem{joyce}
D.~Joyce.
\newblock An introduction to {C}-infinity schemes and {C}-infinity algebraic
  geometry.
\newblock {\em Surveys in Differential Geometry}, 79:299--325, 2012.

\bibitem{jubin2014tangent}
B.~Jubin.
\newblock The tangent functor monad and foliations.
\newblock {\em arXiv preprint arXiv:1401.0940}, 2014.

\bibitem{sdg99}
A.~Kock.
\newblock {\em Synthetic differential geometry}.
\newblock Cambridge University Press, 2006.

\bibitem{convenient}
A.~Kriegl and P.~Michor.
\newblock {\em The convenient setting of global analysis}.
\newblock American Mathematical Society, 1997.

\bibitem{lawvereCInfinity}
W.~Lawvere.
\newblock {\em Categorical Dynamics}, pages 1--28.
\newblock Aarhus University Press, 1979.

\bibitem{macadam2021vector}
B.~MacAdam.
\newblock Vector bundles and differential bundles in the category of smooth
  manifolds.
\newblock {\em Applied categorical structures}, 29(2):285--310, 2021.

\bibitem{connections_cqft}
L.~Mangiarotti and G.~Sardanashvily.
\newblock {\em Connections In Classical And Quantum Field Theory}.
\newblock World Scientific Publishing Company, 2000.

\bibitem{rosicky1984abstract}
J.~Rosick{\`y}.
\newblock Abstract tangent functors.
\newblock {\em Diagrammes}, 12:JR1--JR11, 1984.

\bibitem{stacks-project}
The {Stacks Project Authors}.
\newblock \textit{Stacks Project}.
\newblock \url{https://stacks.math.columbia.edu}, 2018.

\bibitem{vakil}
R.~Vakil.
\newblock {\em The rising sea: Foundations of algebraic geometry}.
\newblock Online textbook, 2017.

\end{thebibliography}

\end{document}